\documentclass[11pt]{amsart}
\calclayout
\usepackage[utf8]{inputenc}
\usepackage[T1]{fontenc} 
\usepackage{anyfontsize}
\usepackage{mathtools}
\usepackage{amsmath}
\usepackage{amssymb}
\usepackage{amsthm}
\usepackage{amsfonts}
\usepackage{stmaryrd}
\usepackage{upgreek}
\usepackage{enumitem}
\usepackage[dvipsnames]{xcolor}
\usepackage{tikzit}
\usetikzlibrary{decorations.markings}
\usetikzlibrary{arrows.meta, arrows}
\usetikzlibrary{decorations.pathreplacing,angles,quotes}
\usetikzlibrary{patterns}

\tikzstyle{dot}=[fill=black, draw=black, shape=circle, scale=0.5]
\tikzstyle{white-dot}=[fill=white, draw=black, shape=circle, scale=0.5]
\tikzstyle{blue-dot}=[fill=blue, draw=blue, shape=circle, scale=0.5]
\tikzstyle{red-dot}=[fill=red, draw=red, shape=circle, scale=0.5]
\tikzstyle{light-green-rounded-rectangle}=[fill={rgb,255: red,189; green,255; blue,197}, draw=black, shape=rectangle, rounded corners=3pt]
\tikzstyle{light-pink-rounded-rectangle}=[fill={rgb,255: red,255; green,225; blue,255}, draw=black, shape=rectangle, rounded corners=3pt]
\tikzstyle{light-yellow-rounded-rectangle}=[fill={rgb,255: red,255; green,254; blue,215}, draw=black, shape=rectangle, rounded corners=3pt]
\tikzstyle{thick-blue-implies}=[{=>}, draw=blue, thick]

\tikzstyle{dashed-black}=[-, dashed]
\tikzstyle{dotted-gray}=[-, dotted, draw={rgb,255: red,191; green,191; blue,191}]
\tikzstyle{red}=[-, draw=red]
\tikzstyle{thick}=[-, line width=1.5pt]
\tikzstyle{to}=[->]
\tikzstyle{red-to}=[->, draw=red]
\tikzstyle{blue-to}=[draw=blue, ->]
\tikzstyle{dash-to}=[->, dashed]
\tikzstyle{squig-to}=[->, decorate, decoration={{zigzag, segment length=4, amplitude=.9, post=lineto,postlength=2pt}}, line join=round]
\tikzstyle{red-dash-to}=[->, draw=red, dashed]
\tikzstyle{thick-blue-to}=[->, double distance=1pt, line width=1pt, draw=blue]
\tikzstyle{mapsto}=[{|->}]
\tikzstyle{directed}=[-, decoration={markings,mark=at position 0.7 with {\arrow{To[scale=1.25]}}}, postaction=decorate]
\tikzstyle{directed-dash}=[-, dashed, decoration={markings,mark=at position 0.7 with {\arrow{To[scale=1.25]}}}, postaction=decorate]
\tikzstyle{directed-dotted}=[-, dotted, decoration={markings,mark=at position 0.7 with {\arrow{To[scale=1.25]}}}, postaction=decorate]
\tikzstyle{directed-red}=[-, draw=red, decoration={markings,mark=at position 0.7 with {\arrow{To[scale=1.25]}}}, postaction=decorate]
\tikzstyle{directed-red-dash}=[-, draw=red, dashed, decoration={markings,mark=at position 0.7 with {\arrow{To[scale=1.25]}}}, postaction=decorate]
\tikzstyle{directed-blue}=[-, draw=blue, decoration={markings,mark=at position 0.7 with {\arrow{To[scale=1.25]}}}, postaction=decorate]
\tikzstyle{directed-blue-dash}=[-, draw=blue, dashed, decoration={markings,mark=at position 0.7 with {\arrow{To[scale=1.25]}}}, postaction=decorate]
\tikzstyle{light-gray-05-fill}=[-, fill={rgb,255: red,220; green,220; blue,220}, draw=none, fill opacity=0.5]
\tikzstyle{dark-gray-05-fill}=[-, fill={rgb,255: red,128; green,128; blue,128}, fill opacity=0.50, draw=none]
\tikzstyle{solid-green-fill}=[-, fill={rgb,255: red,66; green,192; blue,139}, draw=none]
\tikzstyle{light-yellow-fill-black-border}=[-, draw=black, fill={rgb,255: red,255; green,254; blue,215}]
\tikzstyle{light-yellow-05-fill}=[-, fill={rgb,255: red,255; green,254; blue,215}, draw=none, fill opacity=0.5]
\tikzstyle{light-magenta-05-fill}=[-, fill={rgb,255: red,255; green,225; blue,255}, draw=none, fill opacity=0.5]
\tikzstyle{light-cyan-05-fill}=[-, fill={rgb,255: red,196; green,255; blue,249}, draw=none, fill opacity=0.5]
\tikzstyle{dotted-fill}=[-, pattern=dots, pattern color=gray, draw=none]
\usepackage{quiver}
\usepackage{hyperref}
\usepackage{adjustbox}
\usepackage{graphicx}
\usepackage{textcomp}
\usepackage[normalem]{ulem}
\usepackage[all]{xy}
\usepackage{cancel}
\usepackage{scalerel}
\usepackage[backend=biber, 
style=alphabetic,
minalphanames=4,
maxalphanames=5, 
maxbibnames=5,doi=false,isbn=false,url=false]{biblatex}

\DeclareFontFamily{U}{mathx}{}
\DeclareFontShape{U}{mathx}{m}{n}{<-> mathx10}{}
\DeclareSymbolFont{mathx}{U}{mathx}{m}{n}
\DeclareMathAccent{\widehat}{0}{mathx}{"70}
\DeclareMathAccent{\widecheck}{0}{mathx}{"71}
\DeclareMathAccent{\widebar}{0}{mathx}{"73}

\theoremstyle{plain}
\newtheorem{theorem}{Theorem}[section]
\newtheorem{lemma}[theorem]{Lemma}
\newtheorem{proposition}[theorem]{Proposition}
\newtheorem{corollary}[theorem]{Corollary}

\theoremstyle{definition}

\newtheorem{definition}[theorem]{Definition} 
\newtheorem{example}[theorem]{Example}
\newtheorem{remark}[theorem]{Remark}

\newcommand{\Cat}{\mathsf{C}}
\newcommand{\Set}{\mathsf{Set}}
\newcommand{\Top}{\mathsf{Top}}
\newcommand{\Mfd}{\mathsf{Mfd}}

\newcommand{\Gpd}{\mathsf{Gpd}}

\newcommand{\PSh}{\mathsf{PSh}} 
\DeclareMathOperator{\Maps}{Maps}

\newcommand{\stkout}[1]{\ifmmode\text{\sout{\ensuremath{#1}}}\else\sout{#1}\fi} 

\renewcommand{\epsilon}{\varepsilon}

\newcommand{\into}{\hookrightarrow}

\makeatletter
\newcommand{\newrightleftarrows}[2]{%
  \mathrel{\mathop{%
    \vcenter{\offinterlineskip\m@th
      \ialign{\hfil##\hfil\cr
        \hphantom{$\scriptstyle\mspace{8mu}{#1}\mspace{8mu}$}\cr
        \rightarrowfill\cr
        \vrule height0pt width 2em\cr
        \leftarrowfill\cr
        \hphantom{$\scriptstyle\mspace{8mu}{#2}\mspace{8mu}$}\cr
        \noalign{\kern-0.3ex}
      }%
    }%
  }\limits^{#1}_{#2}}%
}
\makeatother

\DeclareMathOperator{\Kan}{Kan}
\DeclareMathOperator{\Hom}{Hom}
\DeclareMathOperator{\Img}{Im}

\DeclareMathOperator{\id}{id}

\newcommand{\huaW}{\mathcal{W}}

\newcommand{\huaH}{\mathcal{H}}

\newcommand{\Liengpd}{\mathsf{Lie_{n}Gpd}}
\newcommand{\Ana}{\mathsf{Ana}}
\newcommand{\Wbar}{\widebar{W}}
\newcommand{\X}{\mathcal{X}}
\newcommand{\Acyc}{\mathrm{Acyc}}
\newcommand{\hormor}{{\dashrightarrow}_h}
\newcommand{\vermor}{{\dashrightarrow}_v}
\newcommand{\hortilde}[1]{{\widetilde{#1}^h}}
\newcommand{\vertilde}[1]{{\widetilde{#1}^v}}
\newcommand{\vJoin}{\mathbin{\rotatebox[origin=c]{90}{$\Join$}}} 
\renewcommand{\Maps}{\mathrm{Maps}}

\title{Morphisms of double Lie groupoids: a simplicial approach}

\author{Daniel Álvarez}
\address{Instituto de Matem\'{a}tica Pura e Aplicada, Estrada Dona Castorina 110, Rio de Janeiro 22460-320, Brazil}
\email{uerbum@impa.br}

\author{Kalin Krishna}
\address{Mathematics Institute\\Georg-August-University of G\"ottingen\\Bunsenstra{ss}e 3-5\\G\"ottingen 37073\\Germany}
\email{kkrishn@mathematik.uni-goettingen.de}

\author{Stefano Ronchi}
\address{Mathematics Institute\\Georg-August-University of G\"ottingen\\Bunsenstra{ss}e 3-5\\G\"ottingen 37073\\Germany}
\email{stefano.ronchi@mathematik.uni-goettingen.de}

\date{\today}

\begin{document}

\begin{abstract}
Motivated by recent developments in generalized K\"ahler geometry, we study morphisms of double Lie groupoids from a simplicial viewpoint. We show that the codiagonal functor $\Wbar$ extends from objects to horizontal and vertical principal bibundles and to square-shaped $(1,1)$-morphisms. The latter determine 2-dimensional morphisms of square and globular shape between anafunctors of Lie 2-groupoids. We illustrate these constructions through two applications:

First, we revisit the equivalence between two models of nonabelian gerbes: groupoid bundle gerbes and principal 2-bundles. We show that a principal 2-bundle and its associated groupoid bundle gerbe determine equivalent anafunctors from a manifold to the corresponding structure Lie 2-group.

Second, we refine the integration of Manin triples in transitive Courant algebroids. We show that integrations associated with different choices of suitably transverse Manin triples in the same Courant algebroid are related by symplectic Morita equivalences. Consequently, the resulting integration of the background Courant algebroid is well defined up to symplectic Morita equivalence.
\end{abstract}

\maketitle
\tableofcontents
\section{Introduction}
In differential geometry, singular spaces such as the leaf spaces of foliations or orbit spaces of Lie group actions may be described using smooth objects called Lie groupoids. When a Lie group action preserves a foliation, it induces an action on the leaf space; taking the quotient of the leaf space by this action combines these two sources of singularity. Double Lie groupoids serve to describe such iterated or double quotients: they consist of a space of objects, horizontal and vertical arrows, and squares whose boundary edges are these arrows, with compatible horizontal and vertical composition, see Fig.~\ref{fig:double_groupoid}. 

\begin{figure}[!h]
    \centering
    \includegraphics[width=0.25\linewidth]{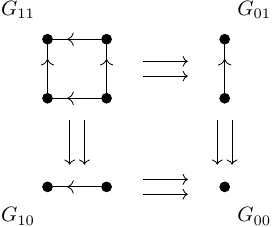}
    \caption{Pictorial representation of the elements of a double groupoid $G$.}
    \label{fig:double_groupoid}
\end{figure}

Besides their role in modeling such spaces, double (Lie) groupoids play an important role in algebraic topology \cite{BrownHigginsSivera2011}, higher gauge theory \cite{MartinsPicken2011,BaezSchreiber2007}, and Lie theory \cite{Weinstein1988,LuWeinstein1989,macdou}; more recently, they have become essential for describing generalized K\"ahler classes and their corresponding metrics \cite{AGJ2024}. The square pattern for double Lie groupoids reappears at the level of morphisms. In the double bicategory constructed in \cite{AGJ2024}, horizontal and vertical principal bibundles are called $(1,0)$- and $(0,1)$-morphisms, respectively, while the square-shaped cells are called $(1,1)$-morphisms. 

The constructions of \cite{AGJ2024} raise two natural questions: how the symplectic $(1,1)$-morphisms arising in generalized K\"ahler geometry relate to the more general problem of integrating Manin triples, and how the horizontal, vertical, and square-shaped morphisms introduced there fit into the existing theory of higher principal bundles and nonabelian gerbes. In this paper we answer both questions. For the first, we show that, under suitable hypotheses, different choices of Manin triples within the same transitive Courant algebroid determine symplectic double Lie groupoids related by a symplectic $(1,1)$-morphism. This extends the comparison underlying the generalized K\"ahler construction and shows that the resulting integrations of the Courant algebroid are symplectically Morita equivalent. For the second, we show that a principal 2-bundle and its associated groupoid bundle gerbe are related by a $(1,1)$-morphism, giving an explicit geometric comparison between these two models of nonabelian gerbes. The common mechanism behind these results is an extension of the codiagonal functor $\Wbar$, introduced in \cite{ArtinMazur1966}, to generalized morphisms of double Lie groupoids.

Under a suitable smooth filling condition, the codiagonal functor takes a double Lie groupoid to a simplicial manifold known as a Lie 2-groupoid \cite{MehtaTang2011}. In this work, we limit our attention to those double Lie groupoids which satisfy such a condition and we call them full. With a complete description of higher morphisms for double groupoids now available, we can continue the work initiated in \cite{MehtaTang2011}. We show that $\Wbar$ extends to a morphism from the double bicategory of full double Lie groupoids and principal bibundles to a suitable double bicategory of Lie 2-groupoids and anafunctors. The correspondence between horizontal and vertical morphisms of double Lie groupoids and anafunctors of Lie 2-groupoids determined by $\Wbar$ reveals another instance of the subtle interplay between the vertical and horizontal structures of a bisimplicial manifold. Just as $\Wbar$ maps a double Lie groupoid to a Lie 2-groupoid only when the double Lie groupoid is full, a morphism of double Lie groupoids that is levelwise a hypercover in the vertical or horizontal groupoid direction remains a hypercover under $\Wbar$ only if an additional filling condition is satisfied. We refer to this condition as fullness for vertical or horizontal hypercovers. In the case considered here, this filling condition is automatically satisfied for the morphisms of double Lie groupoids determined by principal actions of full double Lie groupoids on multiplicative bibundles. In fact, fullness also enters the picture when showing that a $(1,1)$-morphism gives rise to a double anafunctor, by applying the biaction groupoid construction.

Lie 2-groups provide natural examples of double Lie groupoids, serving as the structure groups for categorified principal bundles --- also known as nonabelian gerbes --- in higher gauge theory \cite{BaezSchreiber2007,MartinsPicken2011,GinotStienon2015}. While it is known that various definitions of nonabelian gerbes are related by higher categorical equivalences \cite{NikolausWaldorf2013}, making these relationships more explicit is desirable in order to obtain a clearer picture of the theory. Here, we show that a principal 2-bundle and its associated groupoid bundle gerbe are explicitly related by a $(1,1)$-morphism. Consequently, under the $\Wbar$ functor, they yield equivalent anafunctors (or generalized morphisms) into the corresponding structure Lie 2-group. This explicit correspondence paves the way for a unified treatment of connective structures on these equivalent objects.

Symplectic double Lie groupoids for Manin triples inside transitive Courant algebroids were constructed in \cite{Alvarez2023} under natural integrability assumptions. Applying $\Wbar$ to these symplectic double Lie groupoids yields 2-shifted symplectic Lie 2-groupoids that integrate the background Courant algebroid \cite{Alvarez2023,CuecaZhu2023,MehtaTang2011}. We show that different choices of suitably transverse Manin triples within the same Courant algebroid canonically generate a symplectic $(1,1)$-morphism of symplectic double Lie groupoids. Consequently, applying $\Wbar$ yields symplectic Morita equivalences between the corresponding integrations of the background Courant algebroid. In hindsight, this explains the natural appearance of $(1,1)$-morphisms in generalized K\"ahler geometry, where the relevant Courant algebroids are exact, and clarifies how deforming a Manin triple by a pair of suitable Maurer-Cartan elements reflects at the global level in the integration of the corresponding Courant algebroid by a 2-shifted symplectic 2-groupoid.

{\bf Interpretation and outlook.} As a comment on our results in the context of localization, let us mention that the
construction in Theorem~\ref{thm:11MorphToSquareAna-Wbar} retains the square shape of $(1,1)$-morphisms; on the other hand, we show in Theorem~\ref{thm:1-1-morph-to-2-hom-diamond} that the
same data determines a diagram of the form:
\[
\begin{tikzcd}[cramped,sep=small]
K_{\bullet,\bullet}
  & D_{\bullet}
  & H_{\bullet,\bullet}
  &&& \Wbar\vertilde{C}\times_{\Wbar H}\Wbar\hortilde{D}
  & {}
  \\
B_{\bullet}
  & \Omega
  & C_{\bullet}
  && \Wbar G
  & \Wbar\widetilde{\widetilde{\Omega}}
  & \Wbar K
  \\
J_{\bullet,\bullet}
  & A_{\bullet}
  & G_{\bullet,\bullet}
  &&& \Wbar\hortilde{A}\times_{\Wbar J}\Wbar\vertilde{B}
  & {}
  \\
\arrow[from=1-2, to=1-1]
\arrow[two heads, from=1-2, to=1-3]
\arrow["\sim"', two heads, from=1-6, to=2-5]
\arrow[from=1-6, to=2-7]
\arrow[from=2-1, to=1-1]
\arrow[two heads, from=2-1, to=3-1]
\arrow[from=2-2, to=1-2]
\arrow[from=2-2, to=2-1]
\arrow[two heads, from=2-2, to=2-3]
\arrow[two heads, from=2-2, to=3-2]
\arrow[from=2-3, to=1-3]
\arrow[dashed, from=2-3, to=2-5]
\arrow[two heads, from=2-3, to=3-3]
\arrow["\sim"', from=2-6, to=1-6]
\arrow["\sim"', two heads, from=2-6, to=2-5]
\arrow["\sim", from=2-6, to=3-6]
\arrow[from=3-2, to=3-1]
\arrow[two heads, from=3-2, to=3-3]
\arrow["\sim", two heads, from=3-6, to=2-5]
\arrow[from=3-6, to=2-7]
\end{tikzcd}
\]

This globular 2-morphism can be interpreted in different ways, a comprehensive comparison of which we leave to future work:
\begin{itemize}
    \item  A representative of a 2-morphism in the bicategory of fractions with respect to weak equivalences in the sense of Pronk \cite{Pronk1996, Tommasini2016, PronkScull2022}. This is a bicategorical localization of the bicategory of Lie 2-groupoids with mapping spaces the fundamental groupoids $\Pi_1\Maps(X,Y)$. 
    \item A span of morphisms of anafunctors in the  category $\mathsf{Span_{HC^{-1}}}(X,Y)$ with objects anafunctors between X and Y and morphisms commutative diagrams of the form: 
    \begin{equation}\label{diag:stricter-morphism-of-spans}
    \begin{tikzcd}[ampersand replacement=\&,cramped,sep=tiny]
	\&\& Z \&\& \\
	X \&\&\&\& Y \\
	\&\& Z'
	\arrow["\sim"', two heads, from=1-3, to=2-1]
	\arrow[from=1-3, to=2-5]
	\arrow["\sim", from=1-3, to=3-3]
	\arrow["\sim", two heads, from=3-3, to=2-1]
	\arrow[from=3-3, to=2-5]
    \end{tikzcd}
    \end{equation}
    The nerve of this category of spans is shown to compute (up to weak homotopy equivalence) the mapping spaces of the Dwyer-Kan localization of the simplicial category of Lie 2-groupoids in \cite[\S 2.2, \S 7.2.1]{RogersZhu2020}.
\end{itemize}

Our choice of examples is guided by the desire to provide a global definition of the Courant sigma model \cite{Roytenberg2007}. In standard gauge theory, formulating global Chern--Simons theory requires moving beyond local connection 1-forms to study a global moduli space of principal bundles for a chosen Lie group \cite{Freed1995}. Similarly, promoting the Courant sigma model to a global topological field theory requires two components: an integration of the background Courant algebroid serving as the global target space and a definition of the categorified gauge fields mapping into it. Our two applications are closely related to these two components: the first concerns the configuration space of fields, described by principal 2-bundles and groupoid bundle gerbes, for a Lie 2-group target, while the second controls the global target geometry through integrations of a Courant algebroid.

\paragraph{\bf Structure of the paper.}
In \S\ref{sec:simp stuff}, we review the standard theory of higher Lie groupoids, Morita equivalences and generalized morphisms in the context of localization. We specifically review localizations by bibundles and by anafunctors and how to pass between these in the case of Lie groupoids. Finally we construct a double bicategory whose horizontal and vertical morphisms are anafunctors and whose square-shaped cells are double anafunctors. In \S\ref{sec:double stuff}, we review the double bicategory of double Lie groupoids and principal bibundles based on the definitions introduced in \cite{AGJ2024}, and show how to convert this into a double bicategory of anafunctors for (full) double Lie groupoids.
In \S\ref{sec:double to simplicial}, we establish the compatibility of $\Wbar$ with the different higher categories of anafunctors. The main results are Theorems~\ref{thm:11MorphToSquareAna-Wbar} and~\ref{thm:1-1-morph-to-2-hom-diamond}, which describe the two constructions obtained by applying $\Wbar$ to $(1,1)$-morphisms of double Lie groupoids.

These results are applied in \S\ref{sec:nonab gerbes}, where we show that a principal 2-bundle for a Lie 2-group and its associated groupoid bundle gerbe as in \cite{NikolausWaldorf2013} form two sides of a $(1,1)$-morphism from a manifold to the structure Lie 2-group. Applying $\Wbar$ to this $(1,1)$-morphism produces a pair of equivalent anafunctors.

Finally, in \S\ref{sec:int manin}, we show that four suitably transverse integrable Dirac structures in a transitive Courant algebroid determine a symplectic $(1,1)$-morphism relating the corresponding symplectic double Lie groupoids, following \cite{AGJ2024,Alvarez2023}. Consequently, $\Wbar$ sends this diagram to a diagram relating the 2-shifted symplectic Lie 2-groupoids that integrate the background Courant algebroid.

{\bf Acknowledgments.} We thank C. Angulo, M. del Hoyo, E. Getzler, M. Gualtieri, Y. Jiang, E. Meinrenken, C. Zhu and all the other members of the Higher Structures Seminar for many insightful suggestions and inspiring discussions that made this work possible. 
We also thank the developers of \href{https://tikzit.github.io}{TikZiT} and \href{https://q.uiver.app}{Quiver}, which made it possible for us to have fun complementing this article with plenty of diagrams without the help of generative AI. K. K. gratefully acknowledges that this work was made possible with the support of a scholarship from the German Academic Exchange Service (DAAD).
S. R. is supported by Deutsche Forschungsgemeinschaft (DFG) project number 543037407.

\section{Higher Lie groupoids and 2-categories of anafunctors}\label{sec:simp stuff}

In this article we focus on the category of (real) smooth manifolds $\Mfd$ with the Grothendieck singleton pretopology of surjective submersions. 

\subsection{Simplicial manifolds and Lie \texorpdfstring{$n$}{n}-groupoids}

We begin by recalling the fundamental notions of higher Lie groupoids and hypercovers. Our references for this are mainly \cite{Henriques2008,Getzler2009,Zhu2009,BehrendGetzler2017,RogersZhu2020,delHoyoOrtizStudzinski2024}.

A \textbf{simplicial manifold $X$} is a contravariant functor from $\Delta$, the category of finite ordinals $[0]=\{0\}, [1]=\{0, 1\}, \dots, [n]=\{0, 1,\dotsc, n\}, \dots$ with order-preserving maps, to $\Mfd$, the category of smooth manifolds and smooth maps. 
The morphisms in the category $\Delta$ are all generated by two sets of maps which satisfy the so-called cosimplicial identities: the \textbf{coface} maps $\delta^n_i:[n-1] \to [n]$, which are the unique inclusions skipping $i$, for $0\le i \le n$ and the \textbf{codegeneracy} maps $\sigma^n_i: [n+1] \to [n]$, which are the unique surjections identifying $i$ and $i+1$. 

Therefore, a simplicial manifold $X$ consists of a series of manifolds $X_n$, \textbf{face} maps $d^n_i: X_n \to X_{n-1}$ and \textbf{degeneracy} maps $s^n_i: X_n \to X_{n+1}$, for $i=0, \dots, n$ satisfying the following simplicial identities
\begin{equation}\label{eq:face-degen}
    \begin{array}{lll}
        d^{n-1}_i d^{n}_j = & d^{n-1}_{j-1} d^n_i &\text{if}\; i<j,  \\
       s^{n}_i s^{n-1}_j =& s^{n}_{j+1} s^{n-1}_i & \text{if}\; i\leq j,
    \end{array}\qquad  d^n_i s^{n-1}_j =\left\{\begin{array}{ll}
    s^{n-2}_{j-1} d^{n-1}_i  & \text{if}\; i<j, \\
    \id  & \text{if}\; i=j, j+1,\\
    s^{n-2}_j d^{n-1}_{i-1} & \text{if}\; i> j+1.
 \end{array}\right.
\end{equation}
From now on, we drop the upper indices and write instead $d_i$ and $s_i$ for simplicity.  A \textbf{simplicial map $f \in \Hom(X,Y)$} is a natural transformation, or in other words, a family of smooth maps  $f_n:X_n\to Y_n$ that intertwine the face and degeneracy maps on $X$ and~$Y$.

Simplicial sets, i.e. functors $\Delta^{op} \to \Set$, can be seen as discrete simplicial manifolds. The \textbf{standard $n$-simplex} $\Delta^n := \Delta(\_, [n])$ is an example.

An \textbf{$n$-simplex} in $X$ is a point $x\in X_n$. Alternatively, by the Yoneda lemma, this can be seen as a simplicial map $x: \Delta^n \to X$. Since the combinatorics of the simplicial structure of $X$ encode the geometry of the simplex, it is helpful to think of a map $K \to X$ for $K$ a simplicial set, as a \textbf{simplicial diagram} of shape $K$ in $X$. The \textbf{space of simplicial diagrams of shape $K$} in $X$ is denoted by $\hom(K, X)$. Ideally we would like to consider this as a smooth manifold, but this is not always possible, due to the nature of $\hom(K, X)$ as the limit of the diagram
\begin{equation}\label{eq:simplicial-diagram-limit}
    \hom(K, X) \to \prod\limits_{k\ge 0} \Mfd(K_k, X_k) \rightrightarrows \prod\limits_{{\begin{array}{c} {\scriptstyle [m],[n] \in \Delta} \\ {\scriptstyle f \in \Delta([m],[n])}\end{array}}} \Mfd(K_n, X_m).
\end{equation}
Indeed, this limit may not always exist in $\Mfd$. However, this limit always exists, when considering the diagram in a complete category in which $\Mfd$ is embedded by a functor that preserves and reflects limits. The usual choices are: $\PSh(\Mfd)$, the category of presheaves over $\Mfd$, in which $\Mfd$ is embedded via the Yoneda embedding; or $\Top$, the category of topological spaces, in which $\Mfd$ is embedded by the forgetful functor. In the first case, when $\hom(K, X)$ is the Yoneda embedding of a manifold, we say it is \textbf{representable}, while in the second one, when $\hom(K,X)$ is a smooth manifold in the right topology, we simply say it is \textbf{smooth}. These are equivalent, so in the following we make the choice to say $\hom(K, X)$ is \textbf{smooth}, when it is a smooth manifold. 
The first approach was used for example in \cite{Henriques2008, Zhu2009, Li2014, RogersZhu2020, Ronchi2025-thesis}, while the second one appears in \cite{delHoyoOrtizStudzinski2024}.

There are two choices of $K$ that are fundamental in the following: the \textbf{$m$-boundary} $\partial\Delta^m$ and the \textbf{$(m,j)$-horn} $\Lambda^m_j$. The boundary $\partial\Delta^m$ is the simplicial subset obtained from the $m$-simplex $\Delta^m$ by removing the interior, i.e. the unique non-degenerate $m$-simplex $01\dots (m-1)m \in \Delta^m_m$ (and all of its degeneracies). The horn $\Lambda^m_j$ is obtained from the $m$-boundary by additionally removing the $j$-th face $01\dots (j-1)(j+1) \dots m \in \Delta^m_{m-1}$ (and all of its degeneracies). Therefore
\begin{equation*}
    \Lambda^m_j \subset \partial\Delta^m \subset \Delta^m. 
\end{equation*}
More precisely, 
\begin{equation*}
\begin{split}
    (\partial\Delta^m)_l &=  \{ f\in (\Delta^m)_l \mid \Img f \nsupseteq \{0,\dots, m\} \} \subseteq \Delta^m_l,\\
    (\Lambda^m_j)_l &= \{ f \in (\Delta^m)_l \mid \Img f \nsupseteq \{0,\dots,\widehat{j},\dots,m\}\} \subseteq (\partial\Delta^m)_l \subseteq \Delta^m_l.
\end{split}
\end{equation*}
In particular $\partial\Delta^0 = \Lambda^0_0 = \emptyset$. 
The spaces of horn and boundary shaped simplicial diagrams are the \textbf{$(m,j)$-horn space} $\Lambda^m_j(X) := \hom(\Lambda^m_j, X)$ and the \textbf{$m$-boundary} space (or \textbf{matching space}) $\partial\Delta^m(X) := \hom(\partial\Delta^m, X)$. To give some intuition, $\partial\Delta^m(X)$ is the space of all possible boundaries of $m$-simplices in $X$. In the same way, $\Lambda^m_j(X)$ is the space of all possible configurations of $(m-1)$-simplices in the shape of an $(m,j)$-horn that exist in $X$. Some of these are pictured in Figure \ref{fig:horns}.

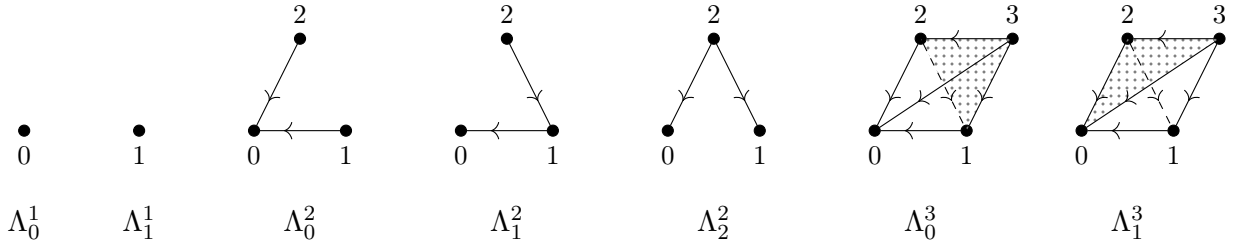
\begin{figure}[!ht]
\centering
\begin{adjustbox}{width=\textwidth}
\pgfdeclarelayer{nodelayer}
\pgfdeclarelayer{edgelayer}
\pgfsetlayers{nodelayer,main,edgelayer}
\begin{tikzpicture}[scale=0.6,
    none/.style={}, 
    dot/.style={fill=black, draw=black, shape=circle, scale=0.40}, 
    fill-grey/.style={-, fill={rgb,255: red,220; green,220; blue,220}, draw=none, fill opacity=0.5},
    directed/.style={decoration={markings,mark=at position 0.7 with {\arrow{To[scale=1.25]}}}, postaction={decorate}},
    dashdir/.style={dashed, decoration={markings,mark=at position 0.7 with {\arrow{To[scale=1.25]}}}, postaction={decorate}},
    ]
    \begin{pgfonlayer}{nodelayer}
        \node [style=none] (25) at (20.5, 1) {};
        \node [style=none] (26) at (19.5, 3) {};
        \node [style=none] (27) at (21.5, 3) {};
        \node [style=none] (0) at (0, -1) {$\Lambda^1_0$};
        \node [style=none] (1) at (2.5, -1) {$\Lambda^1_1$};
        \node [style=none] (2) at (6, -1) {$\Lambda^2_0$};
        \node [style=none] (3) at (10.5, -1) {$\Lambda^2_1$};
        \node [style=none] (4) at (15, -1) {$\Lambda^2_2$};
        \node [style=none] (5) at (19.5, -1) {$\Lambda^3_0$};
        \node [style=none] (6) at (24, -1) {$\Lambda^3_1$};
        \node [style=dot, label={below:\small{0}}] (7) at (0, 1) {};
        \node [style=dot, label={below:\small{1}}] (8) at (2.5, 1) {};
        \node [style=dot, label={below:\small{0}}] (9) at (5, 1) {};
        \node [style=dot, label={below:\small{1}}] (10) at (7, 1) {};
        \node [style=dot, label={above:\small{2}}] (11) at (6, 3) {};
        \node [style=dot, label={below:\small{0}}] (12) at (9.5, 1) {};
        \node [style=dot, label={below:\small{1}}] (13) at (11.5, 1) {};
        \node [style=dot, label={above:\small{2}}] (14) at (10.5, 3) {};
        \node [style=dot, label={below:\small{0}}] (15) at (14, 1) {};
        \node [style=dot, label={below:\small{1}}] (16) at (16, 1) {};
        \node [style=dot, label={above:\small{2}}] (17) at (15, 3) {};
        \node [style=dot, label={below:\small{0}}] (18) at (18.5, 1) {};
        \node [style=dot, label={below:\small{1}}] (19) at (20.5, 1) {};
        \node [style=dot, label={above:\small{2}}] (20) at (19.5, 3) {};
        \node [style=dot, label={above:\small{3}}] (21) at (21.5, 3) {};
        \node [style=none] (22) at (23, 1) {};
        \node [style=none] (23) at (24, 3) {};
        \node [style=none] (24) at (26, 3) {};
        \node [style=dot, label={below:\small{0}}] (28) at (23, 1) {};
        \node [style=dot, label={below:\small{1}}] (29) at (25, 1) {};
        \node [style=dot, label={above:\small{2}}] (30) at (24, 3) {};
        \node [style=dot, label={above:\small{3}}] (31) at (26, 3) {};
    \end{pgfonlayer}
    \begin{pgfonlayer}{edgelayer}
        \draw [style=dotted-fill] (27.center)
                to (25.center)
                to (26.center)
                to cycle;
        \draw [style=directed] (10) to (9);
        \draw [style=directed] (11) to (9);
        \draw [style=directed] (13) to (12);
        \draw [style=directed] (14) to (13);
        \draw [style=directed] (17) to (15);
        \draw [style=directed] (17) to (16);
        \draw [style=directed] (20) to (18);
        \draw [style=directed] (19) to (18);
        \draw [style=directed] (21) to (20);
        \draw [style=directed] (21) to (19);
        \draw [style=dashdir] (20) to (19);
        \draw [style=dotted-fill] (23.center)
                to (24.center)
                to (22.center)
                to cycle;
        \draw [style=directed] (30) to (28);
        \draw [style=directed] (29) to (28);
        \draw [style=directed] (31) to (30);
        \draw [style=directed] (31) to (29);
        \draw [style=dashdir] (30) to (29);
        \draw [style=directed] (31) to (28);
        \draw [style=directed] (21) to (18);
    \end{pgfonlayer}
\end{tikzpicture}
\end{adjustbox}
\caption{A few low-dimensional horns. The dotted faces denote boundaries of triangles without interior.}
\label{fig:horns}
\end{figure}

\begin{definition}
    A \textbf{higher Lie groupoid} is a simplicial manifold $X$ such that the horn spaces $\Lambda^m_j(X)$ are smooth and the horn projections 
    \begin{equation*}
        p^m_j: X_m \to \Lambda^m_j(X)
    \end{equation*}
    induced by the inclusion $\Lambda^m_j \into \Delta^m$ are surjective submersions for all $m \ge 0$, $0 \le j \le m$. If, additionally, $p^m_j$ are diffeomorphisms for all $m > n$, $0 \le j \le m$, we say $X$ is a \textbf{Lie $n$-groupoid}. A \textbf{morphism of higher Lie groupoids} is a simplicial map. We denote the category of Lie $n$-groupoids by $\Liengpd$. Its hom-sets are then $\mathsf{SMfd}(X, Y)$. 
\end{definition}

We usually think of these conditions as horn filling conditions, i.e. for any diagram on the left of Figure \ref{fig:horn-fillers}, there is a diagram on the right. This is convenient in pictorial proofs. Of course, in the smooth case, one needs to make sure that the right-to-left map is actually a surjective submersion.

\begin{figure}[!h]
    \centering
    \includegraphics[width=0.6\linewidth]{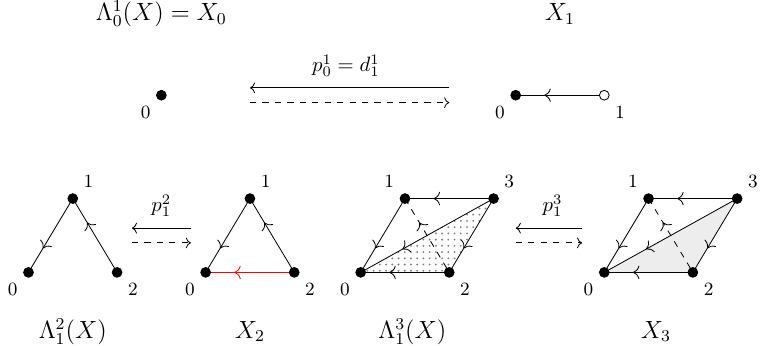}
    \caption{Some horn filling conditions.}
    \label{fig:horn-fillers}
\end{figure}

\begin{example}
    A Lie groupoid in the usual sense can be seen as a Lie 1-groupoid by identifying it with its nerve. We implicitly perform this identification all throughout this article. In particular, a smooth manifold can be seen as a Lie 0-groupoid, by identifying it with the nerve of its identity groupoid. For a detailed description of a Lie 2-groupoid as an object with a set of ternary multiplications on it see \cite[Prop.-Def. 2.16]{Zhu2009}. For this perspective in the case of general Lie $n$-groupoids see \cite[Prop. 1.2.67]{Ronchi2025-thesis}.
\end{example}

The category of Lie $n$-groupoids has the natural structure of a simplicial category, that is, a category in which mapping spaces $\Maps(X, Y)$ are simplicial sets and the structure maps of the category are compatible with this structure (see e.g. \cite[\href{https://kerodon.net/tag/00JQ}{Tag 00JQ}]{kerodon}). We recall the main points of this notion and refer for example to \cite[\S 1.1.2]{Ronchi2025-thesis} and references therein for more details on the definition of mapping space and its interpretation. Our goal here is to recall the bicategorical structures of the full subcategories of Lie 1-groupoids and Lie 2-groupoids, respectively.\footnote{Importantly, the subcategory of Lie 2-groupoids would naturally be a tricategory, so we perform a truncation in this case.} These turn out to be (2,1)-categories, i.e. bicategories in which all 2-morphisms are invertible.

\begin{definition}
    Let $K$ be a levelwise finite simplicial set, and $X$ be a simplicial manifold. We denote by $X \otimes K$ the simplicial manifold defined at each level by 
    \begin{equation*}
        (X \otimes K)_m := \coprod_{r\in K_m} X_m,
    \end{equation*}
    This is known as the \textbf{copowering}. As a simplicial set, it is simply $X \times K$.
    The \textbf{mapping space} between two simplicial manifolds $X$ and $Y$ is the simplicial set defined at each level by 
    \begin{equation*}
        \Maps(X, Y)_m = \mathsf{SMfd}(X \otimes \Delta^m, Y).
    \end{equation*}
\end{definition}

The fact that $\Maps(X,Y)$ between Lie $n$-groupoids has a natural structure of a simplicial set can be interpreted as the fact that there are higher dimensional morphisms between morphisms. In fact, simplicial maps $X \otimes \Delta^1\to Y$ are homotopies between morphisms $X\to Y$ (the 0-simplices), and simplicial maps $X \otimes \Delta^n \to Y$ are higher homotopies between homotopies. It is important to note that even if $X$ is a Lie $n$-groupoid, $X \otimes \Delta^n$ is generally just a simplicial manifold. 

In particular, for Lie 1-groupoids, a straightforward computation shows that 1-homotopies between simplicial maps are the same as natural transformations between functors, when seeing the Lie 1-groupoids as internal categories in $\Mfd$. These turn out to be equivalent to the data of a map $X_0 \to Y_1$ satisfying the naturality conditions. For Lie 2-groupoids, a similar interpretation is possible, as it was proven in \cite[\S 7]{Duskin2001} that they are nerves of bigroupoids, see e.g. \cite[Prop. 3.3]{delHoyoStefani2017}, \cite[Prop. 4.4]{BullejosFaroBlanco2005}. 

By extending a classical result in the theory of simplicial sets (\cite[Thm. 6.9]{May1967}), as in \cite[Lemma 3.18]{RonchiZhu2026}, we have that the mapping space $\Maps(X,Y)$ between Lie groupoids is a groupoid. Therefore, we naturally obtain a (2,1)-category of Lie groupoids, as all the 2-morphisms are arrows of their respective hom-groupoids, and thus invertible. This can also be seen by noticing that when 2-morphisms are interpreted as natural transformations, they take values in $Y_1$, so they are invertible pointwise. We denote the (2,1)-category of Lie groupoids, functors and natural transformations by $\mathsf{LieGpd}$. 

For the same reason, Lie 2-groupoids naturally form a (3,1)-category, because when $X$ and $Y$ are Lie 2-groupoids, $\Maps(X,Y)$ is a 2-groupoid. To obtain a (2,1)-category, we truncate each mapping space by taking its fundamental groupoid $\Pi_1\Maps(X,Y)$. This forces us to effectively consider 2-morphisms as natural transformations up to \textit{modifications}, i.e. the 3-morphisms in this category. By abuse of notation, we denote the (2,1)-category of Lie 2-groupoids, simplicial maps (2-functors) and 1-homotopies (natural transformations) up to 2-homotopies (modifications) by $\mathsf{Lie_2Gpd}$. Occasionally, we will refer to $\Liengpd$ for general $n$ as a bicategory. By this we mean the bicategory obtained by taking the hom-spaces $\Pi_1\Maps$.

\subsection{Hypercovers and Morita equivalences}
We recall the following notion of generalized isomorphism between higher Lie groupoids, which serves to express when two of them model the same higher differentiable stack, as we discuss below. 
 
\begin{definition}
    Let $\partial: X_k \to \partial\Delta^k(X)$ be the canonical boundary projection $\partial = (d_0, \dots, d_k)$.
    A morphism $f: X \to Y$ between Lie $n$-groupoids is a \textbf{hypercover} if the conditions $\Acyc(k)$ hold for all $0\le k < n$ and $\Acyc!(n)$ holds. These conditions are:
    \begin{itemize}
        \item $\Acyc(k)$: The space $\partial\Delta^k(X)\times_{f, \partial \Delta^k(Y), \partial} Y_{k}$ is smooth, and the map $$(\partial, f): X_k \to \partial \Delta^k(X)\times_{f, \partial \Delta^k(Y), \partial} Y_{k}$$is a surjective submersion.
        \item $\Acyc!(n)$: The space $\partial\Delta^n(X)\times_{f, \partial \Delta^n(Y), \partial} Y_{n}$ is smooth, and the map $$(\partial, f): X_n \to \partial \Delta^n(X)\times_{f, \partial \Delta^n(Y), \partial} Y_{n}$$ is a diffeomorphism.
    \end{itemize}
    We denote hypercovers by the notation $\overset{\sim}{\twoheadrightarrow}$. We picture the spaces $\partial \Delta^k(X)\times_{\partial \Delta^k(Y)} Y_{k}$ for low $k$ in Figure \ref{fig:acyclic-conditions-target-spaces}. 
\end{definition}

While the space $\partial\Delta^k(X)\times_{\partial \Delta^k(Y)} Y_{k}$ may generally not be a manifold because it is obtained as a limit similar to \eqref{eq:simplicial-diagram-limit}, it can be shown recursively that it is, provided $f$ satisfies $\Acyc(k-1)$, see e.g. \cite[Lemma 2.4]{Zhu2009}. Additionally, if a map $f: X\to Y$ between Lie $n$-groupoids satisfies the conditions in our definition of hypercover, then it satisfies the condition $\Acyc!(k)$ for all $k>n$, as well. Particularly, there is no distinction between hypercovers of simplicial manifolds and hypercovers of Lie $n$-groupoids. See e.g. \cite[Lemma 3.9]{Zhu2009a}. 

\begin{figure}[!h]
    \centering
    \includegraphics[width=0.70\linewidth]{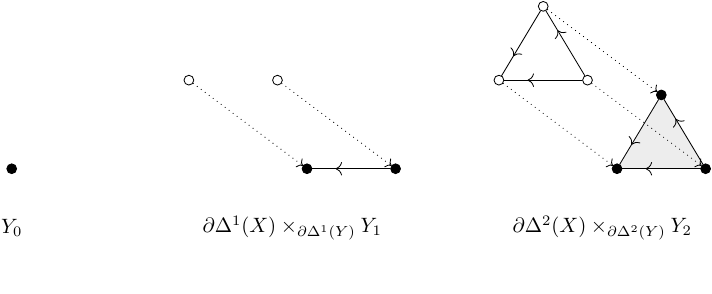}
    \caption{The target spaces appearing in the conditions $\Acyc(k)$ and $\Acyc!(k)$ for $k=0,1,2$. The dotted arrows represent $f$: they connect a point in $X$ to its image in $Y$. The shaded triangle is a 2-simplex with interior, as opposed to a boundary.}
    \label{fig:acyclic-conditions-target-spaces}
\end{figure}

In the incomplete category of fibrant objects of higher Lie groupoids, as defined in \cite{RogersZhu2020}, hypercovers are at the same time weak equivalences and also Kan fibrations, and they are precisely the acyclic fibrations \cite[Prop. 6.7]{RogersZhu2020}, hence the arrow notation. Weak equivalences in this setting are characterized by conditions similar to (but generally weaker than) those defining hypercovers, as in \cite[Prop. 7.16]{RogersZhu2020}. Additionally any weak equivalence factors as a span of hypercovers due to the factorization Lemma \cite[Prop. 2.5]{RogersZhu2020}, so we define the following equivalence relation. 

\begin{definition}
    Two Lie $n$-groupoids $X$ and $Y$ are \textbf{Morita equivalent} if there exists a third Lie $n$-groupoid $Z$ and a span of hypercovers
    \begin{equation*}
        X \overset{\sim}{\twoheadleftarrow} Z \overset{\sim}{\twoheadrightarrow} Y. 
    \end{equation*}
    This generates an equivalence relation, which coincides with the one generated by spans of weak equivalences. 
\end{definition}

\begin{example}
    Weak (i.e. Morita) equivalences of Lie groupoids are ``smooth'' categorical equivalences \cite{MoerdijkMrcun2005, delHoyo2013}. A hypercover of Lie 1-groupoids is in particular a categorical equivalence, as $\Acyc(0)$ implies surjectivity on objects, which is stronger than essential surjectivity, while $\Acyc!(1)$ is the condition of full faithfulness. 
\end{example}

\begin{remark}\label{rem:hypercover-properties}
    Hypercovers satisfy the following properties (see e.g. \cite[Thm. 2.12]{Wolfson2016}, \cite[\S 2]{Zhu2009}):
    \begin{itemize}
        \item[(H1)] Isomorphisms are hypercovers. 
        \item[(H2)] Hypercovers are closed under composition: the composition of two hypercovers is a hypercover. 
        \item[(H3)] Hypercovers are closed under pullback (also called base change): If $X \overset{\sim}{\twoheadrightarrow} Y \leftarrow Z$ is a diagram of Lie $n$-groupoids in which the map $X \overset{\sim}{\twoheadrightarrow} Y$ is a hypercover, then the pullback $X\times_Y Z$ exists, is a Lie $n$-groupoid  and the map $X\times_Y Z \overset{\sim}{\twoheadrightarrow} Z$ is a hypercover. 
    \[\begin{tikzcd}[ampersand replacement=\&,cramped]
	{X\times_Y Z} \& Z \\
	X \& Y
	\arrow["\sim", dashed, two heads, from=1-1, to=1-2]
	\arrow[from=1-1, to=2-1]
	\arrow["\lrcorner"{anchor=center, pos=0.125}, draw=none, from=1-1, to=2-2]
	\arrow[from=1-2, to=2-2]
	\arrow["\sim", two heads, from=2-1, to=2-2]
    \end{tikzcd}\]
    \end{itemize}
\end{remark}

\subsection{Localization} 

It is a well-known fact in the theory of Lie groupoids that Morita equivalence classes of Lie groupoids correspond one-to-one with isomorphism classes of differentiable stacks. To make this precise as a correspondence of bicategories, one needs to extend the bicategory of Lie groupoids to a bicategory where Morita equivalences are invertible. This procedure is known as localization (see e.g. \cite[\href{https://kerodon.net/tag/01M4}{Tag 01M4}]{kerodon}).

\begin{definition}
    Let $\Cat$ be a category and $\huaW$ a class of morphisms. A category $\mathsf{D}$ with a functor $L:\Cat \to \mathsf{D}$ is a \textbf{localization of $\Cat$ at $\huaW$} if $L$ maps morphisms in $\huaW$ to isomorphisms in $\mathsf{D}$ and $(\mathsf{D}, L)$ is universal with this property: for any other functor $L': \Cat \to \mathsf{D'}$ that sends $\huaW$ to isomorphisms, $L'$ factors uniquely as $F \circ L$ for a unique functor $F: \mathsf{D} \to \mathsf{D'}$. 

    A 1-morphism $f:X \to Y$ in a bicategory $\Cat$ is an \textbf{equivalence} if there exist a 1-morphisms $g$ and invertible 2-morphisms $f \circ g \implies id_Y$, $g \circ f \implies id_X$. 
    
    If $\Cat$ is a bicategory, a \textbf{localization of $\Cat$ at a class of 1-morphisms $\huaW$} is a bicategory $\mathsf{D}$ equipped with a 2-functor\footnote{I.e. a morphism of bicategories \cite{Benabou1967}. These are sometimes referred to as pseudofunctors.} $L:\Cat \to \mathsf{D}$ that maps 1-morphisms in $\huaW$ to \textit{equivalences} and $L$ is universal with this property. The same process can also be defined for simplicial categories and more general $\infty$-categories.
\end{definition}

Clearly, because localizations are defined by a universal property, they are not unique, but any two localizations are equivalent as categories or bicategories. In fact, localizations of Lie groupoids that make Morita equivalences invertible can be constructed in several different ways:
\begin{enumerate}
    \item by left/right-principal bibundles, which are also known as Hilsum-Skandalis (HS) morphisms,
    \item by calculus of fractions with respect to weak equivalences $\huaW$ \cite{Pronk1996, Tommasini2016, PronkScull2022},
    \item by anafunctors \cite{Makkai1996, Bartels2006, Roberts2012}. 
\end{enumerate}
These different approaches and partial comparisons between them can be found for example in the references \cite{MoerdijkMrcun2005, Blohmann2008, Zhu2009, delHoyo2013, Roberts2012, MeyerZhu2015, Tommasini2016, FarsiScullWatts2024}. For the correspondence with differentiable stacks in particular, see \cite{Pronk1996, BehrendXu2011, Blohmann2008, delHoyo2013}. Generally, localizations encounter size and choice issues: localizations of locally small categories may not be locally small, and some procedures may require choices of pullbacks. For these more abstract issues we refer to \cite{Roberts2016, Tommasini2016, PronkScull2022}, and references therein. In the present work, we mainly use approaches 1 and 3 for Lie groupoids, a good comparison of which can be found in \cite{MeyerZhu2015}, and approach 3 for Lie 2-groupoids, though we make no claim of whether this is indeed a localization of the bicategory $\mathsf{Lie_2Gpd}$ (\cite[Thm 7.2]{Roberts2012} only works for bicategories of internal categories such as Lie 1-groupoids). We also recall the definition of morphisms of anafunctors of Lie 2-groupoids in the more general bicategory obtained by approach 2, as these appear in Theorem \ref{thm:1-1-morph-to-2-hom-diamond}.

\begin{remark}\label{rmk:higher-HSbibundles}
Hilsum-Skandalis (HS) morphisms were also constructed for Lie 2-groupoids in \cite{Li2014}, but we do not need them in the present discussion. A way to extend HS morphisms of Lie groupoids to the case of higher Lie groupoids and to construct generalized morphisms of higher Lie groupoids modeled on principal $n$-bibundles is the subject of the upcoming work \cite{BlohmannKrishnaZhu:2026}. Analogously to the Lie groupoid case, it is shown there that biprincipal bibundles, spans of weak equivalences, and spans of hypercovers all provide equivalent ways to define Morita equivalences of  Lie $n$-groupoids. 
\end{remark}

\begin{remark}
    In \cite{GinotStienon2015}, the authors show that hypercovers form a left multiplicative system for the category of strict Lie 2-groupoids. Hence, they produce a localization of this category by hypercovers. To localize the \textit{bicategory} of Lie 2-groupoids at the class of hypercovers $\huaH$ and obtain a bicategory, one would additionally need to check BF1, BF4, and BF5 from \cite{Pronk1996}. Note that BF1 does not literally hold, but it can be replaced by the standard weakening of BF1 requiring only identity morphisms to belong to the class $\huaH$ (see \cite{Tommasini2016}, \cite[\S~2]{PronkScull2022}). On the other hand, BF4 can be shown to hold for hypercovers, but the condition BF5 is that $\huaH$ is closed under 2-isomorphisms, which is not generally true. Consider for example the pair groupoid $P(M):= M\times M \rightrightarrows M$ of a manifold $M$ with more than one point. The identity $id_{P(M)}$ is a hypercover. Consider now the map $c_p: P(M) \to P(M)$ induced by the constant map $M \to \ast \to M$, that sends each point of $M$ to the same point $p\in M$. There is a 2-isomorphism $\alpha: id_{P(M)} \implies c_p$ given by the map 
    \begin{equation*}
        M \longrightarrow M \times M, \qquad
        x \longmapsto (x,p) : x \to p,
    \end{equation*}
    but $c_p$ is clearly not a hypercover, though it is a weak equivalence. Nevertheless, one can extend the class of hypercovers by adding those weak equivalences that are 2-isomorphic to the hypercovers, thus forming the completed class $\widebar{\huaH}$. This class then admits a calculus of fractions. Either way, the localization at $\widebar{\huaH}$ ends up being equivalent to the localization at weak equivalences by the factorization lemma. 
\end{remark}

\begin{remark}
    Because of the equivalence of localizations of the category of Lie groupoids at Morita equivalences and the category of differentiable stacks, localizations of the category of Lie $n$-groupoids at Morita equivalences can be thought of as models for the category of differentiable $n$-stacks. More general localizations of simplicial subcategories of Lie $n$-groupoids are discussed in \cite[\S 2.2,  \S 7.2.1]{RogersZhu2020}. See also \cite{BehrendGetzler2017} and references therein.
\end{remark}

\begin{definition}\label{def:HS-bibun}
A \textbf{Hilsum-Skandalis morphism} (or \textbf{HS bibundle}) of Lie groupoids $G \overset{A}{\dashrightarrow} H$ is a left-principal bibundle
\[\begin{tikzcd}[ampersand replacement=\&,cramped, row sep=small, column sep=scriptsize]
	{H_1} \& A \& {G_1} \\
	{H_0} \&\& {G_0,}
	\arrow["\curvearrowright"{description}, draw=none, from=1-1, to=1-2]
	\arrow[shift left, from=1-1, to=2-1]
	\arrow[shift right, from=1-1, to=2-1]
	\arrow["l", from=1-2, to=2-1]
	\arrow["r"', two heads, from=1-2, to=2-3]
	\arrow["\curvearrowleft"{description}, draw=none, from=1-3, to=1-2]
	\arrow[shift left, from=1-3, to=2-3]
	\arrow[shift right, from=1-3, to=2-3]
\end{tikzcd}\]
that is, a smooth manifold $A$ with a moment map $r: A \to G_0$ on which $G$ acts on the right and a moment map $l: A \to H_0$ on which $H$ acts \textit{principally} on the left. In particular, $r$ is a surjective submersion. A HS bibundle which is also right principal is a Morita equivalence. 
\end{definition}

The composition of two HS morphisms $G \overset{A}{\dashrightarrow} H \overset{B}{\dashrightarrow} K$ is given by the $K$-$G$-bibundle $(A \times_{H_0} B)/H$, where the quotient is taken with respect to the diagonal $H$-action $h\cdot(a,b) \mapsto (h \cdot a, b \cdot h^{-1})$. \textbf{2-morphisms} of HS bibundles are biequivariant maps. Any such map is in fact an isomorphism, since equivariant maps of principal groupoid bundles are diffeomorphisms. With this data, we obtain a (2,1)-category denoted $\mathsf{Bun(LieGpd)}$. Note, in fact, that well-definedness of composition, associativity, and unitality of the trivial bibundle all hold up to natural invertible 2-morphisms. See \cite[Thm. 7.15]{MeyerZhu2015}. It can be readily shown that the equivalences in this bicategory are precisely the biprincipal bibundles, i.e. the Morita equivalences. A proof that $\mathsf{Bun(LieGpd)}$ is a localization can be found in \cite[Thm. 2.11]{MoerdijkMrcun2005}.

The anafunctor approach produces a smaller but equivalent localization to the calculus of fractions method. In fact, we have that an anafunctor can be seen as a particular kind of 1-morphism in the localization by calculus of fractions at weak equivalences. While 2-morphisms in the latter are defined by equivalence classes, 2-morphisms of anafunctors can also be defined as particular representatives of these. 
Since anafunctors can be defined in general for Lie $n$-groupoids of any order $n$, we write the following definitions in general, but we will mostly be thinking of the bicategories $\mathsf{LieGpd}$ and $\mathsf{Lie_2Gpd}$. We denote the anafunctor bicategory by $\mathsf{Ana(Lie_n Gpd)}$, and the calculus of fractions one by $\mathsf{Lie_n Gpd}[\huaW^{-1}]$. We will describe the 1- and 2-morphisms of $\mathsf{Ana(Lie_n Gpd)}$ and only the 2-morphisms of $\mathsf{Lie_n Gpd}[\huaW^{-1}]$. 

\begin{definition}
    An \textbf{anafunctor} of Lie $n$-groupoids $X \dashrightarrow Y$ consists of a Lie $n$-groupoid $Z$ and a span $X \overset{\sim}{\twoheadleftarrow} Z \to Y$ where the \textit{domain-side leg} is a hypercover. By definition, a Morita equivalence is then an anafunctor where both legs are hypercovers.  
\end{definition}

Because of properties (H2) and (H3) in Remark \ref{rem:hypercover-properties}, anafunctors can be composed by the usual pullback (fiber product) of simplicial manifolds as follows: If $X \overset{\sim}{\twoheadleftarrow} Z \to Y$ and $Y \overset{\sim}{\twoheadleftarrow} Q \to Y'$ are two anafunctors, then their composition is given by the pullback

\[\begin{tikzcd}[ampersand replacement=\&,cramped,sep=tiny]
	\&\& {Z\times_YQ} \&\& \\
	\& Z \&\& Q \\
	X \&\& Y \&\& {Y',}
	\arrow["\sim"', two heads, from=1-3, to=2-2]
	\arrow[from=1-3, to=2-4]
	\arrow["\lrcorner"{anchor=center, pos=0.125, rotate=-45}, draw=none, from=1-3, to=3-3]
	\arrow["\sim"', two heads, from=2-2, to=3-1]
	\arrow[from=2-2, to=3-3]
	\arrow["\sim"', two heads, from=2-4, to=3-3]
	\arrow[from=2-4, to=3-5]
\end{tikzcd}\]
which is well defined up to canonical isomorphism. Associativity, and unitality for the identity anafunctor $X \overset{id}{\leftarrow} X \overset{id}{\rightarrow} X$ both hold up to natural isomorphism, as well. 

\begin{definition}\label{def:2mor ana}
A \textbf{2-morphism of anafunctors in} $\mathsf{Ana(Lie_n Gpd)}$ between $X \overset{\sim}{\twoheadleftarrow} Z \to Y$ and $X \overset{\sim}{\twoheadleftarrow} Z' \to Y$, is a natural transformation $\beta$ between the maps $Z \times_X Z' \to Z \to Y$ and $Z \times_X Z' \to Z' \to Y$ as in the following diagram:
\begin{equation}\label{diag:2-hom-ananatural}
\begin{tikzcd}[ampersand replacement=\&,cramped]
	\&\& Z \&\& \\
	X \&\& {Z\times_XZ'} \&\& Y \\
	\&\& {Z'}
	\arrow["\sim"'{pos=0.8}, two heads, from=1-3, to=2-1]
	\arrow[from=1-3, to=2-5]
	\arrow["\beta", curve={height=-44pt}, between={0.2}{0.7}, Rightarrow, nfold, from=1-3, to=3-3]
	\arrow["\sim", two heads, from=2-3, to=1-3]
	\arrow["\sim"', two heads, from=2-3, to=3-3]
	\arrow[""{name=0, anchor=center, inner sep=0}, "\sim"{pos=0.8}, two heads, from=3-3, to=2-1]
	\arrow[from=3-3, to=2-5]
	\arrow["\lrcorner"{anchor=center, pos=0.125, rotate=-90}, draw=none, from=2-3, to=0]
\end{tikzcd}
\end{equation}

A \textbf{2-morphism of anafunctors in} $\mathsf{Lie_n Gpd}[\huaW^{-1}]$ between $X \overset{\sim}{\twoheadleftarrow} Z \to Y$ and $X \overset{\sim}{\twoheadleftarrow} Z' \to Y$ is an equivalence class\footnote{This equivalence relation is detailed in \cite[\S 2.3]{Pronk1996}, \cite[p. 260]{Tommasini2016}. We do not recall it here, as our discussion only ever requires existence of single representatives. For a useful comparison criterion for 2-morphisms see \cite[Prop. 1.1]{Tommasini2016}.} of diagrams of the form 
\begin{equation}\label{diag:2-hom-diamond}
\begin{tikzcd}[ampersand replacement=\&,cramped]
	\&\& Z \&\& \\
	X \&\& Q \&\& Y \\
	\&\& {Z'}
	\arrow["f"', two heads, from=1-3, to=2-1]
	\arrow["\sim"'{pos=0.8}, draw=none, from=1-3, to=2-1]
	\arrow[from=1-3, to=2-5]
	\arrow["\alpha"', curve={height=30pt}, between={0.2}{0.7}, Rightarrow, nfold, from=1-3, to=3-3]
	\arrow["\beta", curve={height=-30pt}, between={0.2}{0.7}, Rightarrow, nfold, from=1-3, to=3-3]
	\arrow["q", from=2-3, to=1-3]
	\arrow["\sim"{description}, shift right=3, draw=none, from=2-3, to=1-3]
	\arrow["{q'}"', from=2-3, to=3-3]
	\arrow["\sim"{description}, shift left=3, draw=none, from=2-3, to=3-3]
	\arrow["{f'}", two heads, from=3-3, to=2-1]
	\arrow["\sim"{pos=0.8}, draw=none, from=3-3, to=2-1]
	\arrow[from=3-3, to=2-5]
\end{tikzcd}
\end{equation}
where the maps $f \circ q$ and $f' \circ q'$ are hypercovers, and $\alpha, \beta$ are 2-morphisms in the bicategory $\Liengpd$ that make the diagram commute. Note that $q$ and $q'$ are not necessarily hypercovers, but they are weak equivalences, by the 2-out-of-3 property.
\end{definition}

Vertical composition of 2-morphisms in $\mathsf{Ana(Lie_n Gpd)}$ is straightforward: it is obtained by composing natural transformations by pasting diagrams and using pullbacks when appropriate, to obtain a diagram like \eqref{diag:2-hom-ananatural} again. Horizontal composition is more involved and we refer to \cite{Roberts2012, MeyerZhu2015}. We refer to \cite{Pronk1996} for the compositions of 2-morphisms in $\mathsf{Lie_n Gpd}[\huaW^{-1}]$. A detailed discussion of the choices required can be found in \cite{Tommasini2016}. 

\begin{remark}
    In \cite[Lemma 34]{FarsiScullWatts2024}, it is shown that any 2-morphism of anafunctors between action (1-)groupoids of the form \eqref{diag:2-hom-diamond} is equivalent to one of the form \eqref{diag:2-hom-ananatural}. To extend this result to general Lie 2-groupoids would allow to rewrite \eqref{diag:1-1-morph-to-2-hom-diamond} in Theorem \ref{thm:1-1-morph-to-2-hom-diamond} as an equivalent 2-morphism in the form of \eqref{diag:2-hom-ananatural}. We leave this to future work. See also the comparison result \cite[Prop. 1.1]{Tommasini2016}, which details a criterion for two 2-morphisms in $\mathsf{Lie_n Gpd}[\huaW^{-1}]$ to be representatives of the same equivalence class. 
\end{remark}

\begin{remark}
    In \cite{ChenDuWang2019}, the authors show explicitly that $\mathsf{Ana(LieGpd)}$ is a (2,1)-category, i.e. its hom-spaces are groupoids. 
\end{remark}

\subsection{From bibundles to anafunctors of Lie groupoids}\label{sec:Bun-to-Ana-LieGpd}

We now recall the biaction groupoid construction which defines a 2-functor from $\mathsf{Bun(LieGpd)}$ to $\mathsf{Ana(LieGpd)}$. This is written in full in \cite[\S 6.10]{MeyerZhu2015} for this specific case. The construction also appears in \cite[Theorem 2.11]{MoerdijkMrcun2005}, where it is used to prove the equivalence between $\mathsf{Bun(LieGpd)}$ and the Pronk bicategorical localization of $\mathsf{LieGpd}$ at weak equivalences.  

We begin by constructing an anafunctor of Lie groupoids from an HS bibundle. Let $A$ and $B$ be Lie groupoids, and $P: A \dashrightarrow B$ a bibundle between them which is left principal (with respect to the $B$-action) as in Definition \ref{def:HS-bibun}. Then we construct the \textbf{biaction groupoid} $\widetilde{P}$ by taking 
\begin{equation}\label{eq:biaction-groupoid}
\begin{split}
    &\widetilde{P}_1 := B_1 \times_{s, B_0, l} P \times_{r,A_0,t} A_1,  \qquad \widetilde{P}_0 := P,\\
    &s(b, p, a) = p \cdot a, \qquad t(b,p,a) = b \cdot p,\\
    &\text{and whenever } p'\cdot a' = b \cdot p, \quad (b',p',a')(b, p, a) = (b'b, b^{-1}\cdot p', a'a) = (b'b, p\cdot (a')^{-1}, a'a).
\end{split}
\end{equation}
We depict these structure maps in Figure \ref{fig:biaction-groupoid}.

\begin{figure}[!h]
    \centering
    \includegraphics[width=0.65\linewidth]{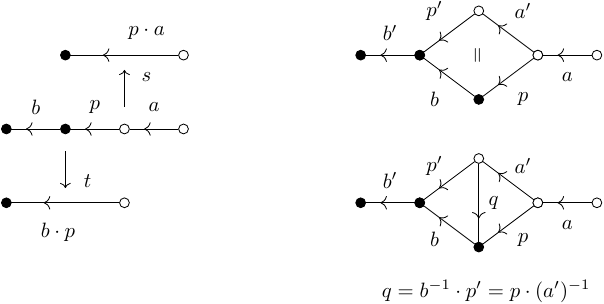}
    \caption{Source, target and multiplication of the biaction groupoid $\widetilde{P}$.}
    \label{fig:biaction-groupoid}
\end{figure}

This construction gives a span of Lie groupoids $B \leftarrow \widetilde{P} \overset{\sim}{\twoheadrightarrow} A$, where the arrow $\widetilde{r}$ to the right is a hypercover, i.e. an anafunctor. The proof of this is a straightforward consequence of the principality of the $B$-action.

Now starting from a biequivariant map $\alpha:P\to Q$, where $Q$ is another HS bibundle from $A$ to $B$, we can produce a 2-morphism in $\mathsf{Ana(LieGpd)}$ in the following way:
\[\begin{tikzcd}[ampersand replacement=\&,cramped, row sep=scriptsize]
	{B_1} \& P \& {A_1} \&\&\&\& {\widetilde{P}} \&\& \\
	{B_0} \&\& {A_0} \&\& A \&\& {\widetilde{P}\times_A\widetilde{Q}} \&\& B \\
	\& Q \&\&\&\&\& {\widetilde{Q}}
	\arrow[shift right, from=1-1, to=2-1]
	\arrow[shift left, from=1-1, to=2-1]
	\arrow["l"', from=1-2, to=2-1]
	\arrow["r", two heads, from=1-2, to=2-3]
	\arrow["\alpha", from=1-2, to=3-2]
	\arrow["\cong"{description}, shift right=3, draw=none, from=1-2, to=3-2]
	\arrow[shift right, from=1-3, to=2-3]
	\arrow[shift left, from=1-3, to=2-3]
	\arrow["{\widetilde{r}}"', two heads, from=1-7, to=2-5]
	\arrow["\sim"{description, pos=0.9}, shift right=3, draw=none, from=1-7, to=2-5]
	\arrow["{\widetilde{\,l\,}}", from=1-7, to=2-9]
	\arrow["{\widetilde{\alpha}}", curve={height=-44pt}, between={0.2}{0.8}, Rightarrow, nfold, from=1-7, to=3-7]
	\arrow["{\widetilde{\,\cdot\,}}", between={0.2}{0.8}, Rightarrow, nfold, from=2-3, to=2-5]
	\arrow["\sim", two heads, from=2-7, to=1-7]
	\arrow["\sim"', two heads, from=2-7, to=3-7]
	\arrow["{l'}", from=3-2, to=2-1]
	\arrow["{r'}"', two heads, from=3-2, to=2-3]
	\arrow["{\widetilde{r'}}", two heads, from=3-7, to=2-5]
	\arrow["\sim"{description, pos=0.9}, shift left=2, draw=none, from=3-7, to=2-5]
	\arrow["{\widetilde{l'}}"', from=3-7, to=2-9]
\end{tikzcd}\]
Here $\widetilde{\alpha}$ is the natural transformation represented by the map $(\widetilde{P}\times_A \widetilde{Q})_0 \to B_1$ given by $\widetilde{\alpha}(p,q) = b$, for the unique $b$ such that $q = b\cdot\alpha(p)$. Such a $b$ exists, is unique, and varies smoothly by the fact that $\alpha$ is an equivariant isomorphism and principality of the $B$ action on $P$ and $Q$. Naturality can be easily checked by the fact that $\widetilde\alpha(b\cdot p,b'\cdot q)
  = b'\widetilde\alpha(p,q)b^{-1}$.

After checking that compositions and units are preserved up to a canonical equivalence (i.e. an invertible 2-morphism), by the biaction groupoid construction, one gets the following result.

\begin{proposition}[{\cite[Thm. 7.18]{MeyerZhu2015}}]\label{prop:Bun-to-Ana-LieGpd}
    The biaction groupoid construction defines a pseudofunctor that is an equivalence of bicategories
    \begin{equation*}
        \widetilde{\,\cdot\,} : \mathsf{Bun(Lie Gpd)} \longrightarrow \mathsf{Ana(LieGpd)}.
    \end{equation*}
\end{proposition}

\subsection{A double bicategory of anafunctors of Lie 2-groupoids}

One can consider more general kinds of 2-cells between anafunctors. For example, by considering square 2-cells, one can obtain a \textit{double bicategory} as defined in \cite[Def. 3.1.1]{Morton2009}. Hence, we explain how to construct the double bicategory $\mathsf{2}\Ana(\Liengpd)$ from the bicategory $\Ana(\Liengpd)$. We picture this in Figure \ref{fig:double-bicat-anafunctors}.

\begin{definition}\label{def:1-1-2-1-1-2-homs-anafunctors}
    A \textbf{double anafunctor} of Lie $n$-groupoids is a commutative diagram 
	\[\begin{tikzcd}[ampersand replacement=\&,cramped, sep=scriptsize]
	Z \& D \& {Y} \\
	B \& \Omega \& C \\
	X \& A \& W
	\arrow[from=1-2, to=1-1]
	\arrow["\sim", two heads, from=1-2, to=1-3]
	\arrow[from=2-1, to=1-1]
	\arrow["\sim", two heads, from=2-1, to=3-1]
	\arrow[from=2-2, to=1-2]
	\arrow[from=2-2, to=2-1]
	\arrow["\sim", two heads, from=2-2, to=2-3]
	\arrow["\sim", two heads, from=2-2, to=3-2]
	\arrow[from=2-3, to=1-3]
	\arrow["\sim", two heads, from=2-3, to=3-3]
	\arrow[from=3-2, to=3-1]
	\arrow["\sim", two heads, from=3-2, to=3-3]
	\end{tikzcd}\]
    where each of the spaces is a Lie $n$-groupoid and each span is an anafunctor. Note that we orient double anafunctors with their horizontal sides pointing left and vertical sides pointing up, following the conventions in \cite{MehtaTang2011, AGJ2024}.

    A \textbf{(2,1)-morphism} of double anafunctors is a pair of double anafunctors with the same vertical sides, and whose rows are each connected by a 2-morphism of anafunctors in $\mathsf{Ana(Lie_n Gpd)}$ as in Figure \ref{fig:2-1-1-2-morph-ana}. These three $2$-morphisms are required to be compatible with the vertical structure maps of the double anafunctors, by making the diagram commute.

    A \textbf{(1,2)-morphism} of double anafunctors is a pair of double anafunctors with the same horizontal sides, and whose columns are each connected by a 2-morphism of anafunctors in $\mathsf{Ana(Lie_n Gpd)}$ as in Figure \ref{fig:2-1-1-2-morph-ana}, with a compatibility condition analogous to the previous one. 
\end{definition}

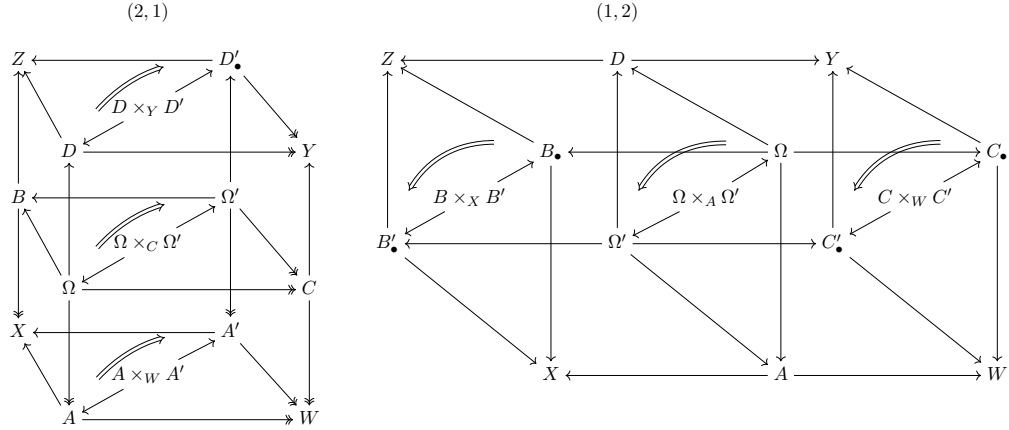
\begin{figure}[h]
    \centering
\begin{adjustbox}{width=0.8\textwidth}
\begin{tikzcd}[ampersand replacement=\&,cramped,sep=small]
	\&\& {(2,1)} \&\&\&\&\&\&\&\& {(1,2)} \&\&\&\&\& \\
	Z \&\&\& {D'_\bullet} \&\&\&\& Z \&\&\& D \&\&\& Y \\
	\&\& {D \times_Y D'} \\
	\& D \&\&\&\& Y \&\&\&\& {B_\bullet} \&\&\& \Omega \&\&\& {C_\bullet} \\
	B \&\&\& {\Omega'} \&\&\&\&\& {B\times_X B'} \&\&\& {\Omega\times_A \Omega'} \&\&\& {C \times_W C'} \\
	\&\& {\Omega \times_C\Omega'} \&\&\&\&\& {B'_\bullet} \&\&\& {\Omega'} \&\&\& {C'_\bullet} \\
	\& \Omega \&\&\&\& C \\
	X \&\&\& {A'} \\
	\&\& {A\times_W A'} \&\&\&\&\&\&\& X \&\&\& A \&\&\& W \\
	\& A \&\&\&\& W
	\arrow[from=2-4, to=2-1]
	\arrow[two heads, from=2-4, to=4-6]
	\arrow[from=2-11, to=2-8]
	\arrow[from=2-11, to=2-14]
	\arrow[from=3-3, to=2-4]
	\arrow[from=3-3, to=4-2]
	\arrow[from=4-2, to=2-1]
	\arrow[curve={height=-22pt}, between={0.2}{0.7}, Rightarrow, nfold, from=4-2, to=2-4]
	\arrow[two heads, from=4-2, to=4-6]
	\arrow[from=4-10, to=2-8]
	\arrow[curve={height=36pt}, between={0.2}{0.8}, Rightarrow, nfold, from=4-10, to=6-8]
	\arrow[from=4-10, to=9-10]
	\arrow[from=4-13, to=2-11]
	\arrow[from=4-13, to=4-10]
	\arrow[from=4-13, to=4-16]
	\arrow[curve={height=36pt}, between={0.2}{0.8}, Rightarrow, nfold, from=4-13, to=6-11]
	\arrow[from=4-13, to=9-13]
	\arrow[from=4-16, to=2-14]
	\arrow[curve={height=36pt}, between={0.2}{0.8}, Rightarrow, nfold, from=4-16, to=6-14]
	\arrow[from=4-16, to=9-16]
	\arrow[from=5-1, to=2-1]
	\arrow[two heads, from=5-1, to=8-1]
	\arrow[from=5-4, to=2-4]
	\arrow[from=5-4, to=5-1]
	\arrow[two heads, from=5-4, to=7-6]
	\arrow[two heads, from=5-4, to=8-4]
	\arrow[from=5-9, to=4-10]
	\arrow[from=5-9, to=6-8]
	\arrow[from=5-12, to=4-13]
	\arrow[from=5-12, to=6-11]
	\arrow[from=5-15, to=4-16]
	\arrow[from=5-15, to=6-14]
	\arrow[from=6-3, to=5-4]
	\arrow[from=6-3, to=7-2]
	\arrow[from=6-8, to=2-8]
	\arrow[from=6-8, to=9-10]
	\arrow[from=6-11, to=2-11]
	\arrow[from=6-11, to=6-8]
	\arrow[from=6-11, to=6-14]
	\arrow[from=6-11, to=9-13]
	\arrow[from=6-14, to=2-14]
	\arrow[from=6-14, to=9-16]
	\arrow[from=7-2, to=4-2]
	\arrow[from=7-2, to=5-1]
	\arrow[curve={height=-22pt}, between={0.2}{0.7}, Rightarrow, nfold, from=7-2, to=5-4]
	\arrow[two heads, from=7-2, to=7-6]
	\arrow[two heads, from=7-2, to=10-2]
	\arrow[from=7-6, to=4-6]
	\arrow[two heads, from=7-6, to=10-6]
	\arrow[from=8-4, to=8-1]
	\arrow[two heads, from=8-4, to=10-6]
	\arrow[from=9-3, to=8-4]
	\arrow[from=9-3, to=10-2]
	\arrow[from=9-13, to=9-10]
	\arrow[from=9-13, to=9-16]
	\arrow[from=10-2, to=8-1]
	\arrow[curve={height=-22pt}, between={0.2}{0.7}, Rightarrow, nfold, from=10-2, to=8-4]
	\arrow[two heads, from=10-2, to=10-6]
\end{tikzcd}
\end{adjustbox}
\caption{(2,1) and (1,2)-morphisms of double anafunctors.}
\label{fig:2-1-1-2-morph-ana}
\end{figure}

(1,1)-morphisms can be composed vertically and horizontally by composing the vertical or horizontal spans, respectively. Both of these compositions are defined again up to canonical isomorphism, as they are defined in terms of pullbacks. More precisely, the horizontal composition is defined up to a canonical $(2,1)$-isomorphism, and the vertical one up to a canonical $(1,2)$-isomorphism. 
Similarly, $(2,1)$ and $(1,2)$-morphisms can be composed in three different directions by using the compositions of $(1,1)$-morphisms or $2$-morphisms. 
This double bicategory structure also includes 4-dimensional cells that are pictured as the $(2,2)$-cells in Figure \ref{fig:double-bicat-anafunctors}.

\begin{definition}
    The double bicategory of anafunctors of Lie $n$-groupoids $\mathsf{2Ana(Lie_n Gpd)}$ consists of the following:
    \begin{itemize}
		\item Its objects are Lie $n$-groupoids;
		\item Its horizontal ($(1,0)$-) morphisms and vertical ($(0,1)$-) morphisms are anafunctors. Morphisms between these (i.e. $(2,0)$-morphisms and $(0,2)$-morphisms, respectively) are the 2-morphisms of anafunctors in $\mathsf{Ana(Lie_n Gpd)}$.
		\item Its $(1,1)$-morphisms are the double anafunctors, and $(2,1)$- and $(1,2)$-morphisms are as in Definition \ref{def:1-1-2-1-1-2-homs-anafunctors}.
        \item $(2,2)$-morphisms follow the schematic in Figure \ref{fig:double-bicat-anafunctors}.
	\end{itemize}
\end{definition}

\begin{figure}[!h]
    \centering
    \includegraphics[width=0.7\linewidth]{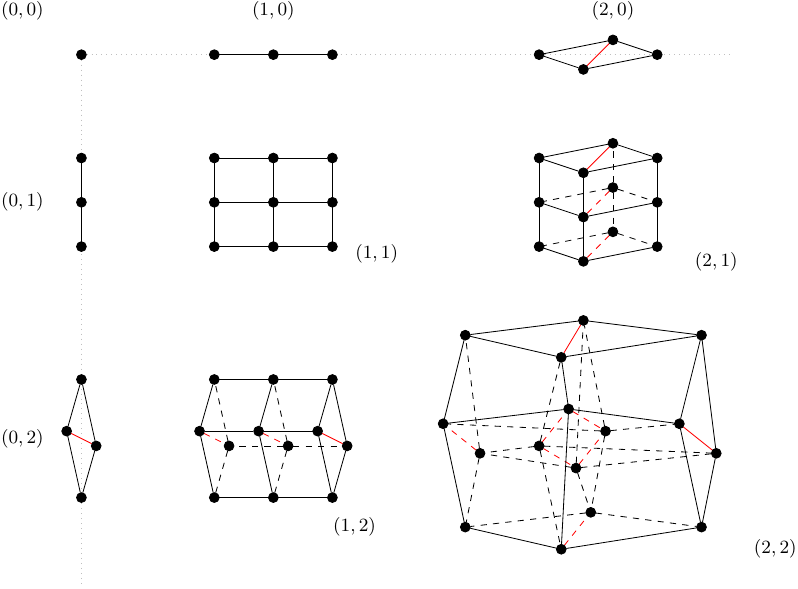}
    \caption{A schematic of the double bicategory of anafunctors $\mathsf{2Ana(Lie_n Gpd)}$. The red lines indicate 2-morphisms of anafunctors. The dashed lines are only meant to improve readability.}
    \label{fig:double-bicat-anafunctors}
\end{figure}

\section{Double Lie groupoids}\label{sec:double stuff}

We begin by recalling the main definitions in \cite[Sec. 2]{AGJ2024}, which allow us to construct a double bicategory $\mathsf{Bun}(\mathsf{LieGpd^2})$ of double Lie groupoids with horizontal and vertical bibundles as generalized morphisms.  
This double bicategory is a generalization of the bicategory of Lie groupoids and HS bibundles $\mathsf{Bun}(\mathsf{LieGpd})$, and it is obtained by a process of localization by bibundles of the category of double groupoids considered as internal Lie groupoids in the category of Lie groupoids first in the horizontal and then in the vertical direction, thus giving two directions of morphisms. This characteristic of double groupoids which makes one able to view them equivalently as internal groupoids in one or the other direction gives rise  to a concept of vertical/horizontal duality which we make explicit in Remark \ref{rem:vert-hor-duality}. We then show that after restricting to full double groupoids, the biaction groupoid construction maps $\mathsf{Bun(LieGpd^2_{full})}$ to a suitable double bicategory of anafunctors of double groupoids $\Ana(\mathsf{LieGpd^2_{full}})$.

\begin{definition}
    A \textbf{double groupoid} $G$ is a groupoid object in the category of groupoids. We depict this by a diagram
    \begin{equation}\label{diag:double-gpd}
    \begin{tikzcd}[ampersand replacement=\&, column sep=scriptsize]
	{G_{11}} \& {G_{01}} \\
	{G_{10}} \& {G_{00},}
	\arrow["{t_{\bullet 1}}"', shift right, from=1-1, to=1-2]
	\arrow["{s_{\bullet 1}}", shift left, from=1-1, to=1-2]
	\arrow["{s_{1\bullet}}"', shift right, from=1-1, to=2-1]
	\arrow["{t_{1\bullet}}", shift left, from=1-1, to=2-1]
	\arrow["s"', shift right, from=1-2, to=2-2]
	\arrow["t", shift left, from=1-2, to=2-2]
	\arrow["t"', shift right, from=2-1, to=2-2]
	\arrow["s", shift left, from=2-1, to=2-2]
    \end{tikzcd}\end{equation}
    where each side represents a groupoid structure and all structure maps of each groupoid are groupoid morphisms with respect to the other direction.
    \end{definition}
    
    We picture the elements of a double groupoid as in Figure \ref{fig:double_groupoid}. In this convention, source and target maps follow the direction of the arrows: the vertical source $s_{1\bullet}$ of a square is its bottom side, while the horizontal source $s_{\bullet 1}$ is its right side.
    By continuing the diagram \eqref{diag:double-gpd} using the nerve construction at each row and column, one obtains a bisimplicial set where each row and each column is a 1-groupoid. We identify the bisimplicial nerve with the double groupoid itself. We denote the groupoid whose nerve is the $k$-th row of the bisimplicial nerve of $G$ by $G_{\bullet k}$, and the groupoid whose nerve is the $k$-th column of the bisimplicial nerve of $G$ by $G_{k \bullet }$. In particular, $G_{\bullet 0}= G_{10}\rightrightarrows G_{00}$ and $G_{0\bullet}= G_{01}\rightrightarrows G_{00}$. These are called side groupoids of $G$. We denote the source and target of the side groupoids simply by $s$ and $t$ since it will be clear from context which one we are referring to. In addition, there are two compatible multiplications on elements of $G_{11}$, which are represented by squares: the vertical one of $G_{1\bullet}$ which we denote by $\vJoin$ and the horizontal one of $G_{\bullet 1}$ which we denote by $\Join$. We denote the source and target of these latter groupoids by $s_{1 \bullet}, t_{1\bullet}$ and $s_{\bullet 1}, t_{\bullet 1}$, respectively. 

\begin{definition}
    A \textbf{double Lie groupoid} is a double groupoid where each of the sides $G_{0 \bullet}$, $G_{\bullet 0}$, $G_{\bullet 1}$, $G_{1 \bullet}$, is a Lie groupoid. Notably, this implies that the double source map
    \begin{equation*}
        (s_{1\bullet}, s_{\bullet 1 }): G_{11} \longrightarrow G_{10} \times_{s, G_{00}, s} G_{01}
    \end{equation*}
    is a submersion. We say a double Lie groupoid is \textbf{full} if this map is also surjective.
\end{definition}
The fullness assumption holds for many examples, especially those coming from algebraic topology (see \cite{Cegarra2025}). However, in Lie theory, the double source map often fails to be surjective, and such is the situation for many of the examples introduced in \cite{LuWeinstein1989}. It is important to emphasize that all of our main results which involve generalized morphisms of double Lie groupoids, use the fullness assumption.

\begin{remark}
    The double source map is a surjective submersion if and only if any of the maps $(s_{\bullet 1}, t_{1 \bullet})$, $(t_{\bullet 1}, s_{1 \bullet})$, and $(t_{\bullet 1}, t_{1 \bullet})$ are surjective submersions (see Figure \ref{fig:fullness-double-groupoid}). This can be seen by composing with the vertical and horizontal inversions of $G_{11}$ as appropriate. 
\end{remark}

\begin{figure}[h]
        \centering
        \includegraphics[width=0.6\linewidth]{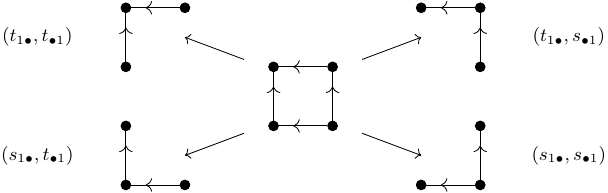}
        \caption{A double Lie groupoid is full if at least one of these maps is a surjective submersion. If this holds, then all of these maps are surjective submersions.}
        \label{fig:fullness-double-groupoid}
    \end{figure}

\begin{definition}
    A \textbf{right horizontal action} of a double Lie groupoid $G$ on a Lie groupoid morphism $r: A \to G_{0 \bullet}$ called the \textbf{moment map} is a Lie groupoid morphism whose components are groupoid actions which is called the \textbf{action morphism} 
    \[\begin{tikzcd}[ampersand replacement=\&]
	{A_1\times_{r_1,G_{01},t_{\bullet 1}} G_{11}} \& {A_1} \\
	{A_0 \times_{r_0, G_{00}, t} G_{10}} \& {A_0,}
	\arrow[from=1-1, to=1-2]
	\arrow[shift left, from=1-1, to=2-1]
	\arrow[shift right, from=1-1, to=2-1]
	\arrow[shift left, from=1-2, to=2-2]
	\arrow[shift right, from=1-2, to=2-2]
	\arrow[from=2-1, to=2-2]
    \end{tikzcd}\]
    where the groupoid structure of the pullback on the left is the one inherited from the product.

    Analogously, a \textbf{left horizontal action} of a double Lie groupoid $J$ on a moment map $l: A \to J_{0 \bullet}$ is a Lie groupoid morphism whose components are groupoid actions
    \begin{equation}\label{diag:left-hor-action}
    \begin{tikzcd}[ampersand replacement=\&]
	{J_{11}\times_{s_{\bullet 1},J_{01},l_1} A_1} \& {A_1} \\
	{J_{10} \times_{s, J_{00}, l_0} A_0} \& {A_0.}
	\arrow[from=1-1, to=1-2]
	\arrow[shift left, from=1-1, to=2-1]
	\arrow[shift right, from=1-1, to=2-1]
	\arrow[shift left, from=1-2, to=2-2]
	\arrow[shift right, from=1-2, to=2-2]
	\arrow[from=2-1, to=2-2]
    \end{tikzcd}
    \end{equation}
\end{definition}

This notion of action appears in \cite[Def. 1.5]{BrownMackenzie1992} and \cite{Stefanini2008}, where this is called a \textit{morphic action}.

\begin{definition}
    Let $p: A \to Q$ be a Lie groupoid morphism and a levelwise surjective submersion. 
    A (left) horizontal action is \textbf{principal} if the double groupoid morphism 
    \begin{equation}\label{diag:left-principal}
    \begin{tikzcd}[ampersand replacement=\&,column sep=small,row sep=scriptsize]
	{J_{11}\times_{s_{\bullet 1},J_{01},l_1} A_1} \& {A_1} \&\& {A_1\times_{p_1,Q_1,p_1} A_1} \& {A_1} \\
	{J_{10} \times_{s, J_{00}, l_0} A_0} \& {A_0} \&\& {A_0\times_{p_0,Q_0,p_0} A_0} \& {A_0}
	\arrow[shift right, from=1-1, to=1-2]
	\arrow[shift left, from=1-1, to=1-2]
	\arrow[shift left, from=1-1, to=2-1]
	\arrow[shift right, from=1-1, to=2-1]
	\arrow[""{name=0, anchor=center, inner sep=0}, shift left, from=1-2, to=2-2]
	\arrow[shift right, from=1-2, to=2-2]
	\arrow[shift left, from=1-4, to=1-5]
	\arrow[shift right, from=1-4, to=1-5]
	\arrow[shift left, from=1-4, to=2-4]
	\arrow[""{name=1, anchor=center, inner sep=0}, shift right, from=1-4, to=2-4]
	\arrow[shift left, from=1-5, to=2-5]
	\arrow[shift right, from=1-5, to=2-5]
	\arrow[shift left, from=2-1, to=2-2]
	\arrow[shift right, from=2-1, to=2-2]
	\arrow[shift left, from=2-4, to=2-5]
	\arrow[shift right, from=2-4, to=2-5]
	\arrow["\cong"{pos=0.4}, between={0.2}{0.6}, from=0, to=1]
    \end{tikzcd}    
    \end{equation}
    given by the action map and the projection is an isomorphism between the action double groupoid and the submersion double groupoid relative to $p:A \to Q$. We picture this in Figure \ref{fig:left-hor-action-principality}. 
\end{definition}

\begin{figure}[h]
    \centering
    \includegraphics[width=0.9\linewidth]{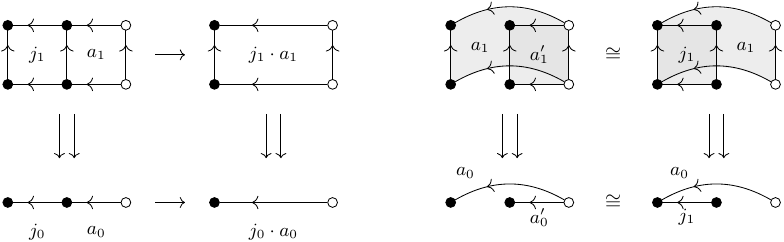}
    \caption{A left horizontal action and its principality condition.}
    \label{fig:left-hor-action-principality}
\end{figure}

The analogous notions of left \textbf{vertical action} of a double groupoid $J$ on a Lie groupoid morphism $B \to J_{\bullet 0}$, and of principality of such an action can be defined similarly by swapping the groupoid $J_{1\bullet}$ for the groupoid $J_{\bullet 1}$ in \eqref{diag:left-hor-action} and \eqref{diag:left-principal}, respectively. The same goes for the notion of right vertical action. 

\begin{remark}[Vertical/horizontal duality]\label{rem:vert-hor-duality}
    This is a general feature of the theory: because our definitions of double groupoids and horizontal and vertical morphisms are symmetric under permutation of the horizontal and vertical directions with each other, many of the results we prove for the horizontal case have a dual vertical counterpart. The proofs thereof are obtained the same way as in the horizontal case, by making the necessary substitutions. 
\end{remark}

We now make use of the notion of horizontal and vertical actions to construct the double bicategory $\mathsf{Bun}(\mathsf{LieGpd^2})$. This is analogous to $\mathsf{Bun}(\mathsf{LieGpd})$, but it has two distinct directions of bibundles. 

\begin{definition}\label{def:hormor}
    A \textbf{horizontal generalized morphism} (or \textbf{(1,0)-morphism}) $A: G \hormor J$ of double groupoids $G,J$ is a Lie groupoid $A$ formed by a pair of compatible HS bibundles $A_1$ and $A_0$, between $G_{\bullet 1}$ and $J_{\bullet 1}$ and between $G_{\bullet 0}$ and $J_{\bullet 0}$, respectively.
    \end{definition}

    For the sake of clarity, we spell this out. A (1,0)-morphism consists of:
    \begin{itemize}
        \item a Lie groupoid $A$ with two moment maps (groupoid morphisms) 
        $$l: A \to J_{0\bullet}, \qquad r: A \to G_{0\bullet},$$ 
        \item a right horizontal $G$-action on $r: A \to G_{0\bullet}$,
        \item a left horizontal $J$-action on $l:A \to J_{0\bullet}$,
    \end{itemize}
    such that:
    \begin{itemize}
        \item The actions commute. 
        \item The left action is principal with quotient map $r:A \to G_{0\bullet}$ (which is thus a levelwise surjective submersion).
    \end{itemize}
    
    We represent this data in the diagram
\begin{equation}\label{diag:hor-gen-mor}
\begin{tikzcd}[ampersand replacement=\&,cramped,sep=small]
	{J_{11}} \&\& {J_{01}} \&\& {A_1} \&\& {G_{01}} \&\& {G_{11}} \\
	\\
	{J_{10}} \&\& {J_{00}} \&\& {A_0} \&\& {G_{00}} \&\& {G_{10}}
	\arrow[shift right, from=1-1, to=1-3]
	\arrow[shift left, from=1-1, to=1-3]
	\arrow[shift right, from=1-1, to=3-1]
	\arrow[shift left, from=1-1, to=3-1]
	\arrow[shift left, from=1-3, to=3-3]
	\arrow[shift right, from=1-3, to=3-3]
	\arrow["{l_1}"', from=1-5, to=1-3]
	\arrow["{r_1}", two heads, from=1-5, to=1-7]
	\arrow[shift left, from=1-5, to=3-5]
	\arrow[shift right, from=1-5, to=3-5]
	\arrow[shift left, from=1-7, to=3-7]
	\arrow[shift right, from=1-7, to=3-7]
	\arrow[shift left, from=1-9, to=1-7]
	\arrow[shift right, from=1-9, to=1-7]
	\arrow[shift left, from=1-9, to=3-9]
	\arrow[shift right, from=1-9, to=3-9]
	\arrow[shift right, from=3-1, to=3-3]
	\arrow[shift left, from=3-1, to=3-3]
	\arrow["{l_0}"', from=3-5, to=3-3]
	\arrow["{r_0}", two heads, from=3-5, to=3-7]
	\arrow[shift left, from=3-9, to=3-7]
	\arrow[shift right, from=3-9, to=3-7]
\end{tikzcd}
\end{equation}
and depict the elements of its constituent spaces in Figure \ref{fig:hor-gen-morphism}.

\begin{figure}[h]
    \centering
    \includegraphics[width=0.7\linewidth]{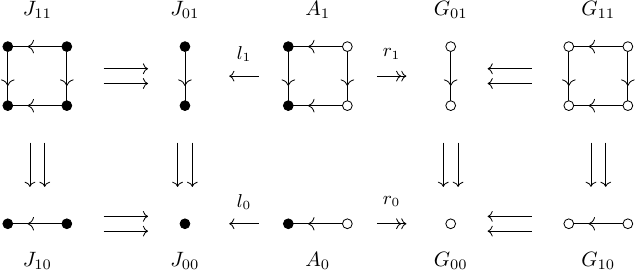}
    \caption{A horizontal generalized morphism}
    \label{fig:hor-gen-morphism}
\end{figure}

Composition of horizontal generalized morphisms is defined by levelwise composition of HS morphisms. Each composition is well defined up to a natural canonical isomorphism, which in this case takes the form of a \textbf{(2,0)-morphism} between (1,0)-morphisms: a biequivariant Lie groupoid morphism between the horizontal bibundles.\footnote{As for HS bibundles of Lie groupoids, biequivariant Lie groupoid morphisms are necessarily isomorphisms.}

\begin{example}
    The notion of $\mathsf{PB}$-groupoid appearing in \cite{CattafiGarmendia2025} is an example of a principal action of a strict Lie 2-groupoid viewed as a double groupoid with one side an identity groupoid. Therefore a $\mathsf{PB}$-groupoid for the strict Lie 2-groupoid $H$ over the groupoid $G \rightrightarrows M$ can be seen as a (1,0)-morphism as follows:
\[\begin{tikzcd}[ampersand replacement=\&,cramped,sep=small]
	{H_{2}} \&\& {H_0} \&\& {P_1} \&\& G \&\& G \\
	\\
	{H_1} \&\& {H_0} \&\& {P_0} \&\& M \&\& M.
	\arrow[shift right, from=1-1, to=1-3]
	\arrow[shift left, from=1-1, to=1-3]
	\arrow[shift right, from=1-1, to=3-1]
	\arrow[shift left, from=1-1, to=3-1]
	\arrow[shift left, no head, from=1-3, to=3-3]
	\arrow[shift right, no head, from=1-3, to=3-3]
	\arrow["{l_1}"', from=1-5, to=1-3]
	\arrow["{r_1}", two heads, from=1-5, to=1-7]
	\arrow[shift left, from=1-5, to=3-5]
	\arrow[shift right, from=1-5, to=3-5]
	\arrow[shift left, from=1-7, to=3-7]
	\arrow[shift right, from=1-7, to=3-7]
	\arrow[shift left, no head, from=1-9, to=1-7]
	\arrow[shift right, no head, from=1-9, to=1-7]
	\arrow[shift left, from=1-9, to=3-9]
	\arrow[shift right, from=1-9, to=3-9]
	\arrow[shift right, from=3-1, to=3-3]
	\arrow[shift left, from=3-1, to=3-3]
	\arrow["{l_0}"', from=3-5, to=3-3]
	\arrow["{r_0}", two heads, from=3-5, to=3-7]
	\arrow[shift left, no head, from=3-9, to=3-7]
	\arrow[shift right, no head, from=3-9, to=3-7]
\end{tikzcd}\]
\end{example}

To fix notation, we write out the corresponding definition of a vertical generalized morphism. 

\begin{definition}
    A \textbf{vertical generalized morphism} (or \textbf{(0,1)-morphism}) $C:G \vermor H$ of double groupoids $G,H$ is a Lie groupoid $C$ formed by a pair of compatible HS bibundles $C_1$ and $C_0$, between $G_{1\bullet}$ and $H_{1\bullet}$ and between $G_{0\bullet}$ and $H_{0\bullet}$, respectively.
\end{definition}

    Therefore, a (0,1)-morphism consists of 
    \begin{itemize}
        \item a Lie groupoid $C$ with two moment maps (groupoid morphisms) 
        $$l: C \to H_{\bullet 0}, \qquad r: C \to G_{\bullet 0},$$
        \item a right vertical $G$-action on $r: C \to G_{\bullet 0}$,
        \item a left vertical $H$-action on $l: C \to H_{\bullet 0}$,
    \end{itemize}
    such that: 
    \begin{itemize}
        \item The actions commute.
        \item The left action is principal with quotient map  $r: C \to G_{\bullet 0}$ (which is thus a levelwise surjective submersion).
    \end{itemize}
    We represent this data analogously to \eqref{diag:hor-gen-mor} as in the vertical sides of \eqref{diag:1-1-morphism}. 

    Vertical generalized morphisms can be composed in an analogous way as horizontal generalized morphisms, and admit similar \textbf{(0,2)-morphisms} between them, consisting of biequivariant Lie groupoid morphisms.

\begin{definition}\label{def:1-1-morphism}
    A \textbf{(1,1)-morphism} $(\Omega,A^{JG},B_J^K,C_G^H,D^{KH})$ of double Lie groupoids consists of the following: \begin{itemize}
        \item Two Lie groupoids $A$ and $D$ which are horizontal generalized morphisms of double Lie groupoids
        \begin{equation*}
            A:G\hormor J, \qquad D:H\hormor K.
        \end{equation*}
        \item Two Lie groupoids $C$ and $B$ which are vertical generalized morphisms of double Lie groupoids 
        \begin{equation*}
            C:G\vermor H, \qquad B:J\vermor K.
        \end{equation*}
        \item A smooth manifold $\Omega$, which is simultaneously a HS morphism of Lie groupoids $C\dashrightarrow B$ and $A\dashrightarrow D$, such that all four actions commute, in the following sense:
\[ (k\cdot_h d)\cdot_v(b\cdot_h\omega)
  =(k\cdot_v b)\cdot_h(d\cdot_v\omega),\qquad k\in K_{11},\quad d\in D_1,\quad b\in B_1,\quad\omega\in\Omega; \] 
  whenever the corresponding actions are defined, and similar equations for the other actions.  
  \end{itemize}
\end{definition}

We represent (1,1)-morphisms as in the following diagram, where the double headed maps are surjective submersions by the principality conditions.
    
\begin{equation}\label{diag:1-1-morphism}
\begin{tikzcd}[ampersand replacement=\&,sep=small]
	{K_{11}} \& {K_{01}} \& {D_1} \& {H_{01}} \& {H_{11}} \\
	{K_{10}} \& {K_{00}} \& {D_0} \& {H_{00}} \& {H_{10}} \\
	{B_1} \& {B_0} \& \Omega \& {C_0} \& {C_1} \\
	{J_{10}} \& {J_{00}} \& {A_0} \& {G_{00}} \& {G_{10}} \\
	{J_{11}} \& {J_{01}} \& {A_1} \& {G_{01}} \& {G_{11}}
	\arrow[shift left, from=1-1, to=1-2]
	\arrow[shift right, from=1-1, to=1-2]
	\arrow[shift left, from=1-1, to=2-1]
	\arrow[shift right, from=1-1, to=2-1]
	\arrow[shift left, from=1-2, to=2-2]
	\arrow[shift right, from=1-2, to=2-2]
	\arrow[from=1-3, to=1-2]
	\arrow[two heads, from=1-3, to=1-4]
	\arrow[shift left, from=1-3, to=2-3]
	\arrow[shift right, from=1-3, to=2-3]
	\arrow[shift left, from=1-4, to=2-4]
	\arrow[shift right, from=1-4, to=2-4]
	\arrow[shift left, from=1-5, to=1-4]
	\arrow[shift right, from=1-5, to=1-4]
	\arrow[shift left, from=1-5, to=2-5]
	\arrow[shift right, from=1-5, to=2-5]
	\arrow[shift left, from=2-1, to=2-2]
	\arrow[shift right, from=2-1, to=2-2]
	\arrow[from=2-3, to=2-2]
	\arrow[two heads, from=2-3, to=2-4]
	\arrow[shift left, from=2-5, to=2-4]
	\arrow[shift right, from=2-5, to=2-4]
	\arrow[from=3-1, to=2-1]
	\arrow[shift right, from=3-1, to=3-2]
	\arrow[shift left, from=3-1, to=3-2]
	\arrow[two heads, from=3-1, to=4-1]
	\arrow[from=3-2, to=2-2]
	\arrow[two heads, from=3-2, to=4-2]
	\arrow[from=3-3, to=2-3]
	\arrow[from=3-3, to=3-2]
	\arrow[two heads, from=3-3, to=3-4]
	\arrow[two heads, from=3-3, to=4-3]
	\arrow[from=3-4, to=2-4]
	\arrow[two heads, from=3-4, to=4-4]
	\arrow[from=3-5, to=2-5]
	\arrow[shift right, from=3-5, to=3-4]
	\arrow[shift left, from=3-5, to=3-4]
	\arrow[two heads, from=3-5, to=4-5]
	\arrow[shift left, from=4-1, to=4-2]
	\arrow[shift right, from=4-1, to=4-2]
	\arrow[from=4-3, to=4-2]
	\arrow[two heads, from=4-3, to=4-4]
	\arrow[shift left, from=4-5, to=4-4]
	\arrow[shift right, from=4-5, to=4-4]
	\arrow[shift right, from=5-1, to=4-1]
	\arrow[shift left, from=5-1, to=4-1]
	\arrow[shift left, from=5-1, to=5-2]
	\arrow[shift right, from=5-1, to=5-2]
	\arrow[shift right, from=5-2, to=4-2]
	\arrow[shift left, from=5-2, to=4-2]
	\arrow[shift left, from=5-3, to=4-3]
	\arrow[shift right, from=5-3, to=4-3]
	\arrow[from=5-3, to=5-2]
	\arrow[two heads, from=5-3, to=5-4]
	\arrow[shift right, from=5-4, to=4-4]
	\arrow[shift left, from=5-4, to=4-4]
	\arrow[shift right, from=5-5, to=4-5]
	\arrow[shift left, from=5-5, to=4-5]
	\arrow[shift left, from=5-5, to=5-4]
	\arrow[shift right, from=5-5, to=5-4]
\end{tikzcd}
\end{equation}
We depict the elements of the spaces forming a (1,1)-morphism as in Figure \ref{fig:1-1-morphism-4-colors}.

\begin{figure}[!h]
    \centering
    \includegraphics[width=0.7\linewidth]{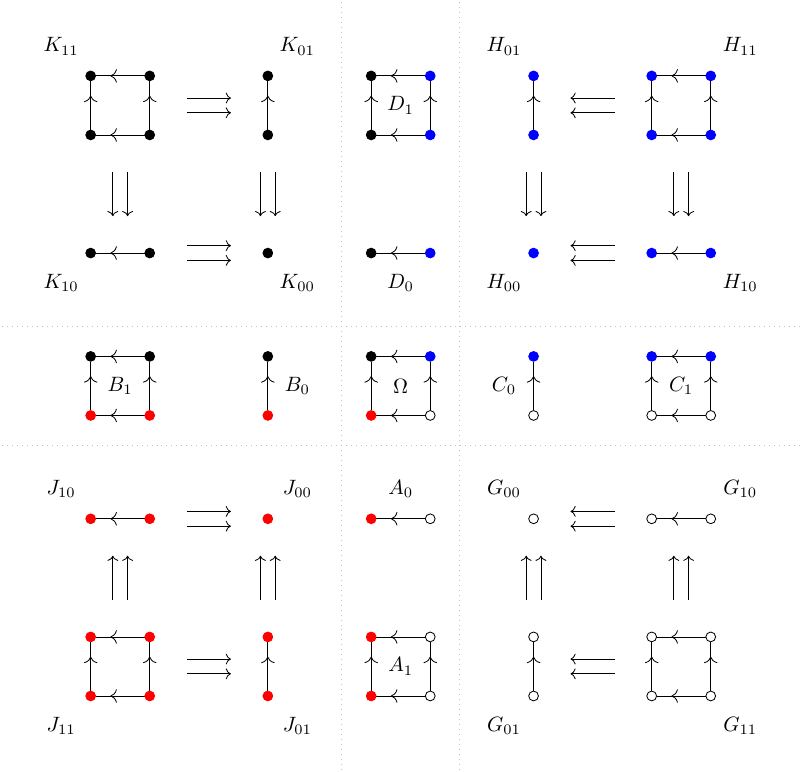}
    \caption{Pictorial representation of the elements of a (1,1)-morphism of double groupoids.}
    \label{fig:1-1-morphism-4-colors}
\end{figure}

At this point, for the sake of having a double bicategory, we can define (1,2)-, (2,1)- and (2,2)-morphisms in the shape of the (1,2)-, (2,1)-, and (2,2)-morphisms in Figure \ref{fig:double-bicat-anafunctors}, respectively, by using biequivariant Lie groupoid morphism between the sides and ``quadriequivariant'' maps in the center accordingly. 

\begin{definition}\label{def:double-bicat-Bun-double-LieGpd}
    The double bicategory $\mathsf{Bun}(\mathsf{LieGpd^2})$ is the double bicategory consisting of double Lie groupoids, (1,0)-, (0,1)-, (1,1)-, (2,0)-, (0,2)-morphisms as defined above and the corresponding (2,1), (1,2) and (2,2)-morphisms as in Figure \ref{fig:double-bicat-anafunctors}, where the red maps are bi- or quadri-equivariant maps accordingly. The full sub-double-bicategory of full double Lie groupoids with the same morphisms will be denoted by $\mathsf{Bun}(\mathsf{LieGpd^2_{full}})$. 
\end{definition}

\subsection{From principal bibundles to anafunctors}

We begin by defining horizontal and vertical hypercovers and anafunctors. Conceptually, these can be thought of as the 1-morphisms of the anafunctor localization of internal Lie groupoids in the category of Lie groupoids, where double groupoids are considered as such in two different directions: alternatively as internal vertical or horizontal Lie groupoids in the category of horizontal or vertical Lie groupoids, respectively. 
Our goal is to arrive at a double bicategory of anafunctors and a biaction groupoid construction for double Lie groupoids, by expanding on \cite[Remark 2.16]{AGJ2024}.

\begin{definition}\label{def:hor-lvlwise-hc}
    Let $f: G \to H$ be a map of double groupoids. We say $f$ is a \textbf{horizontal levelwise hypercover} if the groupoid morphisms 
    \begin{equation*}
    f_{\bullet 0}: 
    \begin{tikzcd}[ampersand replacement=\&,column sep=scriptsize]
	{G_{10}} \& {G_{00}} \\
	{H_{10}} \& {H_{00}}
	\arrow[shift left, from=1-1, to=1-2]
	\arrow[shift right, from=1-1, to=1-2]
	\arrow["{f_{10}}", from=1-1, to=2-1]
	\arrow["{f_{00}}", from=1-2, to=2-2]
	\arrow[shift left, from=2-1, to=2-2]
	\arrow[shift right, from=2-1, to=2-2]
    \end{tikzcd}
    \quad \text{ and } \quad
    f_{\bullet 1}:
    \begin{tikzcd}[ampersand replacement=\&,column sep=scriptsize]
	{G_{11}} \& {G_{01}} \\
	{H_{11}} \& {H_{01}}
	\arrow[shift left, from=1-1, to=1-2]
	\arrow[shift right, from=1-1, to=1-2]
	\arrow["{f_{11}}", from=1-1, to=2-1]
	\arrow["{f_{01}}", from=1-2, to=2-2]
	\arrow[shift left, from=2-1, to=2-2]
	\arrow[shift right, from=2-1, to=2-2]
\end{tikzcd}
    \end{equation*}
    are both hypercovers. 
    Explicitly this means the following conditions hold:
    \begin{equation}\label{eq:hor-lvlwise-hc-conditions}
    \begin{array}{l}
         \Acyc(0)_{\bullet 0}: f_{00}: G_{00} \to H_{00} \text{ is a surjective submersion, } \\
         \Acyc(0)_{\bullet 1}: f_{01}: G_{01} \to H_{01} \text{ is a surjective submersion, }  \\
         \Acyc!(1)_{\bullet 0}: (\partial, f_{10}): G_{10} \overset{\cong}{\to} (G_{00}  \times G_{00})_{H_{00} \times H_{00}} H_{10} \text{ is an isomorphism,}\\
         \Acyc!(1)_{\bullet 1}: (\partial, f_{11}): G_{11} \overset{\cong}{\to} (G_{01}  \times G_{01})_{H_{01} \times H_{01}} H_{11} \text{ is an isomorphism.}
    \end{array}
    \end{equation}
	A span of double Lie groupoids in which the domain-side leg is a horizontal hypercover is called a \textbf{horizontal anafunctor}. 
\end{definition}

\begin{remark}\label{rem:pullback-hor-hypercover}
    Horizontal hypercovers satisfy the same properties as usual hypercovers in \ref{rem:hypercover-properties}; namely, they contain isomorphisms, they are closed under composition, and given a diagram of double groupoids $A \overset{\sim_h}{\twoheadrightarrow} G \leftarrow B$ where $A \overset{\sim_h}{\twoheadrightarrow} G$ is a horizontal hypercover, the pullback exists and $A\times_G B \overset{\sim_h}{\twoheadrightarrow} B$ is a horizontal hypercover as well. These properties follow by a straightforward verification.
	Because of this fact, composition of horizontal anafunctors is well-defined, associative, and has units up to natural isomorphism. 
\end{remark}

We define \textbf{vertical levelwise hypercovers} and \textbf{vertical anafunctors} by applying the principle of vertical/horizontal duality (Remark \ref{rem:vert-hor-duality}) to Definition \ref{def:hor-lvlwise-hc}. 

\begin{remark}\label{rem:hor-vert-hypercov-properties}
    As for horizontal hypercovers (Remark \ref{rem:pullback-hor-hypercover}), we have that vertical hypercovers are also closed under composition and pullback (which exists).
	Hence, composition of vertical anafunctors is also well-defined, associative, and has units up to natural isomorphism. We also note that, a priori, the composition of a horizontal hypercover with a vertical one has neither of these properties.
\end{remark}

\begin{definition}
	A \textbf{double  anafunctor} of double Lie groupoids is a commutative diagram (i.e. a double span)
\begin{equation*}
\begin{tikzcd}[ampersand replacement=\&,cramped]
	K \& D \& H \\
	B \& \Omega \& C \\
	J \& A \& G
	\arrow[from=1-2, to=1-1]
	\arrow["{\sim_h}", two heads, from=1-2, to=1-3]
	\arrow[from=2-1, to=1-1]
	\arrow["{\sim_v}", two heads, from=2-1, to=3-1]
	\arrow[from=2-2, to=1-2]
	\arrow[from=2-2, to=2-1]
	\arrow["{\sim_h}", two heads, from=2-2, to=2-3]
	\arrow["{\sim_v}", two heads, from=2-2, to=3-2]
	\arrow[from=2-3, to=1-3]
	\arrow["{\sim_v}", two heads, from=2-3, to=3-3]
	\arrow[from=3-2, to=3-1]
	\arrow["{\sim_h}", two heads, from=3-2, to=3-3]
\end{tikzcd}
\end{equation*}
    where each of the spaces is a double Lie groupoid and each horizontal (resp. vertical) span is a horizontal (resp. vertical) anafunctor. 
\end{definition}

By applying the construction of Section \ref{sec:Bun-to-Ana-LieGpd} to horizontal generalized morphisms of double groupoids in a compatible way, we convert them into a span of double groupoids. We construct a double groupoid $\hortilde{A}$ by considering the \textbf{horizontal biaction (double) groupoid} relative to the bibundle actions (see Figure \ref{fig:biaction-double-groupoid}).
I.e. we define
\begin{equation}\label{eqn:hor-double-biaction}
     \begin{array}{ll}
    \hortilde{A} _{11}:= J_{11}\times_{s_{\bullet 1},J_{01},l_1}A_1\times_{r_1,G_{01},t_{\bullet 1}}G_{11}, &\hortilde{A}_{01}:=A_1,\\
    \hortilde{A} _{10}:= J_{10}\times_{s,J_{00},l_0}A_0\times_{r_0,G_{00},t}G_{10}, &\hortilde{A} _{00}:=A_0,
    \end{array}
\end{equation}
with structure maps
\begin{equation*}
    \begin{split}
    s_{\bullet 1}, t_{\bullet 1}: \qquad &J_{11}\times_{s_{\bullet 1},J_{01},l_1}A_1\times_{r_1,G_{01},t_{\bullet 1}}G_{11} \longrightarrow A_1\\
    &\quad s_{\bullet1}(j,a,g) = a \cdot g, \qquad t_{\bullet 1}(j,a,g) = j \cdot a,\\
    s_{1 \bullet}, t_{1 \bullet}:  \qquad &J_{11}\times_{s_{\bullet 1},J_{01},l_1}A_1\times_{r_1,G_{01},t_{\bullet 1}}G_{11} \longrightarrow J_{10}\times_{s,J_{00},l_0}A_0\times_{r_0,G_{00},t}G_{10}\\
    &\quad s_{1 \bullet}(j,a,g) = (s_{1\bullet}j, sa,s_{1 \bullet}g), \qquad t_{1 \bullet}(j,a,g) = (t_{1 \bullet}j, ta, t_{1\bullet} g),\\
    \end{split}
\end{equation*}
and the remaining ones defined in the only possible compatible way. The multiplications, whenever $(j', a', g')$ and $(j,a,g)$ are horizontally, resp. vertically composable, are defined by
\begin{equation*}
    \begin{split}
    &(j',a',g')\Join (j,a,g) = (j'\Join j, j^{-1}\cdot a', g'\Join g) =  (j'\Join j, a\cdot (g')^{-1}, g'\Join g), \\
    &(j',a',g')\vJoin (j,a,g) = (j'\vJoin j, a' a , g' \vJoin g).
    \end{split}
\end{equation*}

\begin{figure}[!h]
    \centering
    \includegraphics[width=\linewidth]{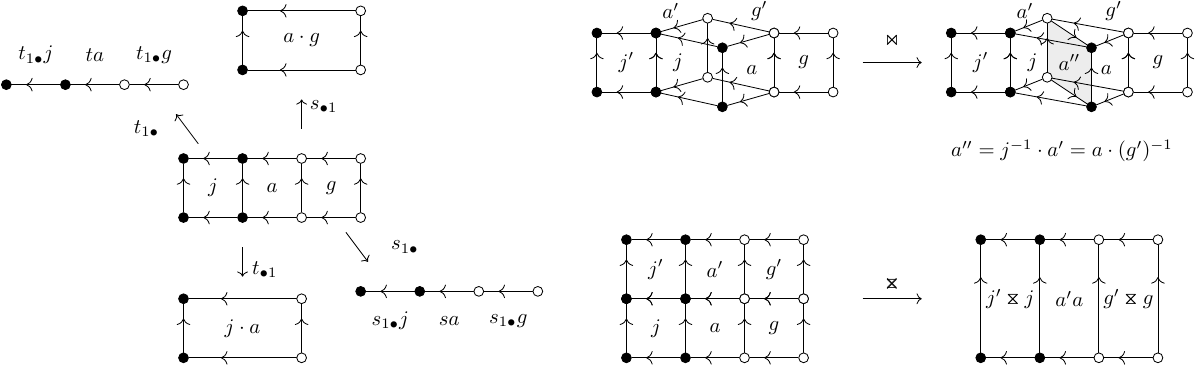}
    \caption{Sources, targets and multiplications of the biaction double groupoid $\hortilde{A}$.}
    \label{fig:biaction-double-groupoid}
\end{figure}

The double groupoid $\hortilde{A}$ comes equipped with natural maps $J\overset{\widetilde{l}}{\leftarrow} \hortilde{A}\overset{\widetilde{r}}{\rightarrow} G$, induced by $l,r$ and the projection maps. As we summarize in the next lemma, this span is a horizontal anafunctor of double Lie groupoids, in the sense that the right map is a horizontal levelwise hypercover. 
    
\begin{lemma}\label{lem:HorGenMorphToHorAna}
    Let $A: G\hormor J$ be a horizontal generalized morphism of double Lie groupoids. Then, $\hortilde{A}$ is a double Lie groupoid and $\hortilde{A}\overset{\sim_h}{\twoheadrightarrow} G$ is a horizontal levelwise hypercover of double Lie groupoids. Additionally, if both $G$ and $J$ are full, then so is $\hortilde{A}$. Thus, the biaction groupoid construction turns horizontal generalized morphisms to horizontal anafunctors of double Lie groupoids. 
\end{lemma}

\begin{proof}
    It is straightforward to check that $\hortilde{A}$ is a double Lie groupoid. 
    Additionally, the fact that $\widetilde{r}$ is a horizontal levelwise hypercover follows from the theory of HS morphisms of Lie groupoids: $\widetilde{r}_{\bullet 0}$ and $\widetilde{r}_{\bullet 1}$ are the hypercovers obtained when applying the biaction groupoid construction to the HS bibundles $A_1$ and $A_0$. Thus $\widetilde{r}$ is a horizontal levelwise hypercover.
    
    Assuming now that $G$ and $J$ are both full, we prove that the double source map $(s_{\bullet 1}, s_{1 \bullet})$ of $\hortilde{A}$ is a surjective submersion.
    To begin with, we observe that
    \begin{equation*}
    \begin{split}
        \hortilde{A}_{01} \times_{s,\hortilde{A}_{00},s} \hortilde{A}_{10} 
        &= (J_{10} \times_{s,J_{00},l_0} A_0 \times_{r_0, G_{00}, t} G_{10}) \times_{s,A_0,s} A_1\\
        &\cong J_{10} \times_{s,J_{00},l_0s} A_1 \times_{r_0s, G_{00}, s} G_{10} 
    \end{split}
    \end{equation*}
    through the diffeomorphism $(j_0,a_0, g_0, \alpha) \mapsto (j_0, \alpha, g_0)$, with inverse $(j_0, \alpha, g_0) \mapsto (j_0,s\alpha\cdot g_0^{-1}, g_0, \alpha)$, since $s\alpha = a_0\cdot g_0$ by definition of the source maps of $\hortilde{A}$.
    Analogously, we have a diffeomorphism
    \begin{equation*}
    \begin{split}
        \hortilde{A}_{11} 
        = J_{11}\times_{s_{\bullet 1},J_{01},l_1}A_1\times_{r_1,G_{01},t_{\bullet 1}}G_{11}&\overset{\cong}{\longleftrightarrow} J_{11}\times_{s_{\bullet 1},J_{01},l_1}A_1\times_{r_1,G_{01},s_{\bullet 1}}G_{11}\\
        (\zeta, \beta, \gamma) &\longmapsto (\zeta, \beta\cdot \gamma, \gamma),\\
        (\zeta, \alpha\cdot \gamma^{-1_h}, \gamma) &\longmapsfrom  (\zeta, \alpha, \gamma),
    \end{split}
    \end{equation*}
    where we denote by $\gamma^{-1_h}$ the horizontal inverse of $\gamma$, and we have $r_1\beta = t_{\bullet 1}\gamma$, $r_1\alpha = s_{\bullet 1} \gamma$. Hence, we can rewrite the double source map as 
    \begin{equation*}
        \begin{split}
            J_{11}\times_{s_{\bullet 1},J_{01},l_1}A_1\times_{r_1,G_{01},s_{\bullet 1}}G_{11} &\longrightarrow J_{10} \times_{s,J_{00},l_0s} A_1 \times_{r_0s, G_{00}, s} G_{10}\\
            (\zeta, \alpha, \gamma) &\longmapsto (s_{1\bullet}\zeta, \alpha, s_{1\bullet}\gamma)
        \end{split}
    \end{equation*}
    This is a surjective submersion because it factors through the top left maps of the following pullback diagrams, which are surjective submersions as they are pullbacks of the double sources of $G$ and $J$, respectively. 
\[\begin{tikzcd}[ampersand replacement=\&,cramped,sep=scriptsize]
	{J_{11}\times_{J_{01}} A_1 \times_{G_{00},s} G_{10}} \& {J_{10}\times_{J_{00}} A_1 \times_{G_{00},s} G_{10}} \& {A_1\times_{G_{00},s}G_{10}} \\
	\\
	{J_{11}} \& {J_{10}\times_{s, J_{00}, s}J_{01}} \& {J_{01}}
	\arrow[two heads, from=1-1, to=1-2]
	\arrow[from=1-1, to=3-1]
	\arrow["\lrcorner"{anchor=center, pos=0.125}, draw=none, from=1-1, to=3-2]
	\arrow[from=1-2, to=1-3]
	\arrow[from=1-2, to=3-2]
	\arrow["\lrcorner"{anchor=center, pos=0.125}, draw=none, from=1-2, to=3-3]
	\arrow["{l_1\circ pr_1}", from=1-3, to=3-3]
	\arrow["{(s_{1\bullet}, s_{\bullet 1})}"', two heads, from=3-1, to=3-2]
	\arrow[from=3-2, to=3-3]
\end{tikzcd}\]

\[\begin{tikzcd}[ampersand replacement=\&,cramped,sep=scriptsize]
	{J_{11}\times_{J_{01}} A_1 \times_{G_{01},s_{\bullet 1}} G_{11}} \& {J_{11}\times_{J_{01}} A_1 \times_{G_{00},s} G_{10}} \& {J_{11}\times_{J_{01}} A_1 } \\
	\\
	{G_{11}} \& {G_{10}\times_{s, G_{00}, s}G_{01}} \& {G_{01}}
	\arrow[two heads, from=1-1, to=1-2]
	\arrow[from=1-1, to=3-1]
	\arrow["\lrcorner"{anchor=center, pos=0.125}, draw=none, from=1-1, to=3-2]
	\arrow[from=1-2, to=1-3]
	\arrow[from=1-2, to=3-2]
	\arrow["\lrcorner"{anchor=center, pos=0.125}, draw=none, from=1-2, to=3-3]
	\arrow["{r_1\circ pr_2}", from=1-3, to=3-3]
	\arrow["{(s_{1\bullet}, s_{\bullet 1})}"', two heads, from=3-1, to=3-2]
	\arrow[from=3-2, to=3-3]
\end{tikzcd}\]
\end{proof}

Analogously, vertical generalized morphisms can be converted into spans of double Lie groupoids as follows. We construct a double Lie groupoid $\vertilde{C}$ by considering the \textbf{vertical biaction (double) groupoid} relative to the bibundle actions.
I.e. we define
\begin{equation}\label{eqn:ver-double-biaction}
    \begin{array}{ll}
    \vertilde{C} _{11}:= H_{11}\times_{s_{1\bullet },H_{10},l_1}C_1\times_{r_1,G_{10},t_{1 \bullet}}G_{11}, &\vertilde{C}_{01}:= H_{01}\times_{s,H_{00},l_0}C_0\times_{r_0,G_{00},t}G_{01},\\
    \vertilde{C} _{10}:= C_1, &\vertilde{C} _{00}:=C_0,
    \end{array}
\end{equation}
with structure maps
\begin{equation*}
    \begin{split}
    s_{\bullet 1}, t_{\bullet 1}: \qquad &H_{11}\times_{s_{1\bullet },H_{10},l_1}C_1\times_{r_1,G_{10},t_{1 \bullet}}G_{11} \to H_{01}\times_{s,H_{00},l_0}C_0\times_{r_0,G_{00},t}G_{01}\\
    &\quad s_{\bullet 1}(h,c,g) = (s_{\bullet 1}h, sc, s_{\bullet 1}g), \qquad t_{\bullet 1}(h,c,g) = (t_{\bullet 1}h, tc, t_{\bullet 1} g),\\
    s_{1 \bullet}, t_{1 \bullet}:  \qquad &H_{11}\times_{s_{1\bullet },H_{10},l_1}C_1\times_{r_1,G_{10},t_{1 \bullet}}G_{11} \to C_1\\
    &\quad s_{1 \bullet}(h,c,g) = c\cdot g, \qquad t_{1 \bullet}(h,c,g) = h\cdot c,\\
    \end{split}
\end{equation*}
and the remaining ones defined in the only possible compatible way. 
A depiction of these structure maps would be a version of Figure \ref{fig:biaction-double-groupoid} with vertical and horizontal structures swapped.
The double groupoid $\vertilde{C}$ comes equipped with natural maps $G\overset{\widetilde{r}}{\leftarrow} \vertilde{C}\overset{\widetilde{l}}{\rightarrow} H$, induced by $l,r$ and the projection maps.

The following result holds, by vertical/horizontal duality, with analogous proofs to its horizontal version above.

\begin{lemma}\label{lem:VerGenMorphToVerAna}
    Let $C: G\vermor H$ be a vertical generalized morphism of double Lie groupoids. Then, $\vertilde{C}$ is a double Lie groupoid and the map $\widetilde{r}: \vertilde{C} \overset{\sim_v}{\twoheadrightarrow} G$ is a vertical levelwise hypercover. Additionally, if both $G$ and $H$ are full, then so is $\vertilde{C}$. Thus, the biaction groupoid construction turns vertical generalized morphisms to vertical anafunctors of double Lie groupoids. 
\end{lemma}

We now introduce an additional condition on levelwise hypercovers of double Lie groupoids, which we require both in order to obtain double  anafunctors from (1,1)-morphisms by the action groupoid construction, and also for horizontal and vertical hypercovers to map to hypercovers of Lie 2-groupoids through $\Wbar$ in Section \ref{sec:double to simplicial}. 

\begin{definition}
    We say that a \textit{horizontal} levelwise hypercover $f: G \overset{\sim_h}{\twoheadrightarrow} H$ is \textbf{full} if either one of the maps
    \begin{equation}\label{eq:hor-extra-condition-map}
    \begin{split}
        (s, f_{01}): \quad &G_{01} \longrightarrow G_{00}\times_{f_{00},H_{00},s}H_{01}\\
        (t, f_{01}): \quad &G_{01} \longrightarrow G_{00}\times_{f_{00},H_{00},t}H_{01}
    \end{split}
    \end{equation}
    is a surjective submersion.\footnote{Notably, if one of these maps is a surjective submersion, the other is one as well, by composing with the vertical inversion.} We say a horizontal anafunctor is full if its domain-side leg is full. 

    We say that a \textit{vertical} levelwise hypercover $f: G \overset{\sim_v}{\twoheadrightarrow} H$ is \textbf{full} if either one of the maps
    \begin{equation}\label{eq:ver-extra-condition-map}
    \begin{split}
        (s, f_{10}): \quad &G_{10} \longrightarrow G_{00}\times_{f_{00},H_{00},s}H_{10}\\
        (t, f_{10}): \quad &G_{10} \longrightarrow G_{00}\times_{f_{00},H_{00},t}H_{10}
    \end{split}
    \end{equation}
    is a surjective submersion.\footnote{Cf. the previous footnote, with the vertical inversion replaced by the horizontal one.} We say a vertical anafunctor is full if its domain-side leg is full. 
\end{definition}

The following two dual lemmas state that the levelwise hypercovers coming from biaction groupoids of bibundles (denoted by $\widetilde{r}$ in our convention) inherit the property of being full from the fullness of the double Lie groupoid acting principally on the bibundle itself. 
Figure \ref{fig:fullness-two-colored} depicts these maps for the $\widetilde{r}$ of a horizontal or vertical bibundle.

\begin{remark}\label{rem:extra-cond-as-Kan-cond}
	In terms of Kan conditions for morphisms, the fact that for a horizontal levelwise hypercover $f: G \overset{\sim_h}{\twoheadrightarrow} H$, $(s, f_{01})$ is a surjective submersion corresponds to the $\Kan(1,1)$ condition for the Lie groupoid morphism $f_{0 \bullet}: G_{0 \bullet} \to H_{0 \bullet}$. The analogous condition for $(t, f_{01})$ is $\Kan(1,0)$ for $f_{0\bullet}$. 
	For a vertical levelwise hypercover $f: G \overset{\sim_v}{\twoheadrightarrow} H$, these are the same level 1 Kan conditions for the Lie groupoid morphism $f_{\bullet 0}$.
\end{remark}

\begin{lemma}\label{lem:fullness-implies-extra-cond-horizontal}
    Let $A: G \hormor J$ be a horizontal morphism of double Lie groupoids. If $J$ is full, then $\widetilde{r}$ is full. That is, the maps 
    \begin{equation*} 
    \begin{split}
        (s, r_1): \hortilde{A}_{01} \longrightarrow \hortilde{A}_{00}\times_{r_0,G_{00},s}G_{01}, \\
        (t, r_1): \hortilde{A}_{01} \longrightarrow \hortilde{A}_{00}\times_{r_0,G_{00},t}G_{01}
    \end{split}
    \end{equation*}
    are surjective submersions. 
\end{lemma}

\begin{lemma}\label{lem:fullness-implies-extra-cond-vertical}
    Let $C: G \vermor H$ be a vertical morphism of double Lie groupoids. If $H$ is full, then $\widetilde{r}$ is full. That is, the maps 
    \begin{equation*}
    \begin{split}
        (s, r_1): \vertilde{C}_{10} \longrightarrow \vertilde{C}_{00}\times_{r_0,G_{00},s}G_{10}, \\
        (t, r_1): \vertilde{C}_{10} \longrightarrow \vertilde{C}_{00}\times_{r_0,G_{00},t}G_{10}
    \end{split}
    \end{equation*}
    are surjective submersions. 
\end{lemma}

\begin{figure}[h]
	\centering
	\includegraphics[width=0.6\textwidth]{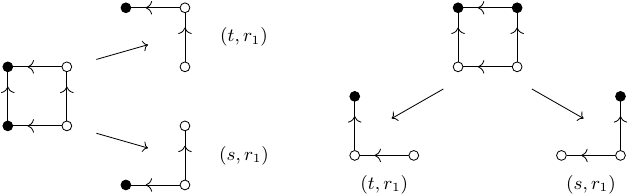}
	\caption{Fullness conditions for $\widetilde{r}$ of a horizontal bibundle on the left, and for $\widetilde{r}$ of a vertical bibundle on the right.}
	\label{fig:fullness-two-colored}
\end{figure}

\begin{proof}[Proof of Lemma \ref{lem:fullness-implies-extra-cond-horizontal}] We show this for $(s, r_1)$, as it is clear from the pictorial representation of the proof in Figure \ref{fig:hor-action-gpd-extra-cond-proof} that the proof for $(t,r_1)$ is analogous. By expanding the definition of $(s, r_1)$ we get
\begin{equation*}
    (s, r_1):  A_1 \longrightarrow A_0\times_{r_0,G_{00},s}G_{01}.
\end{equation*}
  
    \begin{figure}[h]
        \centering
        \includegraphics[width=0.8\linewidth]{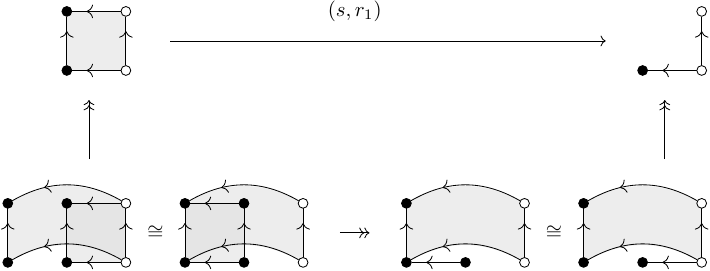}
        \caption{Pictorial proof of Lemma \ref{lem:fullness-implies-extra-cond-horizontal}. Note the use of the principality conditions in Figure \ref{fig:left-hor-action-principality}.}
        \label{fig:hor-action-gpd-extra-cond-proof}
    \end{figure}
    
    We know that $r_1: A_1 \to G_{01}$ is a surjective submersion by Definition \ref{def:hormor}. So, the map 
    \begin{equation*}
       A_0\times_{r_0,G_{00},sr_1} A_1  \longrightarrow A_0\times_{r,G_{00},s}G_{01}
    \end{equation*}
    is a surjective submersion, because of the pullback diagram
    \[\begin{tikzcd}[ampersand replacement=\&,sep=small]
	{A_0\times_{r_0,G_{00},sr_1}A_1} \& {A_0\times_{r_0,G_{00},s}G_{01}} \\
	\\
	{A_1} \& {G_{01}.}
	\arrow[two heads, from=1-1, to=1-2]
	\arrow[from=1-1, to=3-1]
	\arrow["\lrcorner"{anchor=center, pos=0.125}, draw=none, from=1-1, to=3-2]
	\arrow["{pr_2}", from=1-2, to=3-2]
	\arrow["{r_1}"', two heads, from=3-1, to=3-2]
    \end{tikzcd}\]   
    By \eqref{diag:left-principal}, $J_{10} \times_{s, J_{00}, l_0}A_0 \cong A_0 \times_{r_0, G_{00}, r_0} A_0$. So, by the pullback diagram 
    \[\begin{tikzcd}[ampersand replacement=\&,sep=small]
	{J_{10} \times_{s, J_{00}, l_0s}A_1} \& {A_0 \times_{r_0, G_{00}, sr_1} A_1} \\
	\\
	{J_{10} \times_{s, J_{00}, l_0}A_0} \& {A_0 \times_{r_0, G_{00}, r_0} A_0,}
	\arrow["\cong", from=1-1, to=1-2]
	\arrow["{(pr_1, s\circ pr_2)}"', from=1-1, to=3-1]
	\arrow["\lrcorner"{anchor=center, pos=0.125}, draw=none, from=1-1, to=3-2]
	\arrow["{(pr_1, s\circ pr_2)}", from=1-2, to=3-2]
	\arrow["\cong", from=3-1, to=3-2]
    \end{tikzcd}\]
    we have that $J_{10} \times_{s, J_{00}, l_0s}A_1 \cong A_0 \times_{r_0, G_{00}, sr_1} A_1$. 
    At this point, we have another pullback diagram 
    \[\begin{tikzcd}[ampersand replacement=\&,sep=small]
	{J_{11} \times_{s_{\bullet 1}, J_{01}, l_1}A_1} \& {J_{10} \times_{s, J_{00}, l_0s}A_1} \\
	\\
	{J_{11}} \& {J_{10}\times_{s, J_{00},s} J_{01},}
	\arrow[two heads, from=1-1, to=1-2]
	\arrow[from=1-1, to=3-1]
	\arrow["\lrcorner"{anchor=center, pos=0.125}, draw=none, from=1-1, to=3-2]
	\arrow["{(pr_1, l_1\circ pr_2)}", from=1-2, to=3-2]
	\arrow["{(s_{1\bullet}, s_{\bullet 1})}"', two heads, from=3-1, to=3-2]
    \end{tikzcd}\]
    where the bottom map is the double source of $J$, which is a surjective submersion when $J$ is full. So, the map 
    \begin{equation*}
        J_{11} \times_{s_{\bullet 1}, J_{01}, l_1} A_1 \longrightarrow J_{10} \times_{s, J_{00}, l_0s} A_1
    \end{equation*}
    is a surjective submersion. 
    Again by \eqref{diag:left-principal}, $J_{11} \times_{s_{\bullet 1}, J_{01}, l_1} A_1 \cong A_1 \times_{r_1, G_{01}, r_1}A_1$. Hence, the following diagram commutes
    \[\begin{tikzcd}[ampersand replacement=\&,sep=small]
	{A_1 \times_{r_1, G_{01}, r_1}A_1} \&\& {A_0\times_{r_0,G_{00},s}G_{01} } \\
	\& {A_1}
	\arrow["{(s\circ pr_1, r_1 \circ pr_1)}", two heads, from=1-1, to=1-3]
	\arrow["{pr_1}"', two heads, from=1-1, to=2-2]
	\arrow["{(s,r_1)}"', from=2-2, to=1-3]
    \end{tikzcd}\]
    and because both $(s, r_1) \circ pr_1$ and $pr_1$ are surjective submersions ($pr_1$ is indeed a source map in the submersion double groupoid in \eqref{diag:left-principal}), we have that $(s,r_1)$ is a surjective submersion.   
\end{proof}

We now show that a (1,1)-morphism $(\Omega,A^{JG},B_J^K,C_G^H,D^{KH})$ can be rewritten as a double  anafunctor of double Lie groupoids as in \cite[Remark 2.16]{AGJ2024}, by taking
\begin{equation}\label{diag:double-span-total-action-of-square}
\begin{tikzcd}[ampersand replacement=\&,cramped,sep=scriptsize]
	K \& {\hortilde{D}} \& H \\
	{\vertilde{B}} \& {\widetilde{\widetilde{\Omega}}} \& {\vertilde{C}} \\
	J \& {\hortilde{A}} \& {G,}
	\arrow[from=1-2, to=1-1]
	\arrow["{\sim_h}", two heads, from=1-2, to=1-3]
	\arrow[from=2-1, to=1-1]
	\arrow["{\sim_v}", two heads, from=2-1, to=3-1]
	\arrow[from=2-2, to=1-2]
	\arrow[from=2-2, to=2-1]
	\arrow["{\sim_h}", two heads, from=2-2, to=2-3]
	\arrow["{\sim_v}", two heads, from=2-2, to=3-2]
	\arrow[from=2-3, to=1-3]
	\arrow["{\sim_v}", two heads, from=2-3, to=3-3]
	\arrow[from=3-2, to=3-1]
	\arrow["{\sim_h}", two heads, from=3-2, to=3-3]
\end{tikzcd}
\end{equation}
where $\widetilde{\widetilde{\Omega}}$ is the \textbf{total action groupoid} we define below and the double-headed vertical/horizontal arrows are vertical/horizontal levelwise hypercovers, respectively, as we show in Lemma~\ref{lem:doubleheadsRhypercovers}. The construction of the total action groupoid consists of two steps which can be performed interchangeably. First of all, $\Omega$ is a manifold with four groupoid actions: two ``horizontal actions'' of $B$  and $C$ and two ``vertical actions'' of $A$ and $D$. Hence we have the following available constructions:
\begin{itemize}
    \item We can first take the biaction groupoid of $\Omega$ with respect to the ``horizontal actions'', i.e. those of $B$ and $C$. We denote this by $\hortilde{\Omega}$. Together with the horizontal biaction double groupoids of $D$ (with respect to the actions of $K$ and $H$) and $A$ (with respect to the actions of $J$ and $G$) this gives a vertical generalized morphism, as we show in Lemma \ref{lem:hortilde-omega-is-vert-gen-morph}. Then we can make this into a double  anafunctor by applying $\vertilde{\,\cdot\,}$ and we obtain $\vertilde{\hortilde{\Omega}}$.
    \item Alternatively, we can take the biaction groupoid of $\Omega$ with respect to the ``vertical actions'' first, i.e. those of $A$ and $D$. We denote this by $\vertilde{\Omega}$. Together with the vertical biaction double groupoids of $B$ (with respect to the actions of $K$ and $J$) and $C$ (with respect to the actions of $H$ and $G$) this gives a horizontal generalized morphism, as we show in Lemma \ref{lem:vertilde-omega-is-hor-gen-morph}, the dual of Lemma \ref{lem:hortilde-omega-is-vert-gen-morph}. Then we can make this into a double  anafunctor by applying $\hortilde{\,\cdot\,}$ and we obtain $\hortilde{\vertilde{\Omega}}$.
\end{itemize}
A straightforward computation, using the fact that all the actions we consider are compatible with each other by Definition \ref{def:1-1-morphism}, shows that these two constructions are the same (possibly up to isomorphisms rearranging the order of components of each fiber product), so we denote them by $\widetilde{\widetilde{\Omega}}:=\vertilde{\hortilde{\Omega}}=\hortilde{\vertilde{\Omega}}$. This then fits in the double span \eqref{diag:double-span-total-action-of-square}. The following results show that this construction is well-defined.

\begin{lemma}\label{lem:hortilde-omega-is-vert-gen-morph}
    Given a (1,1)-morphism $(\Omega,A^{JG},B_J^K,C_G^H,D^{KH})$, between full double groupoids,
    \begin{equation*}
        \hortilde{D} \leftarrow \hortilde{\Omega} \twoheadrightarrow \hortilde{A}
    \end{equation*}
    is a vertical generalized morphism $\hortilde{A} \vermor \hortilde{D}$. 
\end{lemma}

As anticipated, by vertical/horizontal duality we have the following analogous result.

\begin{lemma}\label{lem:vertilde-omega-is-hor-gen-morph}
    Given a (1,1)-morphism $(\Omega,A^{JG},B_J^K,C_G^H,D^{KH})$ between full double groupoids, 
    \begin{equation*}
        \vertilde{B} \leftarrow \vertilde{\Omega} \twoheadrightarrow \vertilde{C}
    \end{equation*}
    is a horizontal generalized morphism $\vertilde{C} \hormor \vertilde{B}$. 
\end{lemma}

We give an explicit proof of the vertical version.

\begin{proof}[Proof of Lemma \ref{lem:hortilde-omega-is-vert-gen-morph}]
    By applying $\hortilde{\,\cdot\,}$ to the diagram \eqref{diag:1-1-morphism}, we obtain the following diagram
    \[\begin{tikzcd}[ampersand replacement=\&,cramped,sep=small]
	{\hortilde{D}_{\bullet 1}:} \& {K_{11}\times_{K_{01}} D_1 \times_{H_{01}}H_{11}} \& {D_1} \\
	{\hortilde{D}_{\bullet 0}:} \& {K_{10}\times_{K_{00}}D_0\times_{H_{00}}H_{10}} \& {D_0} \\
	{\hortilde{\Omega}_{\bullet}:} \& {B_1\times_{B_0}\Omega \times_{C_0}C_1} \& \Omega \\
	{\hortilde{A}_{\bullet 0}:} \& {J_{10}\times_{J_{00}} A_0 \times_{G_{00}}G_{10}} \& {A_0} \\
	{\hortilde{A}_{\bullet 1}:} \& {J_{11} \times_{J_{01}}A_1 \times_{G_{01}}G_{11}} \& {A_1}
	\arrow[shift left, from=1-2, to=1-3]
	\arrow[shift right, from=1-2, to=1-3]
	\arrow[shift left, from=1-2, to=2-2]
	\arrow[shift right, from=1-2, to=2-2]
	\arrow[shift left, from=1-3, to=2-3]
	\arrow[shift right, from=1-3, to=2-3]
	\arrow[shift left, from=2-2, to=2-3]
	\arrow[shift right, from=2-2, to=2-3]
	\arrow[from=3-2, to=2-2]
	\arrow[shift right, from=3-2, to=3-3]
	\arrow[shift left, from=3-2, to=3-3]
	\arrow[two heads, from=3-2, to=4-2]
	\arrow[from=3-3, to=2-3]
	\arrow[two heads, from=3-3, to=4-3]
	\arrow[shift left, from=4-2, to=4-3]
	\arrow[shift right, from=4-2, to=4-3]
	\arrow[shift right, from=5-2, to=4-2]
	\arrow[shift left, from=5-2, to=4-2]
	\arrow[shift left, from=5-2, to=5-3]
	\arrow[shift right, from=5-2, to=5-3]
	\arrow[shift right, from=5-3, to=4-3]
	\arrow[shift left, from=5-3, to=4-3]
    \end{tikzcd}\]
    For this to be a vertical generalized morphism, we need to show that the left action of $\hortilde{D}$ is principal. Therefore, we begin by showing that the map
    \begin{equation}\label{eq:omega-biaction-proof-eq1}
        B_1\times_{B_0}\Omega \times_{C_0}C_1 \longrightarrow J_{10}\times_{J_{00}} A_0 \times_{G_{00}}G_{10} 
    \end{equation}
    is a surjective submersion. 

	Because $\Omega \twoheadrightarrow A_0$ is a surjective submersion by hypothesis, 
	\begin{equation}\label{eq:omega-biaction-proof-eq2}
		J_{10}\times_{J_{00}} \Omega \times_{G_{00}}G_{10}  \twoheadrightarrow J_{10}\times_{J_{00}} A_0 \times_{G_{00}}G_{10} 
	\end{equation}
	is also one. By Lemma \ref{lem:fullness-implies-extra-cond-vertical}, we get that 
	\begin{equation*}
		B_1 \twoheadrightarrow J_{10} \times_{J_{00}} B_0, \qquad C_1 \twoheadrightarrow C_0 \times_{G_{00}} G_{10}
	\end{equation*}
	are surjective submersions, so 
	\begin{equation}\label{eq:omega-biaction-proof-eq3}
		B_1\times_{B_0}\Omega \times_{C_0}C_1 \twoheadrightarrow J_{10}\times_{J_{00}} B_0 \times_{B_0} \Omega \times_{C_0} C_0 \times_{G_{00}}G_{10} \cong J_{10}\times_{J_{00}} \Omega \times_{G_{00}}G_{10} 
	\end{equation}
	is a surjective submersion. Since \eqref{eq:omega-biaction-proof-eq1} is the composition of \eqref{eq:omega-biaction-proof-eq3} and \eqref{eq:omega-biaction-proof-eq2}, then it is a surjective submersion. 
	Pictorially:
	
	\begin{center}
	\includegraphics[width=\textwidth]{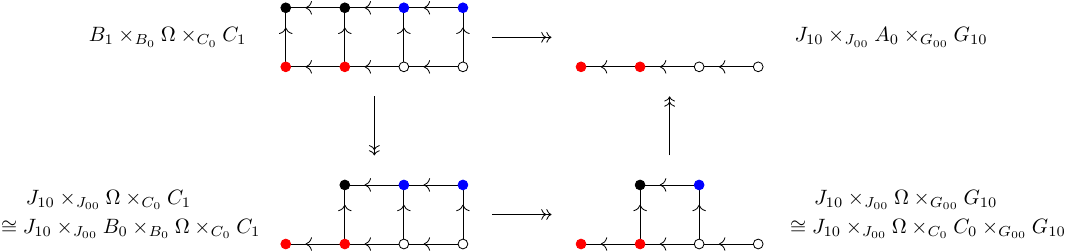}	
	\end{center}
	
	Principality of the action of $\hortilde{D}$ follows immediately by principality of the actions of $K$, $D$ and $H$ on $B$, $\Omega$ and $C$, respectively, as pictured below.

	\begin{center}
	\includegraphics[width=0.9\textwidth]{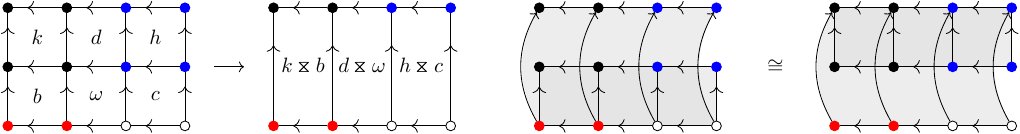}
	\end{center}
    
\end{proof}

\begin{lemma}\label{lem:doubleheadsRhypercovers}
    The vertical/horizontal double-headed maps in \eqref{diag:double-span-total-action-of-square} are vertical/horizontal hypercovers, respectively.
\end{lemma}

\begin{proof}
Given a (1,1)-morphism $(\Omega,A^{JG},B_J^K,C_G^H,D^{KH})$, construct the double Lie groupoids $\hortilde{A}, \hortilde{D}$ and $ \vertilde{B}, \vertilde{C}$. By Lemma~\ref{lem:HorGenMorphToHorAna}, $\hortilde{A}\overset{\sim_h}{\twoheadrightarrow} G$ and $\hortilde{D}\overset{\sim_h}{\twoheadrightarrow}H$ are horizontal levelwise hypercovers. Furthermore, by applying Lemma~\ref{lem:HorGenMorphToHorAna} to the horizontal generalized morphism $\vertilde{B} \leftarrow \vertilde{\Omega} \twoheadrightarrow \vertilde{C}$ from Lemma \ref{lem:vertilde-omega-is-hor-gen-morph}, we obtain that $\widetilde{\widetilde{\Omega}}\overset{\sim_h}{\twoheadrightarrow} \vertilde{C}$ is also a horizontal levelwise hypercover. By using horizontal/vertical duality and Lemmas~\ref{lem:VerGenMorphToVerAna} and \ref{lem:hortilde-omega-is-vert-gen-morph} all the double headed vertical arrows are also vertical levelwise hypercovers. 
\end{proof}

We now discuss the $(2,k)$ and $(k,2)$ morphisms of the double bicategory of anafunctors and apply the biaction construction to obtain these from biequivariant maps of bibundles.

\begin{definition}
    A \textbf{horizontal natural transformation} between morphisms of double groupoids $f, f': G\to H$ is equivalently: 
    \begin{enumerate}
        \item A Lie groupoid morphism $\alpha: G_{0\bullet} \to H_{1\bullet}$ such that for all $g \in G_{1k}$, $k \ge 0$, the naturality condition $\alpha(tg) f(g) = f'(g) \alpha(sg)$ holds. 
        \item A bisimplicial map $G \times \Delta^{1,0} \to H$, where $\Delta^{1,0}$ is the bisimplicial set with rows all equal to $\Delta^1$ and columns the identity groupoids of $\Delta^1_k$ for each $k$.
    \end{enumerate}
    
    Dually, a \textbf{vertical natural transformation} between morphisms of double groupoids $f, f': G\to H$ is equivalently: 
    \begin{enumerate}
        \item A Lie groupoid morphism $\alpha: G_{\bullet 0} \to H_{\bullet 1}$ such that for all $g \in G_{k1}$, $k \ge 0$, the naturality condition $\alpha(tg) f(g) = f'(g) \alpha(sg)$ holds. 
        \item A bisimplicial map $G \times \Delta^{0,1} \to H$, where $\Delta^{0,1}$ is the bisimplicial set with columns all equal to $\Delta^1$ and rows the identity groupoids of $\Delta^1_k$ for each $k$.
    \end{enumerate}
\end{definition}

\begin{definition}
    A \textbf{2-morphism of horizontal anafunctors} is a horizontal natural transformation $\alpha$ between the map $A\times_GA' \to A \to J$ and the map $A \times_G A' \to A' \to J$ as in the following diagram
\begin{equation}\label{diag:2-hom-hor-anafunctor}
\begin{tikzcd}[ampersand replacement=\&,cramped]
	\&\& A \&\& \\
	G \&\& {A\times_GA'} \&\& J \\
	\&\& {A'}
	\arrow["{\sim_h}"'{pos=0.8}, two heads, from=1-3, to=2-1]
	\arrow[from=1-3, to=2-5]
	\arrow["\alpha", curve={height=-44pt}, between={0.2}{0.8}, Rightarrow, nfold, from=1-3, to=3-3]
	\arrow["{_h}"{description, pos=0.7}, shift left=2, curve={height=-44pt}, draw=none, from=1-3, to=3-3]
	\arrow["{\sim_h}", two heads, from=2-3, to=1-3]
	\arrow["{\sim_h}"', two heads, from=2-3, to=3-3]
	\arrow[""{name=0, anchor=center, inner sep=0}, "{\sim_h}"{pos=0.8}, two heads, from=3-3, to=2-1]
	\arrow[from=3-3, to=2-5]
	\arrow["\lrcorner"{anchor=center, pos=0.125, rotate=-90}, draw=none, from=2-3, to=0]
\end{tikzcd}
\end{equation}
where the pullback exists by \ref{rem:pullback-hor-hypercover}. We define \textbf{2-morphisms of vertical anafunctors} dually. 

A \textbf{horizontal (resp. vertical) 2-morphism of double  anafunctors} is a pair of double  anafunctors with the same vertical (resp. horizontal) sides, whose rows (resp. columns) are connected by three 2-morphisms of horizontal (resp. vertical) anafunctors, which are compatible with the vertical (resp. horizontal) maps in the double anafunctors. 
\end{definition}

It is evident that applying the horizontal biaction groupoid construction to a biequivariant Lie groupoid morphism of horizontal bibundles immediately yields a 2-morphism of the corresponding horizontal anafunctors. The same is true for the vertical direction. By functoriality, this construction also maps $(2,1)$- and $(1,2)$-morphisms in $\mathsf{Bun(LieGpd_{full}^2)}$ to horizontal and vertical 2-morphisms of double  anafunctors, respectively. 

Finally, we make a precise definition of the double bicategory of anafunctors of double Lie groupoids. Note that when restricting to full double groupoids we also require  anafunctors to be full, which of course forces all the double anafunctors and higher morphisms to be between full anafunctors. 

\begin{definition}
	The double bicategory $\Ana(\mathsf{LieGpd^2})$ of double Lie groupoids and anafunctors consists of the following:
	\begin{itemize}
		\item Its objects are double Lie groupoids;
		\item Its horizontal ($(1,0)$-) morphisms are the horizontal anafunctors. Morphisms between these (i.e. $(2,0)$-morphisms) are the 2-morphisms of horizontal anafunctors.
		\item Its vertical ($(0,1)$-) morphisms are the vertical anafunctors. Morphisms between these (i.e. $(0,2)$-morphisms) are 2-morphisms of vertical anafunctors.
		\item Its $(1,1)$-morphisms are the double anafunctors. 
        \item $(2,1)$- and $(1,2)$-morphisms are horizontal and vertical 2-morphisms of double anafunctors, respectively.
        \item $(2,2)$-morphisms follow the schematic in Figure \ref{fig:double-bicat-anafunctors}, involving the compatibility of both vertical and horizontal natural transformations.
	\end{itemize}
	We denote the sub-double-bicategory of full double groupoids and full anafunctors by $\Ana(\mathsf{LieGpd^2_{full}})$.
\end{definition}

By the discussion above, we have the following theorem.

\begin{theorem} The biaction groupoid construction defines a morphism of double bicategories 
	\begin{equation*}
		\mathsf{Bun}(\mathsf{LieGpd^2_{full}}) \longrightarrow \Ana(\mathsf{LieGpd^2_{full}}).
	\end{equation*}
\end{theorem}

Once the precise definition of a morphism of double bicategories is established by straightforwardly adapting the definition of a pseudofunctor in \cite[Def. 4.1]{Benabou1967} to the double bicategories of \cite[Def. 3.1.1]{Morton2009}, the rest of the proof follows from the previous discussion and the Lie groupoid case from Section \ref{sec:Bun-to-Ana-LieGpd}.

\section{From double Lie groupoids to Lie 2-groupoids}\label{sec:double to simplicial}

We now recall the main definitions and results from \cite{MehtaTang2011}, which we need to construct a Lie 2-groupoid from a full double groupoid. 

\begin{definition}[\cite{ArtinMazur1966,MehtaTang2011}] Let $\mathfrak{X}= X_{\bullet,\bullet}$ be a bisimplicial manifold. The {\em codiagonal of $\mathfrak{X}$} is the simplicial manifold\footnote{It is assumed in this definition that the following fiber products exist. This is automatic in the case that interests us. } $(\overline{W}\mathfrak{X},d_i,s_i)$:
\begin{align*} &(\overline{W}\mathfrak{X})_n=\{(x_0,x_1,\dots, x_{n-1},x_n)\ |\ x_a\in  X_{n-a,a},\, d_0^hx_a=d^v_{a+1}x_{a+1}, \, \forall \, 0\leq a< n\} \\
& d_i(x_0,x_1,\dots, x_{n-1},x_n)=(d_i^hx_0,d_{i-1}^hx_1,\dots, d_2^hx_{i-2}, d_1^hx_{i-1}, d_i^vx_{i+1},\dots,d_i^vx_{n-1},d_i^vx_n) \\
& s_i(x_0,x_1,\dots, x_{n-1},x_n)=(s_i^hx_0,s_{i-1}^hx_1,\dots,s_{1}^hx_{i-1},s_0^hx_i,s_i^vx_i,\dots,s_i^vx_{n-1}, s_i^vx_n); 
\end{align*} 
where $d_i^h,s_i^h$ are the horizontal face and degeneracy maps (i.e. those for the rows $X_{\bullet, k}$), and $d_i^v,s_i^v$ are the vertical face and degeneracy maps (i.e. those for the columns $X_{k \bullet}$).
\end{definition}

\begin{proposition}[{\cite[Theorem 4.5]{MehtaTang2011}}]
	Let $G$ be a double Lie groupoid. Then, $\Wbar G$ is a local Lie 2-groupoid. Furthermore, if $G$ is a full double Lie groupoid, then $\Wbar G$ is a Lie 2-groupoid.
\end{proposition}

\begin{remark}
    A local Lie groupoid is a simplicial manifold where the Kan conditions hold only locally --- i.e. the horn projections are only required to be submersions (resp. injective étale) instead of surjective submersions (resp. diffeomorphisms). Our constructions can only be extended to certain situations where the starting double Lie groupoids are not full, by carefully adjusting the extra condition of fullness on morphisms accordingly. 
\end{remark}

We now recall the explicit form of $\Wbar G$ for a double Lie groupoid $G$ from \cite[\S 4.2]{MehtaTang2011}. This is specified by the following finite data in the sense of \cite[Prop.-Def. 2.16]{Zhu2009}:
The first three levels are
\begin{equation*}
\begin{split}
    \Wbar G_0 &= G_{00}, \qquad \Wbar G_1 = G_{10}\times_{s,G_{00}, t} G_{01},\\
    \Wbar G_2 &= \Lambda^{2}_1(G_{\bullet 0})\times_{d_0, G_{10}, t_{1\bullet}} G_{11} \times_{s_{\bullet 1}, G_{01}, d_2} \Lambda^2_1(G_{0\bullet}) \\
    &\cong G_{10} \times_{s, G_{00}, t^2} G_{11} \times_{s^2, G_{00}, t} G_{01}, 
\end{split}
\end{equation*}
where $t^2 = t \circ t_{\bullet 1} = t \circ t_{1\bullet}$ and $s^2 = s \circ s_{\bullet 1} = s \circ s_{1\bullet}$. The structure maps between these levels are
\begin{equation*}
\begin{array}{lcl}
    d_0(g_{10}, g_{01}) = s (g_{01}), &\quad    &d_0(g_{10}, g_{11}, g_{01})= (s_{1\bullet}(g_{11}),g_{01}) \\
    d_1(g_{10}, g_{01}) = t (g_{10}), &\quad   &d_1(g_{10}, g_{11}, g_{01})= (g_{10} \cdot t_{1\bullet}(g_{11}), s_{\bullet 1}(g_{11})\cdot g_{01}),   \\
    s_0(g_{00}) = (1_{g_{00}}, 1_{g_{00}}), &\quad   &d_2(g_{10}, g_{11}, g_{01})= (g_{10}, t_{\bullet 1} (g_{11})),  \\
    &\quad        &s_0(g_{10}, g_{01}) = (g_{10}, 1_{g_{01}}, 1_{s(g_{01})}),\\
    &\quad        &s_1(g_{10}, g_{01}) = (1_{t(g_{10})}, 1_{g_{10}}, g_{01}),\\
\end{array}
\end{equation*}
where $1$ is the unit map of either one of the sides of $G$ depending on context. This is summarized in Figure \ref{fig:Wbar}, where we also specify the rest of the finite data, consisting of $\Wbar G_3$ and the face maps between it and $\Wbar G_2$, which give the ternary multiplications. The degeneracy maps are omitted from the figure, but they can easily be recovered by using the simplicial identities. 

\begin{remark}
    The $\Wbar$ functor between bisimplicial sets and simplicial sets is right adjoint to the total decalage functor $\mathrm{Dec}$ (see \cite{CegarraRemedios2005} and references therein). Both constructions extend directly to simplicial and bisimplicial manifolds, so the adjunction still holds. Therefore, the $\Wbar$ functor preserves limits (for example pullbacks) between these two categories, when they exist. 
\end{remark}

\begin{figure}[h]
    \centering
    \includegraphics[width=\linewidth]{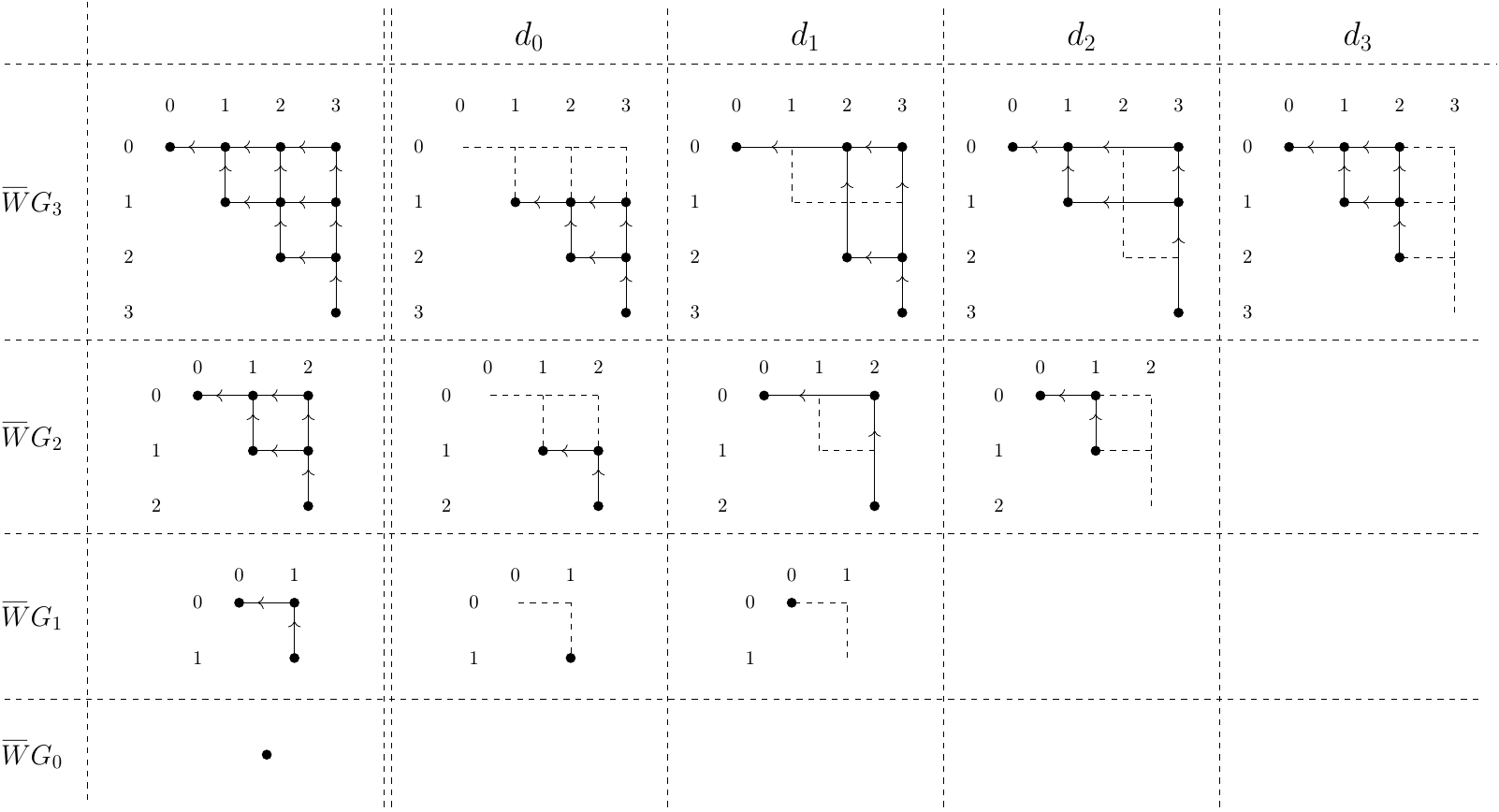}
    \caption{The first 4 levels of $\Wbar$ of a double Lie groupoid $G$ and its face maps. The heuristic goes as follows: $d_i$ of an $n$-simplex is obtained by first deleting the points and arrows in the $i$-th row and column, and then multiplying together any arrows or squares whose shared boundary component was deleted.}
    \label{fig:Wbar}
\end{figure}

In the following pair of horizontal/vertical dual lemmas we show that the extra condition needed for $\Wbar$ to map horizontal/vertical hypercovers to hypercovers is one of the same Kan conditions we found in Lemma \ref{lem:fullness-implies-extra-cond-horizontal}, Lemma \ref{lem:fullness-implies-extra-cond-vertical} and discussed in Remark \ref{rem:extra-cond-as-Kan-cond}. 
We prove the horizontal version of this statement explicitly and argue that the vertical version follows by duality.

\begin{lemma}\label{lem:Wbar-horizontal-hypercover}
    Let $f: G \twoheadrightarrow H$ be a horizontal levelwise hypercover of full double Lie groupoids. Then if $f$ is full, $\Wbar f$ is a hypercover of Lie 2-groupoids.
\end{lemma}

\begin{lemma}\label{lem:Wbar-vertical-hypercover}
    Let $f: G \twoheadrightarrow H$ be a vertical levelwise hypercover of full double Lie groupoids. Then if $f$ is full, $\Wbar f$ is a hypercover of Lie 2-groupoids.
\end{lemma}

\begin{proof}[Proof of Lemma \ref{lem:Wbar-horizontal-hypercover}]
    We show $\Acyc(0)$, $\Acyc(1)$ and $\Acyc!(2)$ for $\Wbar f$, which is
\[\begin{tikzcd}[ampersand replacement=\&,cramped, sep=scriptsize]
	{G_{10}\times_{s, G_{00}, t^2}G_{11}\times_{s^2, G_{00}, t}G_{01}} \& {G_{10} \times_{s, G_{00}, t} G_{01}} \& {G_{00}} \\
	{H_{10}\times_{s, H_{00}, t^2}H_{11}\times_{s^2, H_{00}, t}H_{01}} \& {H_{10}\times_{s, H_{00}, t} H_{01}} \& {H_{00}}
	\arrow[shift right=2, from=1-1, to=1-2]
	\arrow[shift left=2, from=1-1, to=1-2]
	\arrow[from=1-1, to=1-2]
	\arrow["{(f_{10},f_{11},f_{01})}", from=1-1, to=2-1]
	\arrow[shift left, from=1-2, to=1-3]
	\arrow[shift right, from=1-2, to=1-3]
	\arrow["{(f_{10},f_{01})}", from=1-2, to=2-2]
	\arrow["{f_{00}}", from=1-3, to=2-3]
	\arrow[shift right=2, from=2-1, to=2-2]
	\arrow[shift left=2, from=2-1, to=2-2]
	\arrow[from=2-1, to=2-2]
	\arrow[shift left, from=2-2, to=2-3]
	\arrow[shift right, from=2-2, to=2-3]
\end{tikzcd}\]
    Immediately we have that $\Acyc(0)$ follows from the fact that $f_{00}$ is a surjective submersion, which is already true by $\Acyc(0)_{\bullet 0}$.

    For $\Acyc(1)$, we need to show that the map 
    \begin{equation*}
        (\partial, (f_{10}, f_{01})): \Wbar G_1 \to \partial\Delta^1(\Wbar G) \times_{\partial\Delta^1(\Wbar H)} \Wbar H_1
    \end{equation*}
    is a surjective submersion. See Figure \ref{fig:hor-hypercover-acyc1} for a pictorial representation of the following argument. The codomain of $ (\partial, (f_{10}, f_{01}))$ can be written as 
    \begin{equation*}
        \Wbar G_0 \times_{\Wbar H_0, d_1} \Wbar H_1 \times_{d_0, \Wbar H_0} \Wbar G_0 
        = G_{00} \times_{H_{00}, t} H_{10} \times_{s, H_{00}, t} H_{01} \times_{s, H_{00}} G_{00}, 
    \end{equation*}
    where we omit $\Wbar f$ and $f$ in the notation of the fiber products. 
    By using $\Acyc!(1)_{\bullet 0}$ from \eqref{eq:hor-lvlwise-hc-conditions}, we get that
    \begin{equation*}
    \begin{split}
        \Wbar G_1 = G_{10}\times_{G_{00}} G_{01} 
        &\cong  (G_{00} \times_{H_{00},t} H_{10} \times_{s,H_{00}} G_{00}) \times_{\text{pr}_3,G_{00},t} G_{01}\\
        &\cong  G_{00} \times_{H_{00},t} H_{10} \times_{s, H_{00},ft} G_{01}.
    \end{split}
    \end{equation*}
    Then, the following is a pullback diagram:
\[\begin{tikzcd}[ampersand replacement=\&,cramped,row sep=small]
	{\Wbar G_1} \&\& {\partial\Delta^1(\Wbar G) \times_{\partial\Delta^1(\Wbar H)} \Wbar H_1} \& {G_{00} \times_{H_{00},t} H_{10} } \\
	\\
	{G_{01}} \&\& { H_{01} \times_{s, H_{00}} G_{00}} \& {H_{00},}
	\arrow["{(\partial, (f_{10}, f_{01}))}", from=1-1, to=1-3]
	\arrow[from=1-1, to=3-1]
	\arrow["\lrcorner"{anchor=center, pos=0.125}, draw=none, from=1-1, to=3-3]
	\arrow[from=1-3, to=1-4]
	\arrow[from=1-3, to=3-3]
	\arrow["\lrcorner"{anchor=center, pos=0.125}, draw=none, from=1-3, to=3-4]
	\arrow["{s\circ pr_2}", from=1-4, to=3-4]
	\arrow["{(f_{01}, s)}"', two heads, from=3-1, to=3-3]
	\arrow["{t \circ pr_1}"', from=3-3, to=3-4]
\end{tikzcd}\]
    where the map $(f_{01},s)$ is a surjective submersion. Hence, $(\partial, (f_{10}, f_{01}))$ is a surjective submersion. 

    \begin{figure}[h]
        \centering
        \includegraphics[width=0.9\linewidth]{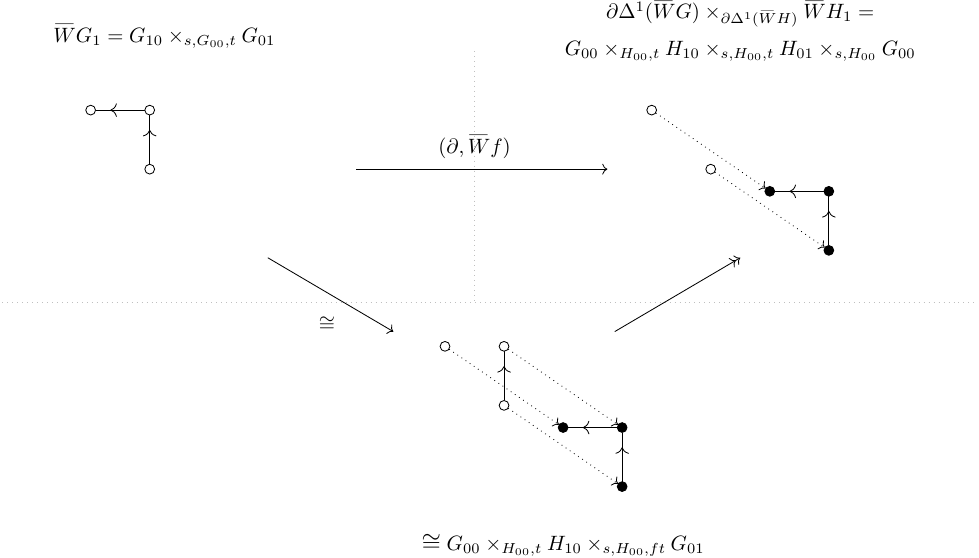}
        \caption{Proving $\Acyc(1)$.}
        \label{fig:hor-hypercover-acyc1}
    \end{figure}
    
    For $\Acyc!(2)$, we need to show that the map 
     \begin{equation*}
        (\partial, (f_{10}, f_{11}, f_{01})): \Wbar G_2 \to \partial\Delta^2(\Wbar G) \times_{\partial\Delta^2(\Wbar H)} \Wbar H_2
    \end{equation*}
    is a diffeomorphism. We depict this proof in Figure \ref{fig:hor-hypercover-acyc2}. 
    
    An element of $\partial\Delta^2(\Wbar G)$ can be depicted as in the second panel of Figure \ref{fig:hor-hypercover-acyc2}. Thus, we can write
    \begin{equation*}
        \partial\Delta^2(\Wbar G) \cong \Lambda^2_0(G_{\bullet 0}) \times_{s d_1,G_{00}, t} G_{01} \times_{s, G_{00}, t} G_{10} \times_{s, G_{00}, td_0} \Lambda^2_2(G_{0\bullet}).
    \end{equation*}
    
    \begin{figure}[h]
        \centering
        \includegraphics[width=\linewidth]{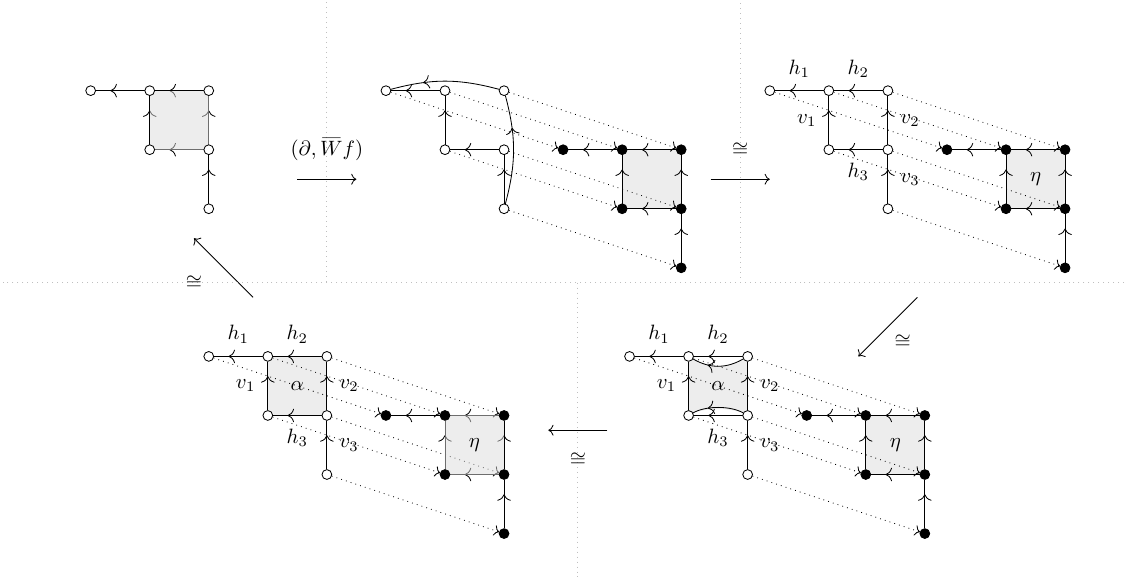}
        \caption{The main steps in the proof of $\Acyc!(2)$.}
        \label{fig:hor-hypercover-acyc2}
    \end{figure}
    
    Because both $G_{\bullet 0}$ and $G_{0\bullet}$ are 1-groupoids, we can fill the two horn spaces above uniquely, so we can depict $\partial\Delta^2(\Wbar G) \times_{\partial\Delta^2(\Wbar H)} \Wbar H_2$ as in the third panel of Figure \ref{fig:hor-hypercover-acyc2}. Then, $\Acyc!(2)$ is equivalent to the existence of a unique square $\alpha$ with 
    \begin{equation*}
        s_{\bullet 1}\alpha = v_2, \quad t_{\bullet 1}\alpha = v_1, \quad s_{1\bullet }\alpha = h_3, \quad t_{1 \bullet}\alpha = h_2,
    \end{equation*}
    and such that $f_{11}(\alpha)=\eta$, for any diagram as in the third panel of Figure \ref{fig:hor-hypercover-acyc2}. 
    
    Because $(f(v_2), f(v_1)) = (s_{\bullet 1}\eta, t_{\bullet 1}\eta)$, $(v_2, v_1, \eta) \in (G_{01} \times G_{01}) \times_{H_{01}\times H_{01}} H_{11}$. 
    By $\Acyc!(1)_{\bullet 1}$, this space is isomorphic to $G_{11}$, so there exists a unique $\alpha$ such that $f_{11}(\alpha)=\eta$, $s_{\bullet1}\alpha = v_2$, and $t_{\bullet 1} \alpha = v_1$.
    Now $(f(s_{1\bullet}\alpha), f(t_{1\bullet}\alpha)) = (f(h_3), f(h_2)) = (s_{1\bullet}\eta, t_{1\bullet}\eta)$, so by $\Acyc!(1)_{\bullet 0}$, $(s_{1\bullet}\alpha, t_{1\bullet}\alpha) = (h_3, h_2)$. We thus produced the required $\alpha$.
\end{proof}

Following Remark \ref{rem:extra-cond-as-Kan-cond}, the fullness condition that requires $(s, f_{01}): G_{01} \rightarrow G_{00}\times_{f_{00},H_{00},s}H_{01}$ to be a surjective submersion is the Kan condition $\Kan(1,1)$ for the groupoid morphism $f_{0\bullet}$. 
    
\begin{example} We produce an example of a horizontal levelwise hypercover of double groupoids that is sent by $\Wbar$ to a simplicial morphism which is not a hypercover. Let $V$ be the disjoint union of the pair groupoid over $\{a,b\}$ and of the unit groupoid on $\{c\}$, and let $H$ be the pair groupoid over $\{0,1\}$. Consider the morphism of groupoids $f:V\to H$ determined by mapping $a$ and $c$ to $0$ and $b$ to $1$. Now take the respective pair double groupoids over $V$ and $H$ with the additional pair groupoid structure defined horizontally. We have that $f$ extends to a morphism of double groupoids which is a horizontal levelwise hypercover, but
the arrow $(0\to 1)\in H$ has no lift to $V$ which starts at $c$; in other words, \eqref{eq:hor-extra-condition-map} is not surjective.          
\[\begin{tikzcd}
	{V^2} & V && {H^2} & H \\
	\\
	{\{a,b,c\}^2} & {\{a,b,c\}} && {\{0,1\}^2} & {\{0,1\}}
	\arrow[shift right, from=1-1, to=1-2]
	\arrow[shift left, from=1-1, to=1-2]
	\arrow[shift right, from=1-1, to=3-1]
	\arrow[shift left, from=1-1, to=3-1]
	\arrow[shift right, from=1-2, to=3-2]
	\arrow[""{name=0, anchor=center, inner sep=0}, shift left, from=1-2, to=3-2]
	\arrow[shift right, from=1-4, to=1-5]
	\arrow[shift left, from=1-4, to=1-5]
	\arrow[""{name=1, anchor=center, inner sep=0}, shift right, from=1-4, to=3-4]
	\arrow[shift left, from=1-4, to=3-4]
	\arrow[shift right, from=1-5, to=3-5]
	\arrow[shift left, from=1-5, to=3-5]
	\arrow[shift right, from=3-1, to=3-2]
	\arrow[shift left, from=3-1, to=3-2]
	\arrow[shift right, from=3-4, to=3-5]
	\arrow[shift left, from=3-4, to=3-5]
	\arrow[between={0.2}{0.8}, from=0, to=1]
\end{tikzcd}\]
Hence, by taking $\Wbar$ of this map, we can see that the result is not a hypercover.
\end{example}
    
As we showed in Lemmas \ref{lem:fullness-implies-extra-cond-horizontal} and \ref{lem:fullness-implies-extra-cond-vertical}, generalized morphisms between full double groupoids produce anafunctors whose levelwise hypercovers are always full. Hence, by combining Lemma \ref{lem:Wbar-horizontal-hypercover} with Lemma \ref{lem:fullness-implies-extra-cond-horizontal}, and Lemma \ref{lem:Wbar-vertical-hypercover} with Lemma \ref{lem:fullness-implies-extra-cond-vertical} we obtain the following theorems.

\begin{theorem}\label{thm:HorGenMorphToAna-Wbar}
Let $G \overset{A}{\dashrightarrow}_h J$ be a horizontal generalized morphism of full double Lie groupoids. Let $ G\overset{\sim_h}{\twoheadleftarrow} \hortilde{A} \to J$ be the associated horizontal anafunctor of double Lie groupoids. Then, $\Wbar ( 
G\overset{\sim}{\twoheadleftarrow} \hortilde{A} \to J)$ is an anafunctor of Lie 2-groupoids.
\end{theorem}

\begin{theorem}\label{thm:VerGenMorphToAna-Wbar}
    Let $G \overset{C}{\dashrightarrow}_v H$ be a vertical generalized morphism of full double Lie groupoids. Let $ G\overset{\sim_v}{\twoheadleftarrow} \vertilde{C} \to H$ be the associated span of double Lie groupoids. Then, $\Wbar (G\overset{\sim}{\twoheadleftarrow} \vertilde{C} \to H )$ is an anafunctor of Lie 2-groupoids. 
\end{theorem}

\begin{theorem}\label{thm:11MorphToSquareAna-Wbar}
    Let $(\Omega,A^{JG},B_J^K,C_G^H,D^{KH})$ be a $(1,1)$-morphism with sides $A:G\hormor J$, $B:J\vermor K$, $C:G\vermor H$, $D:H\hormor K$ as in \eqref{diag:1-1-morphism}. Assume that $G$, $H$, $J$, $K$ are full. Then the diagram 
    \begin{equation}\label{diag:Wbar-double-anafunctor}
	\begin{tikzcd}[ampersand replacement=\&,cramped]
	{\Wbar K} \& {\Wbar (\hortilde{D})} \& {\Wbar H} \\
	{\Wbar (\vertilde{B})} \& {\Wbar (\widetilde{\widetilde{\Omega}})} \& {\Wbar (\vertilde{C})} \\
	{\Wbar J} \& {\Wbar (\hortilde{A})} \& {\Wbar G,}
	\arrow[from=1-2, to=1-1]
	\arrow["\sim", two heads, from=1-2, to=1-3]
	\arrow[from=2-1, to=1-1]
	\arrow["\sim", two heads, from=2-1, to=3-1]
	\arrow[from=2-2, to=1-2]
	\arrow[from=2-2, to=2-1]
	\arrow["\sim", two heads, from=2-2, to=2-3]
	\arrow["\sim", two heads, from=2-2, to=3-2]
	\arrow[from=2-3, to=1-3]
	\arrow["\sim", two heads, from=2-3, to=3-3]
	\arrow[from=3-2, to=3-1]
	\arrow["\sim", two heads, from=3-2, to=3-3]
    \end{tikzcd}
	\end{equation}
    is a $(1,1)$-morphism in $\mathsf{2}\Ana(\mathsf{Lie_2}\Gpd)$. 
\end{theorem}

\begin{lemma}
    Applying $\Wbar$ to a 2-morphism of full horizontal (resp. vertical) anafunctors of full double groupoids gives a 2-morphism of anafunctors of Lie 2-groupoids. In the same way, applying $\Wbar$ to a (2,1)-, (1,2)-, (2,2)-morphism of full double anafunctors of full double groupoids gives a (2,1)-, (1,2)-, (2,2)-morphism of double anafunctors of Lie 2-groupoids, respectively.
\end{lemma}

\begin{proof}
    We show this for a 2-morphism of horizontal anafunctors $\alpha$ between $G\overset{\sim_h}{\twoheadleftarrow}A \to H$ and $G\overset{\sim_h}{\twoheadleftarrow}A' \to H$, as the rest will follow by duality and functoriality. Recall that $\alpha$ is equivalent to a bisimplicial map $A \times_G A' \times \Delta^{1,0} \to H$, $\Wbar$ preserves limits, and $\Wbar \Delta^{1,0} \cong \Delta^1$, so we have that $\Wbar\alpha: \Wbar A \times_{\Wbar G} \Wbar A'\times \Delta^1 \to \Wbar H$ is a 2-morphism of anafunctors from $\Wbar G \overset{\sim}{\twoheadleftarrow} \Wbar A\to  \Wbar H$ to $\Wbar G \overset{\sim}{\twoheadleftarrow} \Wbar A'\to  \Wbar H$. To obtain the desired compatibility with $(2,2)$-morphisms, we use the fact that we are considering the fundamental mapping groupoid of the simplicial mapping complex as the hom-space of $\mathsf{Lie_2Gpd}$.
\end{proof}

Because the $\Wbar$ functor commutes with pullbacks, and as such it maps compositions in the double bicategory $\Ana(\mathsf{LieGpd^2_{full}})$ to the corresponding compositions in $\mathsf{2}\Ana(\mathsf{Lie_2}\Gpd)$, we can assemble Theorem \ref{thm:HorGenMorphToAna-Wbar}, \ref{thm:VerGenMorphToAna-Wbar}, and \ref{thm:11MorphToSquareAna-Wbar}, to obtain the following.

\begin{corollary}\label{cor:double-to-simplicial-double-bicats}
    The functor $\Wbar: \mathsf{LieGpd^2_{full}} \to \mathsf{Lie_2}\Gpd$ extends to a morphism of double bicategories
    \begin{equation*}
        \Wbar: \Ana(\mathsf{LieGpd^2_{full}}) \longrightarrow \mathsf{2}\Ana(\mathsf{Lie_2}\Gpd). 
    \end{equation*}
	Therefore, there is a morphism of double bicategories
	\begin{equation*}
        \mathsf{Bun}(\mathsf{LieGpd^2_{full}}) \longrightarrow \mathsf{2}\Ana(\mathsf{Lie_2}\Gpd)
    \end{equation*}
	obtained by composing $\Wbar$ with the biaction construction. 
\end{corollary}

We now relate $(1,1)$-morphisms of full double Lie groupoids to 2-morphisms in the bicategory of fractions of Lie 2-groupoids localized at weak equivalences, as in Definition \ref{def:2mor ana}.

\begin{theorem}\label{thm:1-1-morph-to-2-hom-diamond}
    Let $(\Omega,A^{JG},B_J^K,C_G^H,D^{KH})$ be a $(1,1)$-morphism of full double Lie groupoids. Let \eqref{diag:double-span-total-action-of-square} be the associated double anafunctor.
    Then the diagram 
\begin{equation}\label{diag:1-1-morph-to-2-hom-diamond}\begin{tikzcd}[ampersand replacement=\&,cramped]
	\& {\widebar{W} \widetilde{C}^v \times_{\widebar{W} H} \widebar{W} \widetilde{D}^h} \& \\
	{\widebar{W}G} \& {\widebar{W} \widetilde{\widetilde{\Omega}}} \& {\widebar{W} K} \\
	\& {\widebar{W} \widetilde{A}^h \times_{\widebar{W} J} \widebar{W} \widetilde{B}^v}
	\arrow["\sim"', two heads, from=1-2, to=2-1]
	\arrow[from=1-2, to=2-3]
	\arrow["\sim"', from=2-2, to=1-2]
	\arrow["\sim"', two heads, from=2-2, to=2-1]
	\arrow["\sim", from=2-2, to=3-2]
	\arrow["\sim", two heads, from=3-2, to=2-1]
	\arrow[from=3-2, to=2-3]
\end{tikzcd}\end{equation}
    commutes, the vertical maps are weak equivalences, and the horizontal map is a hypercover. In other words, it is a 2-morphism of anafunctors in $\mathsf{Lie_2 Gpd}[\huaW^{-1}]$, therefore it defines an equivalence of anafunctors.
\end{theorem}

\begin{proof}
	Starting with \eqref{diag:double-span-total-action-of-square} we apply Theorem \ref{thm:11MorphToSquareAna-Wbar} and obtain \eqref{diag:Wbar-double-anafunctor}. Because both pullbacks in \eqref{diag:1-1-morph-to-2-hom-diamond} are over hypercovers, they exist. Additionally, the maps from each of the pullbacks to $\Wbar G$ are hypercovers because hypercovers are closed under composition and pullbacks. The two maps from $\Wbar \widetilde{\widetilde{\Omega}}$ to each of the pullbacks are obtained by the universal property of the pullback, and the diagram \eqref{diag:1-1-morph-to-2-hom-diamond} commutes because of commutativity of the following diagram:
\[\begin{tikzcd}[ampersand replacement=\&,cramped]
	{\Wbar K} \&\& {\Wbar (\hortilde{D})} \&\& {\Wbar H} \\
	\&\&\& {\Wbar\hortilde{D} \times_{\Wbar H} \Wbar\vertilde{C}} \\
	{\Wbar (\vertilde{B})} \&\& {\Wbar (\widetilde{\widetilde{\Omega}})} \&\& {\Wbar (\vertilde{C})} \\
	\& {\Wbar\vertilde{B} \times_{\Wbar J} \Wbar\hortilde{A}} \\
	{\Wbar J} \&\& {\Wbar (\hortilde{A})} \&\& {\Wbar G.}
	\arrow[from=1-3, to=1-1]
	\arrow["\sim", two heads, from=1-3, to=1-5]
	\arrow[from=2-4, to=1-3]
	\arrow["\sim"{pos=0.7}, two heads, from=2-4, to=3-5]
	\arrow[from=3-1, to=1-1]
	\arrow["\sim", two heads, from=3-1, to=5-1]
	\arrow[from=3-3, to=1-3]
	\arrow["{\exists !}", from=3-3, to=2-4]
	\arrow[from=3-3, to=3-1]
	\arrow["\sim", two heads, from=3-3, to=3-5]
	\arrow["{\exists!}"', from=3-3, to=4-2]
	\arrow["\sim", two heads, from=3-3, to=5-3]
	\arrow[from=3-5, to=1-5]
	\arrow["\sim", two heads, from=3-5, to=5-5]
	\arrow[from=4-2, to=3-1]
	\arrow["\sim"{pos=0.7}, two heads, from=4-2, to=5-3]
	\arrow[from=5-3, to=5-1]
	\arrow["\sim", two heads, from=5-3, to=5-5]
\end{tikzcd}\]
	Furthermore, both maps $\Wbar \widetilde{\widetilde{\Omega}} \to \Wbar G$ coincide and this map is a hypercover, since it factorizes as a composition of hypercovers by diagram chasing. Finally, by the two-out-of-three property of weak equivalences, the maps $\Wbar \widetilde{\widetilde{\Omega}} \to \widebar{W} \widetilde{C}^v \times_{\widebar{W} H} \widebar{W} \widetilde{D}^h$ and $\Wbar \widetilde{\widetilde{\Omega}} \to \widebar{W} \widetilde{A}^h \times_{\widebar{W} J} \widebar{W} \widetilde{B}^v$ are weak equivalences. 
\end{proof}

\begin{remark}
	Note that $\Wbar$ does not map composition of squares to composition of diamonds such as \eqref{diag:1-1-morph-to-2-hom-diamond}. However, this simplified construction is useful in applications. 
\end{remark}

Theorems \ref{thm:11MorphToSquareAna-Wbar} and \ref{thm:1-1-morph-to-2-hom-diamond} give complementary simplicial realizations of a
$(1,1)$-morphism. The first retains its square geometry and is compatible
with the horizontal and vertical structures of the relevant double
bicategories. The second forgets this two-directional organization and
produces a representative of a globular $2$-morphism in the localization at
weak equivalences.  Although the latter construction does not preserve the composition of squares, it is the form needed to compare ordinary generalized morphisms in the applications below.

\section{Non-abelian gerbes as generalized morphisms}\label{sec:nonab gerbes} Non-abelian gerbes are categorified versions of principal bundles in which Lie 2-groups play the role of the structure groups. The various incarnations of this idea have led to multiple definitions in the literature \cite{AschieriCantiniJurco2005,BaezSchreiber2007,MartinsPicken2011,NikolausWaldorf2013}. Here we show how two of those definitions may be formulated and compared directly using the language of double groupoids and $(1,1)$-morphisms. Using our main result, we immediately show that they are different incarnations of generalized morphisms between manifolds and Lie 2-groups in their simplicial description.
\subsection{Abelian gerbes and their generalizations} If $T$ is an abelian Lie group, we may view it as a double groupoid with both side groupoids being trivial. An abelian $T$-gerbe on a manifold $M$ may be viewed as a principal $T$-bundle over a submersion groupoid $P_0 \times_M P_0 \rightrightarrows P_0$ associated to a surjective submersion $p:P_0\to M$ \cite{BehrendXu2011}. A morphism between such abelian $T$-gerbes may be described as a $(1,1)$-morphism 
\[\begin{tikzcd}
	T & \ast & {P_1} & {P_{0}\times_M P_0} & {P_{0}\times_M P_0} \\
	\ast & \ast & {P_0} & {P_0} & {P_0} \\
	{Q_1} & {Q_0} & \Omega & {P_0} & {P_0} \\
	{Q_{0}\times_M Q_0} & {Q_{0}} & {Q_0} & M & M \\
	{Q_{0}\times_M Q_0} & {Q_0} & {Q_0} & M & {M,}
	\arrow[shift left, from=1-1, to=1-2]
	\arrow[shift right, from=1-1, to=1-2]
	\arrow[shift left, from=1-1, to=2-1]
	\arrow[shift right, from=1-1, to=2-1]
	\arrow[shift left, from=1-2, to=2-2]
	\arrow[shift right, from=1-2, to=2-2]
	\arrow[from=1-3, to=1-2]
	\arrow[from=1-3, to=1-4]
	\arrow[shift left, from=1-3, to=2-3]
	\arrow[shift right, from=1-3, to=2-3]
	\arrow[shift left, from=1-4, to=2-4]
	\arrow[shift right, from=1-4, to=2-4]
	\arrow[shift left, no head, from=1-5, to=1-4]
	\arrow[shift right, no head, from=1-5, to=1-4]
	\arrow[shift left, from=1-5, to=2-5]
	\arrow[shift right, from=1-5, to=2-5]
	\arrow[shift left, from=2-1, to=2-2]
	\arrow[shift right, from=2-1, to=2-2]
	\arrow[from=2-3, to=2-2]
	\arrow[from=2-3, to=2-4]
	\arrow[shift left, no head, from=2-5, to=2-4]
	\arrow[shift right, no head, from=2-5, to=2-4]
	\arrow[from=3-1, to=2-1]
	\arrow[shift right, from=3-1, to=3-2]
	\arrow[shift left, from=3-1, to=3-2]
	\arrow[from=3-1, to=4-1]
	\arrow[from=3-2, to=2-2]
	\arrow[from=3-2, to=4-2]
	\arrow[from=3-3, to=2-3]
	\arrow[from=3-3, to=3-2]
	\arrow[from=3-3, to=3-4]
	\arrow[from=3-3, to=4-3]
	\arrow[from=3-4, to=2-4]
	\arrow[from=3-4, to=4-4]
	\arrow[from=3-5, to=2-5]
	\arrow[shift right, no head, from=3-5, to=3-4]
	\arrow[shift left, no head, from=3-5, to=3-4]
	\arrow[from=3-5, to=4-5]
	\arrow[shift left, from=4-1, to=4-2]
	\arrow[shift right, from=4-1, to=4-2]
	\arrow[from=4-3, to=4-2]
	\arrow[from=4-3, to=4-4]
	\arrow[shift left, no head, from=4-5, to=4-4]
	\arrow[shift right, no head, from=4-5, to=4-4]
	\arrow[shift right, no head, from=5-1, to=4-1]
	\arrow[shift left, no head, from=5-1, to=4-1]
	\arrow[shift left, from=5-1, to=5-2]
	\arrow[shift right, from=5-1, to=5-2]
	\arrow[shift right, no head, from=5-2, to=4-2]
	\arrow[shift left, no head, from=5-2, to=4-2]
	\arrow[shift left, no head, from=5-3, to=4-3]
	\arrow[shift right, no head, from=5-3, to=4-3]
	\arrow[from=5-3, to=5-2]
	\arrow[from=5-3, to=5-4]
	\arrow[shift right, no head, from=5-4, to=4-4]
	\arrow[shift left, no head, from=5-4, to=4-4]
	\arrow[shift right, no head, from=5-5, to=4-5]
	\arrow[shift left, no head, from=5-5, to=4-5]
	\arrow[shift left, no head, from=5-5, to=5-4]
	\arrow[shift right, no head, from=5-5, to=5-4]
\end{tikzcd}\]
where the identity arrows indicate unit groupoid structures \cite{BehrendXu2011}.

Therefore, by applying Theorem \ref{thm:1-1-morph-to-2-hom-diamond}, we obtain that the generalized morphisms determined by these $T$-gerbes are equivalent. This observation leads to rephrasing and proving in a simple way the equivalence between various notions of abelian gerbes, including their \v{C}ech cocycle description. In this work, we are interested in replacing $T$ with a (strict) Lie 2-group $G_1\rightrightarrows G_0$ and thus, by making use of the conceptual viewpoint for non-abelian gerbes, we consider a generalized morphism from a manifold into $G$.
The detailed interpretation of non-abelian gerbes --- which were defined as Lie groupoid extensions in \cite{LaurentGengouxStienonXu2009} --- in terms of a generalized morphism from a manifold into a Lie 2-group is developed in \cite{GinotStienon2015}. Additionally, non-abelian gerbes admit two additional models studied in \cite{NikolausWaldorf2013} given by principal 2-bundles and groupoid bundle gerbes. Whereas these two models are related by means of a categorical equivalence, a direct comparison between them when seen as anafunctors or generalized morphisms was missing and so we provide it here.
Compared to \cite{NikolausWaldorf2013}, we make stricter definitions to make the construction more transparent. 

\subsection{Groupoid bundle gerbes and principal 2-bundles}
There is an equivalence between the category of strict Lie 2-groups and the category of \textit{crossed modules} \cite{BrownSpencer1976}. The point of view of crossed modules is  common in the literature about gerbes and so we recall it here. 

\begin{definition}
    A \textbf{crossed module} (of Lie groups) consists of a pair of Lie groups $K$ and $G_0$, a Lie group morphism $\partial: K \to G_0$, and a smooth action 
    \begin{equation*}
        G_0 \times K \to K, \qquad (g,k)\mapsto {}^gk
    \end{equation*}
    by Lie group automorphisms, such that 
    \begin{equation*}
    \partial({}^gk)=g\partial(k)g^{-1}, \qquad {}^{\partial(k)}k'=kk'k^{-1}, \qquad \forall k,k'\in K,\  g\in G_0; 
    \end{equation*} 
\end{definition}

Let $G_1\rightrightarrows G_0$ be a strict Lie 2-group and let $K=\ker s$, so that $G_1 \cong K \times G_0$ and $\partial=t|_K:K\to G_0$ is the Lie group morphism underlying the crossed module associated to $G$. 
In terms of this description, the groupoid structure $G_1\rightrightarrows G_0$ is given by the action groupoid with respect to the action of $K$ on $G_0$ by left translations and so its multiplication is
\[ (k,\partial(k')g)(k',g)=(kk',g), \qquad \forall k,k'\in K,\ g\in G_0. \]
The group structure on $G_1 \cong K\times G_0$ is that of a semidirect product, where the multiplication is
\[ (k,g)(k',g')=(k({}^gk'),gg'), \qquad \forall k,k'\in K,\ g,g'\in G_0.\]
Since we may view a Lie 2-group as a double Lie groupoid with a trivial side as in the diagram
\[\begin{tikzcd}
	{G_1} & \ast \\
	{G_0} & \ast,
	\arrow[shift right, from=1-1, to=1-2]
	\arrow[shift left, from=1-1, to=1-2]
	\arrow[shift right, from=1-1, to=2-1]
	\arrow[shift left, from=1-1, to=2-1]
	\arrow[shift right, no head, from=1-2, to=2-2]
	\arrow[shift left, no head, from=1-2, to=2-2]
	\arrow[shift right, from=2-1, to=2-2]
	\arrow[shift left, from=2-1, to=2-2]
\end{tikzcd}\]
the corresponding $(1,0)$-morphisms and $(0,1)$-morphisms into it are quite different from each other. We say that a \textbf{principal 2-bundle} is a $(1,0)$-morphism from a submersion groupoid into this double Lie groupoid, on the other hand, we say that a \textbf{groupoid bundle gerbe} is a $(0,1)$-morphism from a submersion groupoid into the same double Lie groupoid. In what follows, we shall see how to pass from such $(1,0)$-morphisms to $(0,1)$-morphisms in a way that is mediated by a $(1,1)$-morphism.
\begin{remark}
    For the sake of clarity, we illustrate the argument below under the simplifying assumption that our principal 2-bundles are defined over unit groupoids rather than over submersion groupoids. The general case of submersion groupoids (and even that of a base being a general Lie groupoid) may be treated along similar lines. Also, our notion of groupoid bundle gerbe is a strict version of the one in \cite{NikolausWaldorf2013}.
\end{remark}
Explicitly, a principal 
$(G_1\rightrightarrows G_0)$-bundle over a manifold $M$ is a Lie groupoid $P_1\rightrightarrows P_0$ equipped with a left principal action of $G_1\rightrightarrows G_0$ viewed as a double groupoid with trivial side and such that the quotient is the unit groupoid over $M$. So we have a principal group action $G_i\times P_i\to P_i$ at each level $i=0,1$, and together these actions define a groupoid morphism. We say that this type of action is a \textbf{Lie 2-group action}. We denote such an action by $g\cdot x={}^gx$ for all $g\in G_1$ and $x\in P_1$. 

From a Lie 2-group action, we may produce a new principal bundle but now for an action of  $G_1\rightrightarrows G_0$ viewed as a groupoid, which determines a $(0,1)$-morphism from a submersion groupoid to $G_1\rightrightarrows G_0$. This new principal bundle is our desired groupoid bundle gerbe, and its definition follows \cite[\S 7]{NikolausWaldorf2013}, though we are using slightly different conventions. Consider the semidirect product groupoid given by $G_0$ acting on $P_1\rightrightarrows P_0$ via the unit inclusion $G_0\subset G_1$. Indeed, since this action is obtained by restriction of a Lie 2-group action to the unit group, it is an action by groupoid automorphisms, so that we may equip the product $(Q_1\rightrightarrows Q_0)=(G_0\ltimes P_1\rightrightarrows P_0)$ with the semidirect product structure maps 

\[ t(a,x)=t(x), \qquad s(a,x)={}^{a^{-1}}s(x), \qquad (a,x)(b,y)=(ab,x({}^ay)). \]

Additionally, we have a morphism from this groupoid to the submersion groupoid $P_0\times_M P_0 \rightrightarrows P_0 $ given by the identity on $P_0$ at level 0 and
\begin{equation*}
    Q_1 \to P_0\times_M P_0,   \qquad  (a,x)\mapsto (t(x),{}^{a^{-1}}s(x)), 
\end{equation*}
at level 1.
This groupoid morphism is the quotient map for the action of $G_1\rightrightarrows G_0$ --- as a Lie groupoid --- on the projection $\text{pr}:Q_1\to G_0$, defined by
\begin{align}
     (k,a)\cdot (a,x)= (\partial(k)a,{}^{(k^{-1},\partial(k))}x). \label{eq:act 2group}
\end{align}

\begin{lemma}\label{lem:bunger act} The action \eqref{eq:act 2group} defines a morphism of Lie groupoids
\[ G_1\times_{G_0} Q_1 \to Q_1 \]
where we view $G_1\times_{G_0} Q_1\rightrightarrows Q_0$ as a fiber product of the Lie groupoids $G_1\rightrightarrows \ast$ and $Q_1\rightrightarrows Q_0$ over $s:G_1\to G_0$ and $\text{pr}:Q_1\to G_0$.  
\end{lemma}

\begin{proof} For clarity, we distinguish different groupoid multiplications by subindices in what follows. Let $(k,g),(k',g')\in K\times G_0\cong G_1$ and $(g,p),(g',p')\in Q_1$ be respectively composable. Then we have that
    \[\begin{aligned} 
        &m_{Q_1}((k,g)\cdot (g,p),(k',g')\cdot (g',p'))=\\
        &m_{Q_1}((\partial(k)g,{}^{(k^{-1},\partial(k))}p),(\partial(k')g',{}^{((k')^{-1},\partial(k'))}p'))=\\
        &(\partial(k)g\partial(k')g',m_{P_1}({}^{(k^{-1},\partial(k))}p,{}^{({}^{\partial(k)g}(k')^{-1},\partial(k)g\partial(k'))}p')).
    \end{aligned}\]
    But this last expression agrees with the effect of the action on the composition
    \[ (k({}^gk'),gg')\cdot (gg',p({}^gp'))= (\partial(k({}^gk'))gg',{}^{((k({}^gk'))^{-1},\partial(k)g\partial(k')g^{-1})}(p({}^gp'))),\]
    as a consequence of the crossed module identities and the fact that
    \[ m_{P_1}({}^{(k^{-1},\partial(k))}p,{}^{({}^{\partial(k)g}(k')^{-1},\partial(k)g\partial(k'))}p')={}^{((k({}^gk'))^{-1},\partial(k)g\partial(k')g^{-1})}(p({}^gp')),  \]
    since $G_1\rightrightarrows G_0$ acts on $P_1\rightrightarrows P_0$ by a Lie 2-group action.
\end{proof}

Now we see how to compare this groupoid bundle gerbe with the principal 2-bundle from which it originated. In the course of the following discussion, $1_x$ stands for the unit inclusion of $x\in P_0$.
\begin{proposition}\label{prop:2bun to bunger}
    The Lie groupoids $P_1\rightrightarrows P_0$ and $Q_1\rightrightarrows Q_0$ admit left actions on $\Omega=P_1$ defined as follows:
    \begin{equation}
        \begin{aligned}
           &\alpha_V:P_1\times_{P_0} \Omega= P_1\times_{s, P_0, s} P_1 \to P_1, \qquad (x,y)\mapsto yx^{-1}\\
        &\alpha_H:Q_1\times_{Q_0} \Omega=Q_1\times_{s, P_0, t}P_1\to P_1;
        \end{aligned}\label{eq:action and bundle gerbe}
    \end{equation}
    where $\alpha_H$ is determined by
    \[ \alpha_H((g,r),p)=\begin{cases} {}^gp,\quad &\text{if $r={}^g1_{t(p)}$},\\
    p({}^{(k,1)}1_{s(p)})^{-1}, \quad &\text{if $(g,r)=(\partial(k),{}^{(k^{-1},\partial(k))}1_{t(p)})$}.  
    \end{cases}\]
    Moreover, we have that $\alpha_H$ and $\alpha_V$ together with the two actions of $G_1\rightrightarrows G_0$ described above constitute a $(1,1)$-morphism 
    \begin{equation}
\begin{tikzcd}
	{G_1} & \ast & {P_1} & M & M \\
	{G_0} & \ast & {P_0} & M & M \\
	{Q_1} & {Q_0} & \Omega & M & M \\
	{P_0\times_M P_0} & {P_0} & {P_0} & M & M \\
	{P_0\times_M P_0} & {P_0} & {P_0} & {M} & {M,}
	\arrow[shift left, from=1-1, to=1-2]
	\arrow[shift right, from=1-1, to=1-2]
	\arrow[shift left, from=1-1, to=2-1]
	\arrow[shift right, from=1-1, to=2-1]
	\arrow[shift left, from=1-2, to=2-2]
	\arrow[shift right, from=1-2, to=2-2]
	\arrow[from=1-3, to=1-2]
	\arrow[from=1-3, to=1-4]
	\arrow[shift left, from=1-3, to=2-3]
	\arrow[shift right, from=1-3, to=2-3]
	\arrow[shift left, no head, from=1-4, to=2-4]
	\arrow[shift right, no head, from=1-4, to=2-4]
	\arrow[shift left, no head, from=1-5, to=1-4]
	\arrow[shift right, no head, from=1-5, to=1-4]
	\arrow[shift left, no head, from=1-5, to=2-5]
	\arrow[shift right, no head, from=1-5, to=2-5]
	\arrow[shift left, from=2-1, to=2-2]
	\arrow[shift right, from=2-1, to=2-2]
	\arrow[from=2-3, to=2-2]
	\arrow[from=2-3, to=2-4]
	\arrow[shift left, no head, from=2-5, to=2-4]
	\arrow[shift right, no head, from=2-5, to=2-4]
	\arrow[from=3-1, to=2-1]
	\arrow[shift right, from=3-1, to=3-2]
	\arrow[shift left, from=3-1, to=3-2]
	\arrow[from=3-1, to=4-1]
	\arrow[from=3-2, to=2-2]
	\arrow[from=3-2, to=4-2]
	\arrow[from=3-3, to=2-3]
	\arrow[from=3-3, to=3-2]
	\arrow[from=3-3, to=3-4]
	\arrow[from=3-3, to=4-3]
	\arrow[from=3-4, to=2-4]
	\arrow[from=3-4, to=4-4]
	\arrow[from=3-5, to=2-5]
	\arrow[shift right, no head, from=3-5, to=3-4]
	\arrow[shift left, no head, from=3-5, to=3-4]
	\arrow[from=3-5, to=4-5]
	\arrow[shift left, from=4-1, to=4-2]
	\arrow[shift right, from=4-1, to=4-2]
	\arrow[from=4-3, to=4-2]
	\arrow[from=4-3, to=4-4]
	\arrow[shift left, no head, from=4-5, to=4-4]
	\arrow[shift right, no head, from=4-5, to=4-4]
	\arrow[shift left, no head, from=5-1, to=4-1]
	\arrow[shift right, no head, from=5-1, to=4-1]
	\arrow[shift left, from=5-1, to=5-2]
	\arrow[shift right, from=5-1, to=5-2]
	\arrow[shift right, no head, from=5-2, to=4-2]
	\arrow[shift left, no head, from=5-2, to=4-2]
	\arrow[shift left, no head, from=5-3, to=4-3]
	\arrow[shift right, no head, from=5-3, to=4-3]
	\arrow[from=5-3, to=5-2]
	\arrow[from=5-3, to=5-4]
	\arrow[shift right, no head, from=5-4, to=4-4]
	\arrow[shift left, no head, from=5-4, to=4-4]
	\arrow[shift right, no head, from=5-5, to=4-5]
	\arrow[shift left, no head, from=5-5, to=4-5]
	\arrow[shift left, no head, from=5-5, to=5-4]
	\arrow[shift right, no head, from=5-5, to=5-4]
\end{tikzcd}\label{eq:11 from nonab ger}\end{equation}
where the quotient map $\Omega\to M$ is given by sending $p$ to the $G_0$-orbit of $t(p)$. 
\end{proposition}
\begin{proof} Thanks to Lemma \ref{lem:bunger act}, it only remains to check the compatibility of the actions. Note that each element in $Q_1$ may be written uniquely as a product of two elements of the following forms
    \[ \iota_1(g,x)=(g,{}^g1_x), \qquad \iota_2(k,x)=(\partial(k),{}^{(k^{-1},\partial(k))}1_x); \]
    for some $x\in P_0$. Furthermore, the elements of the second kind cover all the isotropy groups in $Q_1\rightrightarrows Q_0$ and so the expression for $\iota_2$ defines a Lie groupoid morphism $\iota_2:K\times P_0 \to Q_1$, where we view $K\times P_0$ as a bundle of Lie groups over $P_0$. Similarly, $\iota_1:G_0\times P_0 \to Q_1$ is a groupoid morphism with respect to the action groupoid associated to the $G_0$-action on $P_0$. So the formula above is enough to define $\alpha_H$. We need to verify that this is indeed an action and that it is compatible with $\alpha_V$. To verify that $\alpha_H$ is a well defined action, it is enough to check that factorizing an element in $Q_1$ as a product of elements in the images of $\iota_1$ and $\iota_2$ in the two possible orders leads to the same action on $\Omega=P_1$. So consider
    \[ \widebar{q}:=\iota_2(k,{}^gt(p))\iota_1(g,t(p))=\iota_1(g,t(p))\iota_2({}^{g^{-1}}k,t(p))\in Q_1 \]
    where $(k,g)\in K\times G_0$ and $p\in P_1$ and use the first factorization to obtain that
    \[ \alpha_H(\iota_2(k,{}^gt(p)),\alpha_H(\iota_1(g,t(p)),p))=({}^gp)({}^{(k,g)}1_{s(p)})^{-1}. \] 
    On the other hand, using the second factorization, we have that
    \begin{align*}
&\alpha_H(\iota_1(g,t(p)),\alpha_H(\iota_2({}^{g^{-1}}k,t(p)),p))  =\alpha_H(\iota_1(g,t(p)) ,p({}^{({}^{g^{-1}}k,1)}1_{s(p)})^{-1})=\\
        &={}^g(p({}^{({}^{g^{-1}}k,1)}1_{s(p)})^{-1})=({}^gp)({}^{(k,g)}1_{s(p)})^{-1};
    \end{align*} 
    where we use the fact that $G_0$ acts by automorphisms on $P_1\rightrightarrows P_0$. This confirms that the action of $\widebar{q}$ is unambiguously defined and it also implies that $\alpha_H$ is an action.  
    
    To make the argument clearer, we denote vertical and horizontal actions by concatenation of boxes in what follows. We may immediately verify the compatibility of the actions in \eqref{eq:action and bundle gerbe} for these two special situations: 
    
    \[ \begin{tikzpicture}[x=2cm, y=2cm] 
    
    \draw[thin] (0,0) grid[step=1] (2,2);

    \node at (0.5, 1.5) {$(1,g)$};
    
    \node at (1.5, 1.5) {$q$};
    
    \node at (0.5, 0.5) {$\iota_1(g,t(p))$};
    
    \node at (1.5, 0.5) {$p$};

\end{tikzpicture} \quad\begin{tikzpicture}[x=2cm, y=2cm] 

    \draw[thin] (0,0) grid[step=1] (2,2);

    \node at (0.5, 1.5) {$(k',\partial(k))$};
    
    \node at (1.5, 1.5) {$q$};
    
    \node at (0.5, 0.5) {$\iota_2(k,t(p))$};
    
    \node at (1.5, 0.5) {$p$};

\end{tikzpicture}\]
where $p,q\in P_1$ and $k,k'\in K$, $g\in G_0$. We say that the first diagram is of type I and the second one of type II. The general case reduces to these two situations by suitable factorizations. In fact, consider the following diagram
\[ \begin{tikzpicture}
    \begin{scope}[step=1.5cm]
        
        \draw[thin] (0,0) grid (3,3);

        \node at (0.75, 2.25) {$(k,g)$}; 
        \node at (2.25, 2.25) {$q$}; 
        \node at (0.75, 0.75) {$(g,r)$}; 
        \node at (2.25, 0.75) {$p$}; 
    \end{scope}

    \node[font=\Huge] at (4, 1.5) {$=$};

    \begin{scope}[shift={(5cm,0)}, x=2cm, y=1.5cm]

        \draw[thin] (0,0) rectangle (3,2);
        \draw[thin] (0,1) -- (3,1);
        \draw[thin] (1,0) -- (1,2);
        \draw[thin] (2,0) -- (2,2);
        \node at (0.5, 1.5) {$(1,g')$};
        \node at (1.5, 1.5) {$(k',\partial(l))$};
        \node at (2.5, 1.5) {$q$};
        \node at (0.5, 0.5) {$\iota_1(g',t(p))$};
        \node at (1.5, 0.5) {$\iota_2(l,t(p))$};
        \node at (2.5, 0.5) {$p$};
    \end{scope}
\end{tikzpicture}\]
in which we factorize $(g,r)=\iota_1(g',t(p))\iota_2(l,t(p))$ for some $(l,g')\in K\times G_0$ and then we express accordingly $(k,g)=(1,g')(k',\partial(l))$ for some $k'\in K$. Then, the fact that the action of the last four squares in the diagram above is well defined follows from the fact that it is a diagram of type II. But then the action of the elements in the first column on the result of acting with the second column on the third is
\[ \begin{tikzpicture}[y=2cm] 
    \draw[thin] (0,0) rectangle (7.5, 2);
    \draw[thin] (2.5, 0) -- (2.5, 2);
    \draw[thin] (0,1) -- (7.5, 1);

    \node at (1.25, 1.5) {$(1,g')$};
    \node at (1.25, 0.5) {$\iota_1(g',t(p))$};
    \node at (5.0, 1.5) {${}^{(k',\partial(l))}q$};
    \node at (5.0, 0.5) {$\alpha_H(\iota_2(l,t(p)),p)$};

\end{tikzpicture}\]
which is also well defined because it is a diagram of type I. This confirms that the starting diagram determines a well defined action and hence the actions in \eqref{eq:action and bundle gerbe} are compatible and thus we have a $(1,1)$-morphism as required. \end{proof}

Applying Theorem \ref{thm:1-1-morph-to-2-hom-diamond} to \eqref{eq:11 from nonab ger}, we obtain the following result.
\begin{corollary}\label{cor:nonab ger}
    A principal 2-bundle $P=(P_1\rightrightarrows P_0)$ over $M$ with structure 2-group $G=(G_1\rightrightarrows G_0)$ and its associated groupoid bundle gerbe $Q=(Q_1\rightrightarrows Q_0)$ determine equivalent anafunctors from $\widebar{W}M=M$ to $\Wbar G$:
    \[\begin{tikzcd}[ampersand replacement=\&,sep=scriptsize]
	\& { \widebar{W} \widetilde{P}^h} \\
	{M} \& {\widebar{W} \widetilde{\widetilde{\Omega}}} \& {\widebar{W} G} \\
	\& { \widebar{W} \widetilde{Q}^v\times_{\Wbar (P_0\times_M P_0)} \Wbar\widetilde{P_0}^h}
	\arrow["\sim"', two heads, from=1-2, to=2-1]
	\arrow[from=1-2, to=2-3]
	\arrow["\sim"', from=2-2, to=1-2]
	\arrow["\sim"', two heads, from=2-2, to=2-1]
	\arrow["\sim", from=2-2, to=3-2]
	\arrow["\sim", two heads, from=3-2, to=2-1]
	\arrow[from=3-2, to=2-3]
\end{tikzcd}\]
where we use the notation in \eqref{eq:11 from nonab ger}. \qed
\end{corollary}
This result is consistent with the modern approach to (non-abelian) cocycles and their equivalence in terms of 2-morphisms, see e.g. \cite[\S 7]{Wolfson2016} and \cite{NikolausSchreiberStevenson2015}. Note that the other two models for non-abelian gerbes compared in \cite{NikolausWaldorf2013} already have the form of generalized morphisms into the structure 2-group, being given by non-abelian Cech cocycles or maps into the classifying space of the Lie 2-group. Therefore, combining these facts with the work in \cite{GinotStienon2015}, all these models of non-abelian gerbes are unified by the language of 2-morphisms and anafunctors of 2-groupoids.  

More importantly, we expect to use Corollary \ref{cor:nonab ger} for comparing different notions of connections (or connective structures) and parallel transport for either principal 2-bundles or groupoid bundle gerbes \cite{AschieriCantiniJurco2005,BaezSchreiber2007,LaurentGengouxStienonXu2009,MartinsPicken2011}, and, therefore, unify their study in higher gauge theory through the use of simplicial methods, as proposed in \cite{JurcoSaemannWolf2016}. 

\section{Double Morita equivalences from groupoid factorizations and Manin triples}\label{sec:int manin}
A natural question in Lie theory is how to integrate infinitesimal objects (Lie algebra-like objects) to global objects, which typically take the form of higher Lie groupoids. In particular, it has been proposed that Lie bialgebroids \cite{MackenzieXu1994} integrate to symplectic double groupoids \cite{Weinstein1988,LuWeinstein1989}. Such a problem was solved for certain Lie bialgebroids in \cite{Alvarez2023}, using ideas coming from shifted symplectic geometry and recovering a construction of double Lie groupoids which, in its general form, goes back to \cite{AndruskiewitschNatale2009}. 

On the other hand, symplectic $(1,1)$-morphisms were introduced in
\cite{AGJ2024} in order to relate the different Manin triples arising
from a given generalized K\"ahler manifold. Since Lie bialgebroids
give rise to Manin triples within Courant algebroids \cite{liuweixu},
we may apply $\Wbar$ to a symplectic double groupoid in order to obtain
the integration of the corresponding Courant algebroid
\cite{MehtaTang2011}. Here, we reconcile the integration conducted in
\cite{Alvarez2023} with the construction of $(1,1)$-morphisms in the
generalized K\"ahler situation described in \cite{AGJ2024}.

The work in
\cite{Alvarez2023} applies to Manin triples sitting inside special Courant algebroids called heterotic. A Courant algebroid $E$ over $M$ is heterotic if the transitive
Lie algebroid $E/\rho^*(T^*M)$ is integrable, where $\rho:E\to TM$ is the anchor map, and $E$ is identified with $E^*$ using its  metric. Consider a pair of Lie groupoids $C,D$ integrating transverse Dirac structures
$\mathfrak C,\mathfrak D\subset E$ in a heterotic Courant algebroid. Then, the main result in
\cite{Alvarez2023} shows how to construct a symplectic double groupoid which integrates the Manin triple $(E,\mathfrak{C},\mathfrak{D})$.
Suppose that we have additional Dirac structures
$\mathfrak A,\mathfrak B\subset E$ which are still suitably transverse
and possess corresponding integrations $A$ and $B$. Our main result in this section relates the corresponding integrations of $(E,\mathfrak{A},\mathfrak{B})$ and of $(E,\mathfrak{C},\mathfrak{D})$. 

Firstly, let us explain the relationship that we assume to exist between $(E,\mathfrak{A},\mathfrak{B})$ and $(E,\mathfrak{C},\mathfrak{D})$ at the infinitesimal level. We may view $\mathfrak A,\mathfrak B$
as determining Maurer--Cartan elements in the differential graded
Lie algebras controlling deformations of $\mathfrak C$ and
$\mathfrak D$ as Dirac structures inside $E$ \cite{liuweixu}.
Relative to $E=\mathfrak C\oplus\mathfrak D$, we write
\[
\begin{aligned}
\mathfrak A
 &=\{c+\varepsilon^\sharp(c):c\in\mathfrak C\},
&
\varepsilon&\in\Gamma(\wedge^2\mathfrak C^*),
&
d_{\mathfrak C}\varepsilon
 +\tfrac12[\varepsilon,\varepsilon]_{\mathfrak D}&=0,
\\
\mathfrak B
 &=\{\pi^\sharp(d)+d:d\in\mathfrak D\},
&
\pi&\in\Gamma(\wedge^2\mathfrak D^*),
&
d_{\mathfrak D}\pi
 +\tfrac12[\pi,\pi]_{\mathfrak C}&=0.
\end{aligned}
\]

Here, we show that the global counterpart of such a pair of
infinitesimal deformations is a symplectic $(1,1)$-morphism, which
then, under $\Wbar$, induces symplectically Morita equivalent
integrations of the same underlying transitive Courant algebroid
$E$. In other words, the main result of this section is an
integration of the infinitesimal diagram of deformations of
Dirac structures indicated below:
\[
\begin{tikzcd}[
  row sep=large,
  column sep=huge,
  cells={anchor=center}
]
(E;\mathfrak A,\mathfrak B)
&
(E;\mathfrak C,\mathfrak B)
  \arrow[l, swap, "\mathfrak C\rightsquigarrow\mathfrak A"]
\\
(E;\mathfrak A,\mathfrak D)
  \arrow[u, "\mathfrak D\rightsquigarrow\mathfrak B"]
&
(E;\mathfrak C,\mathfrak D)
  \arrow[l, "\mathfrak C\rightsquigarrow\mathfrak A"]
  \arrow[u, swap, "\mathfrak D\rightsquigarrow\mathfrak B"]
\end{tikzcd}
\]
Note that this diagram generalizes the situation considered in
\cite[\S 1]{AGJ2024} to the case of a transitive Courant algebroid.
We emphasize that our technique of integration is not based on
$B$-field deformations as in \cite{AGJ2024} but is instead an
extension of the method used in \cite{Alvarez2023}.
\subsection{(1,1)-morphisms from groupoid factorizations} Let $A, B,C,D$ and $K$ be Lie groupoids over the same base $M$ and suppose that we have Lie groupoid morphisms $\Phi_A:A\to K$, $\Phi_B:B\to K$, $\Phi_C:C\to K$ and $\Phi_D:D\to K$ such that the following pairs of morphisms are transverse 
\[ (\Phi_A,\Phi_B),\quad (\Phi_B,\Phi_C),\quad (\Phi_A,\Phi_D),\quad (\Phi_C,\Phi_D). \]
Then we may produce several fiber products using this notation
\begin{align}
    \X_{U,V}^{W,Z}=\{
    (u,v,w,z)\in U_s \times_{t} V \times W_s\times_{t} Z\ |\ \Phi_U(u)\Phi_V(v)=\Phi_W(w)\Phi_Z(z)\}.\label{eq:slim}
\end{align}
The natural projection maps give rise to a $(1,1)$-morphism of double Lie groupoids
\begin{equation}
\begin{tikzcd}
	{\X_{A,B}^{B,A}} & A & {\X_{A,B}^{B,C}} & C & {\X_{C,B}^{B,C}} \\
	B & M & B & M & B \\
	{\X_{A,D}^{B,A}} & A & \X_{A,D}^{B,C} & C & {\X_{C,D}^{B,C}} \\
	D & M & D & M & D \\
	{\X_{A,D}^{D,A}} & A & {\X_{A,D}^{D,C}} & C & {\X_{C,D}^{D,C}}
	\arrow[shift left, from=1-1, to=1-2]
	\arrow[shift right, from=1-1, to=1-2]
	\arrow[shift left, from=1-1, to=2-1]
	\arrow[shift right, from=1-1, to=2-1]
	\arrow[shift left, from=1-2, to=2-2]
	\arrow[shift right, from=1-2, to=2-2]
	\arrow[from=1-3, to=1-2]
	\arrow[two heads, from=1-3, to=1-4]
	\arrow[shift left, from=1-3, to=2-3]
	\arrow[shift right, from=1-3, to=2-3]
	\arrow[shift left, from=1-4, to=2-4]
	\arrow[shift right, from=1-4, to=2-4]
	\arrow[shift left, from=1-5, to=1-4]
	\arrow[shift right, from=1-5, to=1-4]
	\arrow[shift left, from=1-5, to=2-5]
	\arrow[shift right, from=1-5, to=2-5]
	\arrow[shift left, from=2-1, to=2-2]
	\arrow[shift right, from=2-1, to=2-2]
	\arrow[from=2-3, to=2-2]
	\arrow[two heads, from=2-3, to=2-4]
	\arrow[shift left, from=2-5, to=2-4]
	\arrow[shift right, from=2-5, to=2-4]
	\arrow[from=3-1, to=2-1]
	\arrow[shift right, from=3-1, to=3-2]
	\arrow[shift left, from=3-1, to=3-2]
	\arrow[two heads, from=3-1, to=4-1]
	\arrow[from=3-2, to=2-2]
	\arrow[two heads, from=3-2, to=4-2]
	\arrow[from=3-3, to=2-3]
	\arrow[from=3-3, to=3-2]
	\arrow[two heads, from=3-3, to=3-4]
	\arrow[two heads, from=3-3, to=4-3]
	\arrow[from=3-4, to=2-4]
	\arrow[two heads, from=3-4, to=4-4]
	\arrow[from=3-5, to=2-5]
	\arrow[shift right, from=3-5, to=3-4]
	\arrow[shift left, from=3-5, to=3-4]
	\arrow[two heads, from=3-5, to=4-5]
	\arrow[shift left, from=4-1, to=4-2]
	\arrow[shift right, from=4-1, to=4-2]
	\arrow[from=4-3, to=4-2]
	\arrow[two heads, from=4-3, to=4-4]
	\arrow[shift left, from=4-5, to=4-4]
	\arrow[shift right, from=4-5, to=4-4]
	\arrow[shift right, from=5-1, to=4-1]
	\arrow[shift left, from=5-1, to=4-1]
	\arrow[shift left, from=5-1, to=5-2]
	\arrow[shift right, from=5-1, to=5-2]
	\arrow[shift right, from=5-2, to=4-2]
	\arrow[shift left, from=5-2, to=4-2]
	\arrow[shift left, from=5-3, to=4-3]
	\arrow[shift right, from=5-3, to=4-3]
	\arrow[from=5-3, to=5-2]
	\arrow[two heads, from=5-3, to=5-4]
	\arrow[shift right, from=5-4, to=4-4]
	\arrow[shift left, from=5-4, to=4-4]
	\arrow[shift right, from=5-5, to=4-5]
	\arrow[shift left, from=5-5, to=4-5]
	\arrow[shift left, from=5-5, to=5-4]
	\arrow[shift right, from=5-5, to=5-4]
\end{tikzcd}\label{eq:11 from gpd fact}
\end{equation}
To explain the actions and groupoid structures therein it is convenient to adopt the matrix notation
\[ (u,v,w,z)=\begin{bmatrix}
  & w &   \\
u &   & z \\
  & v &  
\end{bmatrix}. \]
So, for instance, the horizontal multiplication on $\mathcal{X}_{A,B}^{B,A}$ is given by
\[ \begin{bmatrix}
  & b_1 &   \\
a_1 &   & a_2 \\
  & b_2 &  
\end{bmatrix} \Join \begin{bmatrix}
  & b_1' &   \\
a_2 &   & a_3 \\
  & b_2' &  
\end{bmatrix}=\begin{bmatrix}
  & b_1b_1' &   \\
a_1 &   & a_3 \\
  & b_2b_2' &  
\end{bmatrix}, \]
the vertical multiplication is defined by the analogous expression
\[ \begin{bmatrix}
  & b_1 &   \\
a_1 &   & a_2 \\
  & b_2 &  
\end{bmatrix} \vJoin \begin{bmatrix}
  & b_2 &   \\
a_1' &   & a_2' \\
  & b_3 &  
\end{bmatrix}=\begin{bmatrix}
  & b_1 &   \\
a_1a_1' &   & a_2a_2' \\
  & b_3 &  
\end{bmatrix},\]
and all the other groupoid structures and actions are defined by formally identical expressions. We shall say, following \cite{AndruskiewitschNatale2009}, that a double Lie groupoid is {\bf slim} if each of its squares is determined by its boundary, such is the case for $\X_{U,V}^{V,U}$ as in \eqref{eq:slim}. Such a double Lie groupoid is full if the following map is a surjective submersion 
\[ U_s\times_{t} V\to K, \qquad (u,v)\mapsto \Phi_U(u)\Phi_V(v), \]
see \cite{AndruskiewitschOchoaTiraboschi2012}. In this situation, we say that $U$ and $V$ provide a {\bf groupoid factorization} for $K$, see \cite{AndruskiewitschNatale2009,AndruskiewitschOchoaTiraboschi2012}. We assume that such a condition is met in what follows; without it, the statements that follow still admit natural local versions. 
\begin{proposition}\label{pro:11 gpd fac}
    The diagram in \eqref{eq:11 from gpd fact} is a $(1,1)$-morphism in which all the generalized morphisms are Morita equivalences. 
\end{proposition}
\begin{proof}
The defining equations of the spaces $\X_{U,V}^{W,Z}$ are preserved by
the componentwise groupoid multiplications, since all the maps $\Phi_U$
are Lie groupoid morphisms. Hence, the actions above are well defined, and their associativity and
mutual commutativity follow directly from the corresponding properties
of the groupoid multiplications.

The fullness assumptions ensure that the relevant moment maps are
surjective submersions. Moreover, every action is principal since it is determined by groupoid multiplication on either $A,B,C$ or $D$. Thus all the bibundles are biprincipal, and hence define Morita equivalences. Therefore, \eqref{eq:11 from gpd fact} is a $(1,1)$-morphism with the asserted properties.
\end{proof}
As a consequence, the Lie 2-groupoids obtained by applying $\Wbar$ to \eqref{eq:11 from gpd fact} are related by a $(1,1)$-morphism of Lie 2-groupoids. In what follows we outline a consequence of this observation in the context of the integration of Manin triples by symplectic double groupoids.
\subsection{Integration of Manin triples and symplectic (1,1)-morphisms} Suppose that the Lie groupoids $A,B,C,D$ are respective integrations of Dirac structures $\mathfrak{A}, \mathfrak{B},\mathfrak{C}, \mathfrak{D}$ inside a transitive Courant algebroid $E\to M$ satisfying the transversality condition
\[ E=\mathfrak{A}\oplus \mathfrak{B}=\mathfrak{B}\oplus \mathfrak{C}=\mathfrak{A}\oplus \mathfrak{D}=\mathfrak{C}\oplus \mathfrak{D}; \]
and further suppose that $K$ is an integration of the Lie algebroid $E/\rho^*(T^*M)$, see \cite{Alvarez2023} for a detailed discussion. The double groupoids defined therein are exactly given by the construction described above, but, in addition, they are equipped with symplectic structures, as explained below. In fact, \eqref{eq:11 from gpd fact} may be rendered into a {\bf symplectic $(1,1)$-morphism} in this situation. To that end, it is enough to equip the spaces $\X_{U,V}^{W,Z}$ with suitable symplectic forms. Recall that if $G\rightrightarrows S$ is a groupoid equipped with a symplectic form and $P$ is a symplectic manifold, an action of $G$ on a map $P\to S$ is {\bf symplectic} if its graph is Lagrangian in $G\times P \times \overline{P}$, where the overline means that the symplectic form is taken with a sign on that factor. In particular, a symplectic form $\omega$ (or just a 2-form) on $G$ is {\bf multiplicative} if the action of $G$ on itself by left multiplication is symplectic (for a possibly degenerate 2-form, the multiplicativity condition is $\delta \omega=0$, for $\delta$ the simplicial differential).

\begin{definition}[Symplectic $(1,1)$-morphism {\cite{AGJ2024}}]
A $(1,1)$-morphism is {\bf symplectic} if its four corner double Lie groupoids are symplectic double groupoids, its four edge groupoids are symplectic groupoids, its central manifold is symplectic, and every groupoid and double-groupoid action appearing in its defining diagram is symplectic.
\end{definition}

Recall that the Bott-Shulman-Stasheff double complex of a simplicial manifold $X_\bullet$ is given by the complexes of differential forms on each degree: $(\Omega^m(X_n),d,\delta)$. where $d$ is the de Rham differential and $\delta$ is the simplicial differential. In the situation at hand, there is a distinguished closed element $\omega_1+\omega_2$ on $(\Omega^m(K_n),d,\delta)$ which classifies a certain multiplicative Courant algebroid (or CA-groupoid) over $K$ as in \cite[Thm. 3.5]{Alvarez2023}, where $\omega_1\in \Omega^3(K)$ and $\omega_2\in \Omega^2(K_2)$.\footnote{In this notation $K_2 := K^{(2)}=K\times_{s,M,t} K$, $K_1:=K$, and $K_0 = M$.} This Courant algebroid is $TK\oplus T^*K$ with the bracket twisted by $\omega_1$ and with its standard multiplication deformed by $\omega_2$. These differential forms are obtained by splitting the (exact) action Courant algebroid determined by $E\times E$ acting on the anchor morphism $(t,s):K\to M\times M$, where this action Courant algebroid structure is defined on $t^*E\oplus s^*E$, see \cite[Thm. 3.5]{Alvarez2023} (we don't specify the signs of the pairing on $E\times E$ for simplicity). So, a choice of splitting identifies $t^*E\oplus s^*E$ with $TK\oplus T^*K$ and, in terms of this splitting, the graph of the natural groupoid multiplication on $t^*E\oplus s^*E$ becomes the graph of $\omega_2$, see \cite[Prop. 2.3, Thm. 2.8]{Alvarez2023}. 

On the other hand, under natural connectivity assumptions, $A,B,C,D$ may be equipped with respective 2-forms $\sigma_A,\sigma_B,\sigma_C,\sigma_D$ which serve as respective primitives for the pullbacks of $\omega_1+\omega_2$ along $\Phi_A, \Phi_B,\Phi_C, \Phi_D$ and satisfy a certain nondegeneracy condition making them into {\em 2-shifted lagrangian structures} \cite[\S 4.2]{Alvarez2023}. We denote the component projections on $\X_{U,V}^{W,Z}$ by
\[
\pi_U,\pi_V,\pi_W,\pi_Z, 
\] and the induced maps to $K_2$ by
\[
q_{U,V}:=(\Phi_U\circ\pi_U,\Phi_V\circ\pi_V),
\qquad
q_{W,Z}:=(\Phi_W\circ\pi_W,\Phi_Z\circ\pi_Z).
\]
These maps allow us to define the 2-form on $\X_{U,V}^{W,Z}$ by
\begin{equation}
\begin{aligned}
\omega_{U,V}^{W,Z}
={}&
\pi_U^*\sigma_U-\pi_Z^*\sigma_Z
-\pi_W^*\sigma_W+\pi_V^*\sigma_V-q_{U,V}^*\omega_2+q_{W,Z}^*\omega_2 .
\end{aligned}
\label{eq:2form}
\end{equation}
To show nondegeneracy of this 2-form we use the following lemma.
\begin{lemma}\label{lem:nondegenerate-bimodule}
Let $(G,\omega_G)\rightrightarrows M$ and
$(H,\omega_H)\rightrightarrows N$ be symplectic groupoids, and let
$P$ be a biprincipal $G$--$H$ bibundle equipped with a $2$-form
$\omega_P$. If the graphs of the left and right actions are isotropic
with respect to the corresponding product forms, then $\omega_P$ is
nondegenerate.
\end{lemma}

\begin{proof}
Consider the Lie groupoid
\[
\mathcal L(P)
   =G\sqcup P\sqcup P^{\mathrm{op}}\sqcup H
   \rightrightarrows M\sqcup N,
\]
whose multiplication is determined by the two actions and by the
division maps associated with the biprincipal structure, so we denote by $P^{\mathrm{op}}$ the bibundle $P$ but regarded as the inverse of $P$. Define
$\Omega\in\Omega^2(\mathcal L(P))$ by
\[
\left.\Omega\right|_G=\omega_G,\qquad
\left.\Omega\right|_P=\omega_P,\qquad
\left.\Omega\right|_{P^{\mathrm{op}}}=-\omega_P,\qquad
\left.\Omega\right|_H=\omega_H.
\]
The fact that the graphs of the actions on $P$ are isotropic implies that
$\Omega$ is multiplicative on $P$; whereas multiplicativity on $P^{\mathrm{op}}$ follows by inversion. Similarly, principality gives us that $\Omega$ is multiplicative with respect to the division maps $P\times_N P^{\mathrm{op}} \to G  $ and $ P^{\mathrm{op}}\times_M P \to H  $.

A multiplicative $2$-form is globally nondegenerate whenever it is
nondegenerate along the tangent bundle restricted to the unit submanifold. Consequently, $\Omega$ is
nondegenerate, and so is its restriction $\omega_P$ to $P$.
\end{proof}
\begin{theorem}\label{pro:symplectic 11}
    The 2-forms introduced above are symplectic and make \eqref{eq:11 from gpd fact} into a symplectic $(1,1)$-morphism.
\end{theorem}
\begin{proof}
    It is immediate to check that \eqref{eq:2form} is closed. A similar calculation to the one in \cite[Prop. 4.8, (2)]{Alvarez2023} shows that the graphs of the corresponding (double) groupoid actions are isotropic.

By \cite[Theorem~4.10]{Alvarez2023}, the forms on the four corner double
groupoids are nondegenerate. Since the four
boundary bibundles are biprincipal, Lemma
\ref{lem:nondegenerate-bimodule} implies that their $2$-forms
are nondegenerate. The central space is likewise a biprincipal bibundle between the resulting symplectic boundary groupoids, so a second application of Lemma
\ref{lem:nondegenerate-bimodule} proves that its $2$-form is nondegenerate as well.

Thus, all the forms in \eqref{eq:2form} are symplectic and all the
actions are symplectic. Hence, \eqref{eq:11 from gpd fact} becomes a
symplectic $(1,1)$-morphism.
\end{proof}
\begin{example} A particular case of this construction arises when $E$ is an exact Courant algebroid and hence $K$ may be chosen as the pair groupoid over $M$ equipped with the forms $\omega_2=0$ and $\omega_1=t^*H-s^*H$, for $H\in \Omega^3(M)$ closed such that $[H]$ is the \v{S}evera class of $E$. The morphisms $\Phi_A,\Phi_B,\Phi_C$ and $\Phi_D$ are then the respective anchor morphisms consisting of the respective source and target combined. This is the relevant case for generalized K\"ahler geometry, and Theorem \ref{pro:symplectic 11} recovers (part of) \cite[Prop. 3.7]{AGJ2024} when the relevant double Lie groupoids are of the form \eqref{eq:slim} for $K$ the pair groupoid over $M$.   
\end{example}

As shown in \cite{Alvarez2023,MehtaTang2011}, $\Wbar$ takes symplectic double groupoids to 2-shifted symplectic (local) Lie 2-groupoids. In our case, $\Wbar$ produces 2-shifted symplectic 2-groupoids, integrating the background Courant algebroids. Since \eqref{eq:11 from gpd fact} consists of Morita equivalences, these Lie 2-groupoids are related by hypercovers, and so we have that integrating different choices of Manin triples in the same transitive Courant algebroid by the method proposed in \cite[Thm. 4.10]{Alvarez2023} produces 2-shifted symplectic Lie 2-groupoids, which are equivalent by means of the anafunctors on the boundary of the following diagram: 
\[\begin{tikzcd}
	\left(\overline{W}\X_{A,B}^{B,A}, \overline{W}^*\omega_{A,B}^{B,A}\right) & \left(\overline{W}\widetilde{\X}_{A,B}^{B,C}, \overline{W}^*\widetilde{\omega}_{A,B}^{B,C}\right) & \left(\overline{W}\X_{C,B}^{B,C}, \overline{W}^*\omega_{C,B}^{B,C}\right) \\
	\left(\overline{W}\widetilde{\X}_{A,D}^{B,A}, \overline{W}^*\widetilde{\omega}_{A,D}^{B,A}\right) & \left(\overline{W}\widetilde{\widetilde{\X}}_{A,D}^{B,C}, \overline{W}^*\widetilde{\widetilde{\omega}}_{A,D}^{B,C}\right) & \left(\overline{W}\widetilde{\X}_{C,D}^{B,C}, \overline{W}^*\widetilde{\omega}_{C,D}^{B,C}\right) \\
	\left(\overline{W}\X_{A,D}^{D,A}, \overline{W}^*\omega_{A,D}^{D,A}\right) & \left(\overline{W}\widetilde{\X}_{A,D}^{D,C}, \overline{W}^*\widetilde{\omega}_{A,D}^{D,C}\right) & \left(\overline{W}\X_{C,D}^{D,C}, \overline{W}^*\omega_{C,D}^{D,C}\right),
	\arrow["\sim", two heads, from=1-2, to=1-1]
	\arrow["\sim", two heads, from=1-2, to=1-3]
	\arrow["\sim", two heads, from=2-1, to=1-1]
	\arrow["\sim", two heads, from=2-1, to=3-1]
	\arrow["\sim", two heads, from=2-2, to=1-2]
	\arrow["\sim", two heads, from=2-2, to=2-1]
	\arrow["\sim", two heads, from=2-2, to=2-3]
	\arrow["\sim", two heads, from=2-2, to=3-2]
	\arrow["\sim", two heads, from=2-3, to=1-3]
	\arrow["\sim", two heads, from=2-3, to=3-3]
	\arrow["\sim", two heads, from=3-2, to=3-1]
	\arrow["\sim", two heads, from=3-2, to=3-3]
\end{tikzcd}\]
where the 2-forms therein are simply obtained by pulling back the corresponding symplectic forms given by \eqref{eq:2form}. As a consequence of \cite[Lemma 4.20]{Alvarez2023}, each of the anafunctors on the boundary of this diagram is a symplectic Morita equivalence as in \cite{CuecaZhu2023}. Since the corners of this diagram integrate the background Courant algebroid $E$, this diagram illustrates the way in which such integrations depend on the choice of the underlying Manin triple.
\begin{corollary} The 2-shifted symplectic Lie 2-groupoids $\overline{W}\X_{A,B}^{B,A}$, $\overline{W}\X_{C,B}^{B,C}$, $\overline{W}\X_{A,D}^{D,A}$ and $\overline{W}\X_{C,D}^{D,C}$ are symplectically Morita equivalent. \qed   
\end{corollary}
Further aspects of this diagram in the context of shifted symplectic geometry shall be studied in future work.
\begin{remark}[Double groupoid extensions and cocycles] As shown in \cite{AndruskiewitschNatale2009}, every discrete double groupoid is an extension of a slim double groupoid by a bundle of abelian groups and these extensions are described by certain double cocycles. For example, let $T$ be an abelian Lie group and suppose that $\gamma_v:(\X_{A,B}^{B,A})_{1,2}\to T$ and $\gamma_h:(\X_{A,B}^{B,A})_{2,1}\to T$ are normalized 2-cocycles satisfying
\[ \delta^v\gamma_h=\delta^h\gamma_v \]
where $\delta^v,\delta^h$ are the corresponding vertical and horizontal simplicial differentials on the nerve of $\X_{A,B}^{B,A}$. Then we may extend $\X_{A,B}^{B,A}$ by $T$ as follows: we have that $\widehat{\X_{A,B}^{B,A}}=\X_{A,B}^{B,A}\times T$ is a double Lie groupoid equipped with the two multiplications respectively deformed by $\gamma_h,\gamma_v$. If $\X_{A,B}^{B,A}$ sits inside a $(1,1)$-morphism as in \eqref{eq:11 from gpd fact}, we may then convert it into a 2-morphism of anafunctors of Lie 2-groupoids using Theorem \ref{thm:1-1-morph-to-2-hom-diamond} and use the Morita invariance of cohomology to transfer the class of $[\gamma_h+\gamma_v]$ across the boundary of the 2-morphism, determining now 2-groupoid extensions. The relevance for prequantization of this fact and the promotion of \eqref{eq:11 from gpd fact} to double groupoid extensions shall appear in future work.\end{remark}

\begingroup
\sloppy
\printbibliography
\endgroup

\end{document}